\documentclass[a4paper,10pt]{article}
\usepackage[margin=0.80in,bottom=0.90in,top=0.85in]{geometry}

\usepackage[english]{babel}
\usepackage{chngcntr}
\usepackage[utf8]{inputenc}
\usepackage{amsmath,mathrsfs}
\usepackage{indentfirst}
\usepackage{fancyhdr}
\usepackage{amssymb}
\usepackage{amsthm}
\usepackage[dvips]{graphicx}
\usepackage{color}
\usepackage{xcolor}
\usepackage{cancel}
\usepackage{subcaption}
\usepackage[dvipsnames]{xcolor}
\usepackage{eucal}
\usepackage{latexsym}
\usepackage{chngcntr}
\usepackage{faktor}
\usepackage{microtype}
\usepackage{newlfont}
\usepackage{enumitem}
\usepackage{geometry}
\usepackage{nicematrix}
\usepackage{tikz}
\usepackage{todonotes}

\usepackage{amsmath, amsfonts, amsthm,amssymb,  bbm}

\usepackage{tikz-cd}
\usepackage{multicol}
\usepackage{cancel}
\usepackage{mathtools}
\usepackage{graphicx}

\usepackage[colorlinks=true,    
            linkcolor=blue,    
            citecolor=blue,     
            urlcolor=blue]{hyperref}

\newcommand{\e}{\epsilon}
\newcommand{\cO}{\mathcal{O}}

\newcommand{\tth}{\mathtt{h}}
\newcommand{\ttf}{\mathtt{f}_\e}
\newcommand{\pa}{\partial}

\newcommand{\de}{\textrm{d}}

\newcommand{\cA}{{\mathcal A}}
\newcommand{\cJ}{{\mathcal J}}

\newcommand{\N}{{\mathbb{N}}}
\newcommand{\R}{{\mathbb{R}}}
\newcommand{\Z}{{\mathbb{Z}}}
\newcommand{\T}{{\mathbb{T}}}

\newcommand{\vet}[2]{\begin{bmatrix}#1 \\ #2 \end{bmatrix}}
\newcommand{\cB}{\mathcal{B}}

\newcommand{\fR}{\mathfrak{R}}

\newcommand{\im}{\textup{i}}

\newcommand{\cL}{\mathcal{L}}

\newcommand{\vertiii}[1]{{\left\vert\kern-0.25ex\left\vert\kern-0.25ex\left\vert #1 
    \right\vert\kern-0.25ex\right\vert\kern-0.25ex\right\vert}}
\newcommand{\opnorm}[1]{{\vert\kern-0.25ex\vert\kern-0.25ex\vert #1 
    \vert\kern-0.25ex\vert\kern-0.25ex\vert}}

\newcommand{\sgn}{\operatorname{sgn}}
\newcommand{\Res}{\operatorname*{Res}}

\theoremstyle{plain}
\newtheorem{lemma}{Lemma}
\newtheorem{theorem}[lemma]{Theorem}
\newtheorem{proposition}[lemma]{Proposition}
\newtheorem{corollary}[lemma]{Corollary}
\newtheorem{remark}[lemma]{Remark}
\newtheorem{definition}[lemma]{Definition}

\theoremstyle{definition} 

\numberwithin{equation}{section}
\counterwithin{lemma}{section}

\title{Modulational spectrum of infinite-depth hydroelastic Stokes waves}

\begin{document}

\author{
Ting-Yang Hsiao%
\thanks{SISSA, Via Bonomea 265, 34136 Trieste, Italy.
\newline
\textit{Emails:}
\texttt{thsiao@sissa.it},
\texttt{yezhang@sissa.it}}
,\quad
Zirui Li%
\thanks{University of Illinois Urbana--Champaign,
1409 W. Green Street, Urbana, IL 61801, USA.
\newline
\textit{Emails:}
\texttt{ziruili2@illinois.edu},
\texttt{cz43@illinois.edu}}
,\quad
Ye Zhang\footnotemark[1]
,\quad
Chengbin Zhu\footnotemark[2]
}

\maketitle

\begin{abstract}
We determine the complete local Bloch spectrum bifurcating from the origin for small-amplitude periodic hydroelastic Stokes waves in infinite depth, under the combined effects of gravity, surface tension, and elastic bending.
Away from the Wilton-type resonance set, we construct a real-analytic Stokes-wave branch and analyze the four eigenvalues emerging from the defective zero eigenvalue of the linearized hydroelastic Euler system. Using analytic spectral perturbation theory and Hamiltonian-reversible
reductions, we decouple them into a Benjamin--Feir pair and a long-wave pair. The long-wave pair remains purely imaginary and has the singular scale $\cO(\sqrt{|\mu|})$, whereas the Benjamin--Feir pair is governed by an
explicit discriminant whose leading sign yields a sharp criterion for modulational stability and instability. We derive the exact non-resonant
phase diagram in the surface-tension-bending parameter plane and identify a bounded stability island generated by elastic bending. In the unstable region, and away from a drift degeneracy, the Benjamin--Feir branches form
a local figure-eight curve. In the zero-bending limit, the reduced coefficients recover the known deep-water gravity and gravity-capillary results, while the change from the finite-depth $\cO(|\mu|)$ long-wave scale to $\cO(\sqrt{|\mu|})$ shows that the infinite-depth problem is singular.
\end{abstract}

\tableofcontents

\section{Introduction}\label{sec1}
Hydroelastic waves arise from the coupling between an inviscid fluid and a
deformable interface, and provide a basic model for waves beneath an ice
sheet or a thin flexible cover. Flexural--gravity waves have been observed
in field measurements of sea ice \cite{DJMF}, while laboratory experiments
have demonstrated nonlinear hydroelastic waves, resonant three-wave
interactions, and hydroelastic wakes beneath thin elastic sheets
\cite{DBF,DBeF,OTLLRDS}.  In this paper, we consider two-dimensional,
irrotational, infinite-depth water waves under gravity, surface tension
\(\kappa\geq0\), and elastic bending \(b>0\).  In Zakharov variables
\((\eta,\psi)\), the elastic contribution is generated by the curvature
energy
\[
\kappa\int_{\mathbb T}\bigl(\sqrt{1+\eta_x^2}-1\bigr)\,\mathrm dx
+
\frac b2\int_{\mathbb T}\sigma(\eta)^2\sqrt{1+\eta_x^2}\,\mathrm dx,
\]
and therefore introduces a quasilinear fourth-order restoring force into
the dynamic boundary condition.

The Benjamin--Feir instability is the long-wave, or modulational,
instability of a periodic wavetrain.  It was discovered for deep-water
gravity waves by Benjamin and Feir \cite{Benjamin,BF}, and has since played
a central role in the spectral theory of water waves.  Rigorous proofs of
the deep-water instability were obtained by Nguyen and Strauss
\cite{NS}, while Berti, Maspero, and Ventura gave a complete description
of the four eigenvalues bifurcating from the origin and proved the
figure-eight geometry of the unstable spectrum \cite{BMV1}; see also
\cite{BMV3,BMV_ed} for the finite-depth theory.  The corresponding
gravity--capillary problem has recently been resolved at the level of the
full Euler equations in \cite{HW2026}.

The hydroelastic problem has a substantial existence theory.  Periodic
and solitary travelling waves have been constructed by a variety of
bifurcation and variational methods; see, among others,
\cite{To2008,BTo,AAS1,AAS2,GHW,GPa}.  Spectral stability of periodic
flexural--gravity waves has also been studied numerically by Floquet
methods \cite{TMPVB,BPW2024}.  More recently, nonlinear modulational
instability in deep water has been obtained through a focusing NLS
approximation in \cite{WY}.  These works, however, do not provide a
complete Euler-level description of all the small Bloch eigenvalues for
the gravity--capillary--bending problem, nor an exact two-parameter
stability diagram. We resolve this problem at the level of the full hydroelastic Euler equations.

\paragraph{The non-resonant Stokes branch.}
After normalizing gravity and the basic wavenumber to one, the linear
dispersion relation at the flat surface is
\[
\omega(\xi)^2
=
|\xi|\bigl(1+\kappa\xi^2+b\xi^4\bigr).
\]
For the \(n\)-th periodic Fourier mode,
\[
\omega_n^2=n(1+\kappa n^2+b n^4),
\qquad
c_n^2=\frac{1+\kappa n^2+b n^4}{n}.
\]
The fundamental mode is resonant with the \(n\)-th harmonic precisely
when \(c_n=c_1\), namely on
\begin{equation}\label{def:infinite-Rn}
\mathfrak R_n
:=
\left\{
(\kappa,b)\in\mathbb R_{\geq0}\times\mathbb R_{>0}:
\kappa+(n^2+n+1)b=\frac1n
\right\},
\qquad n\geq2.
\end{equation}
We set
\begin{equation}\label{def:infinite-R}
\mathfrak R:=\bigcup_{n\geq2}\mathfrak R_n,
\qquad
\mathcal P_\infty
:=
\left(\mathbb R_{\geq0}\times\mathbb R_{>0}\right)\setminus\mathfrak R.
\end{equation}
For \((\kappa,b)\in\mathcal P_\infty\), the flat travelling-wave problem
has a simple kernel in the even--odd symmetry class, and the
Crandall--Rabinowitz theorem yields a real-analytic Stokes branch
\[
\epsilon\longmapsto
(\eta_\epsilon,\psi_\epsilon,c_\epsilon),
\qquad
c_\epsilon=c_{\kappa,b}+\mathcal O(\epsilon^2),
\qquad
c_{\kappa,b}:=\sqrt{1+\kappa+b}.
\]

\paragraph{Main result.}
Let \(\mathcal L_{\mu,\epsilon}\) denote the Bloch operator obtained by
linearizing the hydroelastic Euler system at the Stokes wave, where
\(\mu\) is the Floquet exponent.  At \((\mu,\epsilon)=(0,0)\), the origin
is an isolated defective eigenvalue of algebraic multiplicity four and
geometric multiplicity three.  We prove that, for every
\((\kappa,b)\in\mathcal P_\infty\), the four eigenvalues bifurcating from
the origin can be separated into a Benjamin--Feir pair
\(\lambda_1^\pm\) and a long-wave pair \(\lambda_0^\pm\).

The three coefficients governing the reduced Benjamin--Feir block are
\begin{equation}\label{def:e-coefficients}
\begin{aligned}
e_{11}(\kappa,b)
&:=
-\frac{
88b^2+6b\kappa-54b+2\kappa^2+\kappa+8
}{
8c_{\kappa,b}(14b+2\kappa-1)
},\\
e_{12}(\kappa,b)
&:=
\frac{1+3\kappa+5b}{c_{\kappa,b}},\\
e_{22}(\kappa,b)
&:=
-\frac{
15b^2+3\kappa^2+22b\kappa+30b+6\kappa-1
}{
c_{\kappa,b}^3
},
\end{aligned}
\end{equation}
and we write
\begin{equation}\label{def:breve-c}
\breve c_{\kappa,b}:=2c_{\kappa,b}-e_{12}(\kappa,b).
\end{equation}
Theorem~\ref{Complete BF thm} gives
\begin{equation}\label{lambda1-intro-infinite}
\lambda_1^\pm(\mu,\epsilon)
=
\mathrm{i}\,\frac12\breve c_{\kappa,b}\mu
+
\mathrm{i}\mathcal O(\mu\epsilon^2,\mu^2\epsilon,\mu^3)
\pm
\frac{\mu}{8}
\sqrt{\Delta_{\mathrm{BF}}(\kappa,b;\mu,\epsilon)},
\end{equation}
where
\begin{equation}\label{DeltaBF-intro-infinite}
\Delta_{\mathrm{BF}}(\kappa,b;\mu,\epsilon)
=
8e_{22}(\kappa,b)e_{11}(\kappa,b)\epsilon^2
-
e_{22}(\kappa,b)^2\mu^2
+
\mathcal O(\epsilon^3,\mu\epsilon^2,\mu^2\epsilon,\mu^3).
\end{equation}
The remaining pair is purely imaginary and satisfies
\begin{equation}\label{lambda0-intro-infinite}
\lambda_0^\pm(\mu,\epsilon)
=
\mathrm{i}c_{\kappa,b}\mu
+
\mathrm{i}\mathcal O(\mu\epsilon^2,\mu^2\epsilon,\mu^3)
\pm
\mathrm{i}\sqrt{\mu}\,
\bigl(1+\mathcal O(\epsilon,\mu)\bigr),
\qquad \mu>0.
\end{equation}
The branches for negative Floquet exponents are recovered from
\[
\sigma(\mathcal L_{-\mu,\epsilon})
=
\overline{\sigma(\mathcal L_{\mu,\epsilon})}.
\]

In the sideband scaling \(\mu=\epsilon\nu\), with bounded
\(0<\nu=\mathcal O(1)\), the leading Benjamin--Feir index is
\begin{equation}\label{Ind-infinite}
\operatorname{Ind}_\infty(\kappa,b)
:=
8e_{11}(\kappa,b)e_{22}(\kappa,b).
\end{equation}
If \(\operatorname{Ind}_\infty(\kappa,b)>0\), then, for all sufficiently
small \(\epsilon>0\), the pair \(\lambda_1^\pm\) has opposite nonzero
real parts for
\begin{equation}\label{nu-star-infinite}
0<|\nu|<\nu_*(\kappa,b;\epsilon),
\qquad
\nu_*(\kappa,b;\epsilon)
=
\frac{\sqrt{8e_{22}(\kappa,b)e_{11}(\kappa,b)}}
{|e_{22}(\kappa,b)|}
\bigl(1+\mathcal O(\epsilon)\bigr).
\end{equation}
The corresponding spectral branches form a local figure eight.  If
\(\operatorname{Ind}_\infty(\kappa,b)<0\), all four small eigenvalues
remain purely imaginary in the sideband regime.  On
\(\operatorname{Ind}_\infty=0\), the leading criterion degenerates and
higher-order terms are required.

\begin{figure}[t]
    \centering
    \includegraphics[width=0.68\textwidth]
    {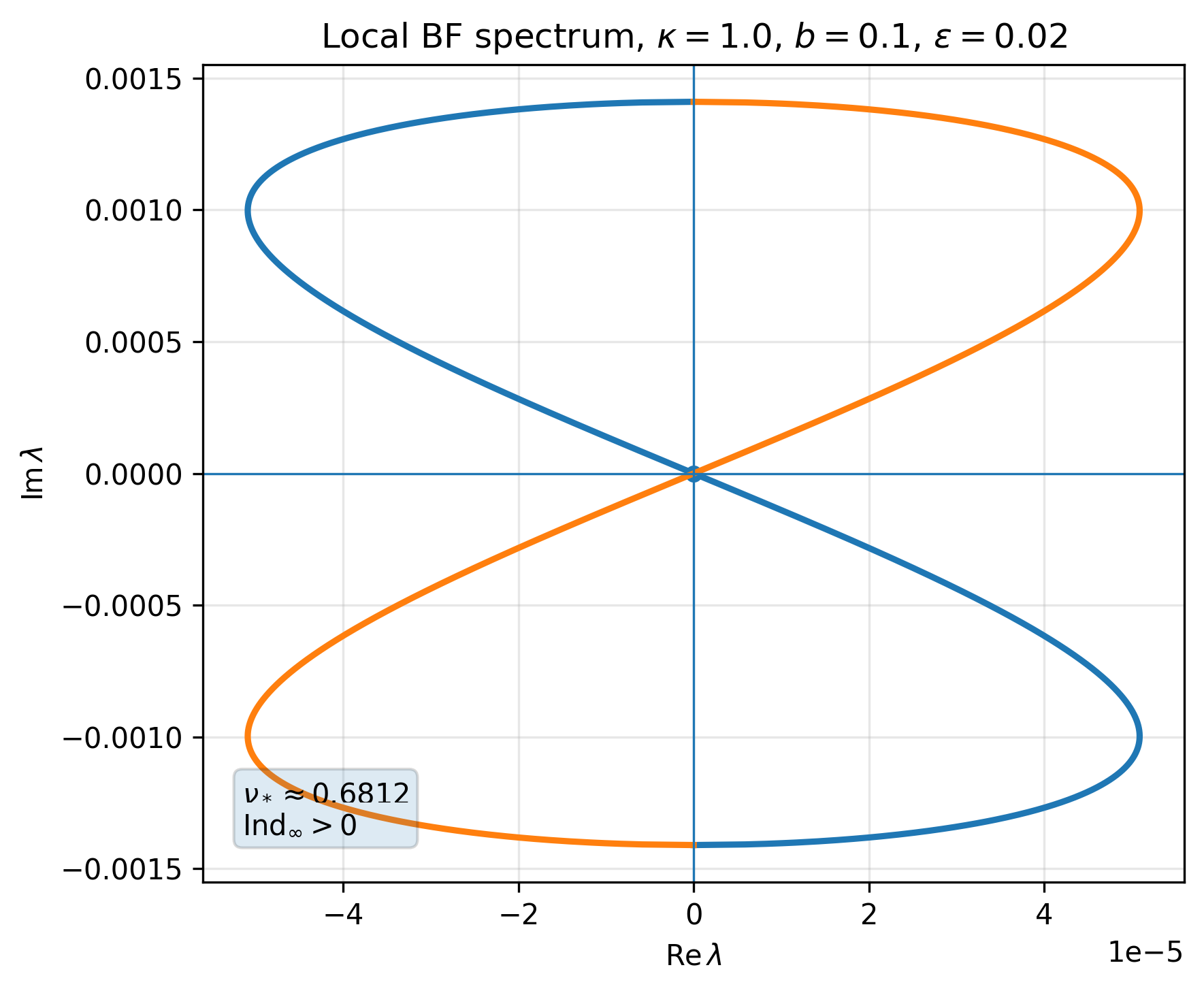}
    \caption{
    Leading-order geometry of the unstable Benjamin--Feir pair for
    \(\kappa=1\), \(b=0.1\), and \(\epsilon=0.02\).  Positive and negative
    Floquet exponents generate the upper and lower halves of the local
    figure-eight curve, respectively.
    }
    \label{fig:infinite-bf-figure-eight}
\end{figure}

\paragraph{The stability diagram.}
Introduce
\[
\begin{aligned}
P_{11}(\kappa,b)
&:=88b^2+(6\kappa-54)b+2\kappa^2+\kappa+8,\\
P_{22}(\kappa,b)
&:=15b^2+(22\kappa+30)b+3\kappa^2+6\kappa-1,\\
R_2(\kappa,b)
&:=14b+2\kappa-1.
\end{aligned}
\]
Then, away from the transition curves \(P_{11}P_{22}=0\),
\begin{equation}\label{eq:sign-index-factorization-intro}
\operatorname{sgn}\operatorname{Ind}_\infty(\kappa,b)
=
\operatorname{sgn}
\bigl(P_{11}(\kappa,b)P_{22}(\kappa,b)R_2(\kappa,b)\bigr).
\end{equation}
Proposition~\ref{prop:exact-phase-diagram} gives the exact decomposition
of the non-resonant parameter plane.  In particular, the stable set has
two qualitatively different components: a low-bending component adjacent
to the second-harmonic resonance and a bounded stability island created
by the sign change of \(P_{11}\).  The latter has no counterpart in the
pure gravity--capillary problem and is a genuinely hydroelastic feature.

\begin{figure}[t]
    \centering
    \includegraphics[width=0.78\textwidth]
    {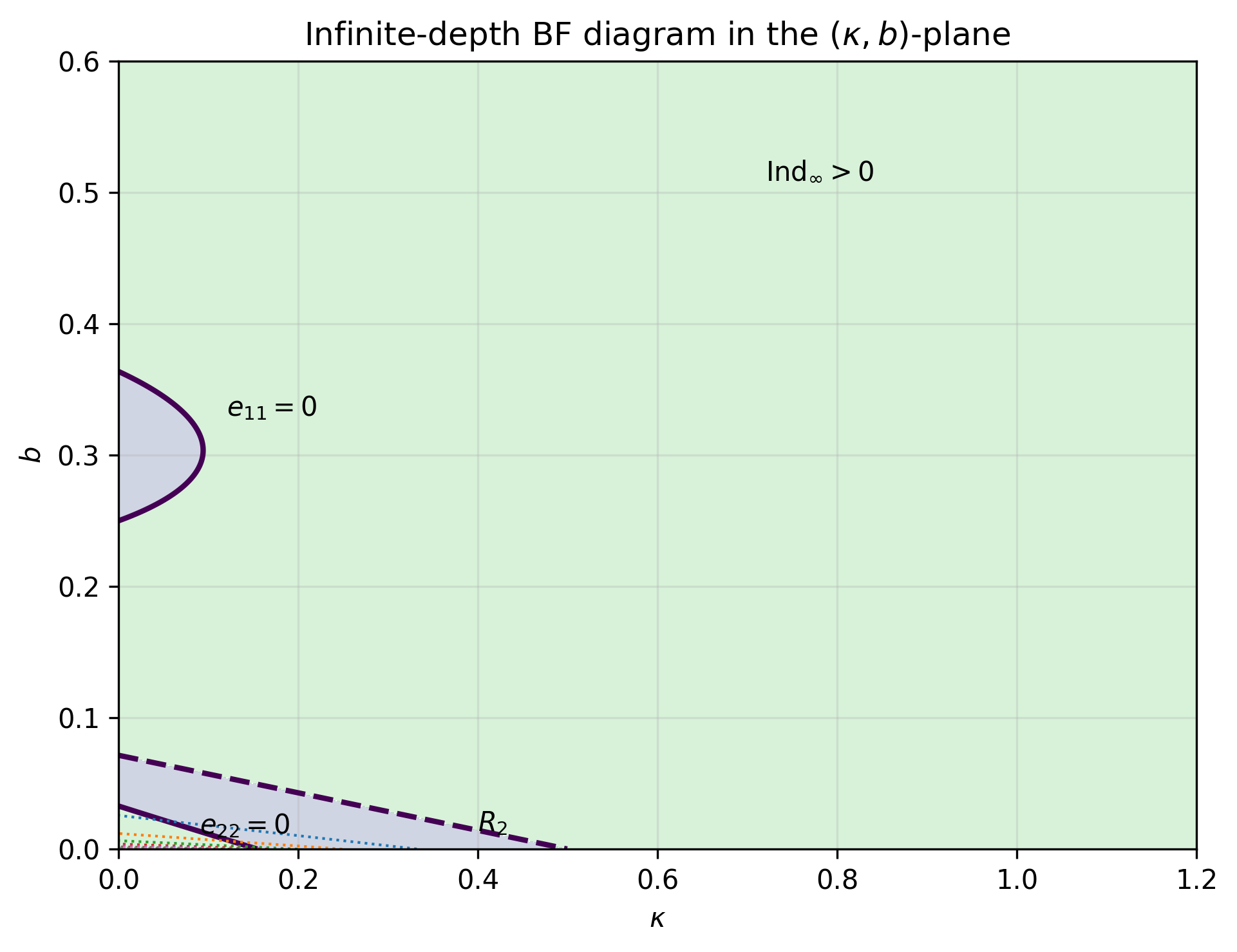}
    \caption{
    Leading-order Benjamin--Feir stability diagram in the
    \((\kappa,b)\)-plane.  The shaded region is characterized by
    \(\operatorname{Ind}_\infty>0\), while the unshaded region has
    \(\operatorname{Ind}_\infty<0\).  The regular transition curves are
    \(P_{11}=0\) and \(P_{22}=0\).  The dashed line is the second-harmonic
    resonance \(R_2=0\), and the dotted curves are the higher resonances
    \(\mathfrak R_n\), \(n\geq3\), all of which are excluded from the
    non-resonant Stokes-wave construction.
    }
    \label{fig:infinite-kappa-b-diagram}
\end{figure}

Each resonance curve \(\mathfrak R_n\) is a line segment in the physical
quadrant, and the family accumulates only at the origin.  Moreover,
\[
\mathfrak R
\subset
\left\{
0\leq\kappa<\frac12,\quad
0<b\leq\frac1{14}
\right\}.
\]
Thus \(\kappa\geq1/2\) or \(b>1/14\) excludes all Wilton-type
resonances.  The resonance curves are exceptional loci for the
bifurcation problem and must not be confused with the regular
Benjamin--Feir transition curves \(P_{11}=0\) and \(P_{22}=0\).

\paragraph{Why the infinite-depth hydroelastic problem is singular.}
Several difficulties prevent the result from following by a routine
modification of the gravity-wave analysis.

First, the physical Bloch multiplier \(|D+\mu|\) is not two-sided
analytic at \(\mu=0\), because its zero Fourier mode is \(|\mu|\).
Analytic perturbation theory can therefore be applied only after
constructing a one-sided analytic continuation on \(\mu\geq0\); the
negative physical spectrum must then be recovered from the reality
symmetry.

Second, the four eigenvalues emerge from a defective zero eigenvalue.
At nonzero amplitude, translation and gauge symmetries provide two
kernel directions, while differentiation along the Stokes branch and
with respect to the Bernoulli parameter produces the associated
generalized vectors.  This structure must be identified before the
Kato basis can be normalized in a way compatible with both the
symplectic form and reversibility.

Third, the two spectral pairs have different scales:
the Benjamin--Feir pair is of order \(\mu\), whereas the long-wave pair
is of order \(\sqrt{\mu}\).  Although this separates the two pairs
spectrally, the Sylvester operator arising in the block-decoupling
degenerates as \(\mu\to0\).  A naive conjugation therefore produces
apparent negative powers of \(\mu\).  The proof uses the precise
Hamiltonian--reversible matrix structure, the exact kernel identities,
and several cancellations in the off-diagonal block to show that the
conjugating transformations remain analytic and bounded.

Finally, the bending energy generates fourth-order
variable-coefficient operators after linearization and flattening.
The explicit index requires the second-order Stokes expansion, the
second-order expansion of the flattened bending operator, and the
corresponding derivatives of the Kato basis.  Keeping these
calculations compatible with self-adjointness, reversibility, and the
entrywise remainder classes is the main perturbative computation of
the paper.

\paragraph{Consistency with limiting problems.}
Although the theorem is formulated for \(b>0\), the reduced coefficients
have a well-defined algebraic limit as \(b\downarrow0\), away from the
second-harmonic pole.  Namely,
\[
\begin{aligned}
e_{11}(\kappa,0)
&=
-\frac{2\kappa^2+\kappa+8}
{8\sqrt{1+\kappa}\,(2\kappa-1)},\\
e_{12}(\kappa,0)
&=
\frac{1+3\kappa}{\sqrt{1+\kappa}},\\
e_{22}(\kappa,0)
&=
-\frac{3\kappa^2+6\kappa-1}{(1+\kappa)^{3/2}},
\end{aligned}
\]
and hence
\begin{equation}\label{eq:b-zero-index-limit}
\operatorname{Ind}_\infty(\kappa,0)
=
\frac{
(2\kappa^2+\kappa+8)(3\kappa^2+6\kappa-1)
}{
(1+\kappa)^2(2\kappa-1)
}.
\end{equation}
Thus the only sign-changing thresholds on the zero-bending boundary are
\[
\kappa_{\mathrm{gc}}
=
\frac{2}{\sqrt3}-1
\qquad\text{and}\qquad
\kappa=\frac12,
\]
in agreement with the deep-water gravity--capillary diagram of
\cite{HW2026}; the Wilton values \(\kappa=1/n\), \(n\geq2\), remain
excluded from the non-resonant statement.

At the pure-gravity point \((\kappa,b)=(0,0)\),
\[
e_{11}=e_{12}=e_{22}=1,
\qquad
\breve c_{0,0}=1,
\qquad
\operatorname{Ind}_\infty(0,0)=8.
\]
Consequently,
\[
\lambda_1^\pm(\mu,\epsilon)
=
\frac{\mathrm i}{2}\mu
+
\mathrm i\mathcal O(\mu\epsilon^2,\mu^2\epsilon,\mu^3)
\pm
\frac{\mu}{8}
\sqrt{
8\epsilon^2-\mu^2
+
\mathcal O(\epsilon^3,\mu\epsilon^2,\mu^2\epsilon,\mu^3)
},
\]
which recovers the deep-water Benjamin--Feir expansion and the
instability threshold
\(\mu=2\sqrt2\,\epsilon(1+\mathcal O(\epsilon))\) of
\cite{BMV1}.

The pure-bending slice \(\kappa=0\) also illustrates a new feature of
the two-parameter problem.  Setting
\[
b_{22}(0)=\frac{4\sqrt{15}-15}{15},
\]
one obtains
\[
\operatorname{Ind}_\infty(0,b)<0
\quad\Longleftrightarrow\quad
b_{22}(0)<b<\frac1{14}
\quad\text{or}\quad
\frac14<b<\frac4{11}.
\]
The second interval is the bending-induced stability island visible in
Figure~\ref{fig:infinite-kappa-b-diagram}.

The zero-bending comparison is a consistency statement for the reduced
coefficients, rather than a uniform operator limit of the proof: at
\(b=0\), the highest differential order drops from four to two.  Likewise,
the present theorem is not obtained by simply sending the depth to
infinity in the finite-depth hydroelastic result \cite{HLZZ2026}.  For
finite depth \(\mathtt h\), the zero-mode Dirichlet--Neumann symbol is
\[
|\mu|\tanh(\mathtt h|\mu|)
\sim
\mathtt h\mu^2
\qquad (\mu\to0),
\]
whereas in infinite depth it is \(|\mu|\).  Hence the limits
\(\mathtt h\to\infty\) and \(\mu\to0\) do not commute, and the long-wave
pair changes from order \(\mu\) in finite depth to order \(\sqrt{\mu}\)
in infinite depth.  This singular change of scale is one of the main
reasons that the infinite-depth problem requires a separate analysis.

\paragraph{Relation to previous work.}
We briefly place the present result within the stability theory of
periodic water waves, distinguishing the pure-gravity,
gravity--capillary, and hydroelastic settings.

\smallskip
\noindent
\emph{Pure-gravity waves.}
The local theory of periodic gravity waves originates in the classical
works of Stokes, Nekrasov, Levi-Civita, and Struik
\cite{stokes,Nek,LC,Struik}.  General modulation criteria for nonlinear
dispersive wavetrains were developed by Lighthill and Whitham
\cite{Li,Whitham}, while Benjamin and Feir
\cite{Benjamin,BF} identified the modulational instability of
deep-water Stokes waves.  Zakharov's Hamiltonian formulation and the
associated envelope equations \cite{Zak68,Zak1,ZK} provided the formal
framework for the subsequent analysis.

The first rigorous proof for the full finite-depth gravity water-wave
problem was obtained by Bridges and Mielke \cite{BrM}.  In infinite
depth, Nguyen and Strauss \cite{NS} proved spectral modulational
instability by Lyapunov--Schmidt reduction, and Yang \cite{Yang21}
obtained an alternative proof using the periodic Evans function; see
also Hur and Yang \cite{HY} for the finite-depth Evans-function theory
and Creedon and Deconinck \cite{CD} for higher-order spectral
asymptotics.

A Hamiltonian--reversible approach was developed by Berti, Maspero, and
Ventura.  They resolved the complete quartet bifurcating from the
origin and established the local figure-eight geometry in deep water
\cite{BMV1}, based on the analytic framework of \cite{BMV2}.  The
finite-depth and critical-depth problems were subsequently treated in
\cite{BMV3,BMV_ed}.  Spectral collisions away from the origin,
including the first and infinitely many high-frequency instability
isolas, were studied in \cite{BMV4,BCMV}.  Related numerical, nonlinear, rotational,
and transverse instability results may be found in
\cite{CDT, Mc1,Mc2,DO,ChenSu,CNS,CNS2,H2024, JRSY,berti2026mclean}.

\smallskip
\noindent
\emph{Gravity--capillary waves.}
The classical modulation calculations of Djordjevi\'c and Redekopp
\cite{DjRe} and Ablowitz and Segur \cite{AS} predicted
surface-tension-dependent stability regions and long--short wave
degeneracies.  Wilton ripples were introduced in \cite{Wilton}, and
analytic resonant branches were constructed in \cite{reeder}.  Hur and
Yang \cite{HY2} developed a rigorous Evans-function theory for both
non-resonant waves and Wilton ripples, while Sun and Wahlén \cite{SW}
derived a spectral stability index and analyzed nonzero crossings.

Further numerical and transverse results appear in \cite{DT,HTW}.
More recently, Hsiao and Maspero \cite{HM2025} determined the complete
small Bloch quartet in finite depth, and Hsiao and Wang \cite{HW2026}
obtained the corresponding deep-water Euler-level description.  The
zero-bending limit of the present coefficients agrees with
\cite{HW2026}, whereas comparison with \cite{HM2025} reveals the
singular change of the long-wave spectral scale in the infinite-depth
limit.

\smallskip
\noindent
\emph{Hydroelastic waves.}
The modelling of flexural--gravity waves and sea-ice interaction is
surveyed in \cite{SDWRL,Squire,KPVB}.  Periodic and solitary
hydroelastic waves have been constructed by bifurcation, variational,
spatial-dynamics, and numerical methods; see
\cite{To2008,BTo,AAS1,AAS2,MVBW,GPa,GHW,AG2024}.

Modulational and spectral stability were studied formally and
numerically in \cite{TMPVB,BPW2024}, and Wan and Yang \cite{WY} proved
nonlinear modulational instability through a focusing NLS
approximation.  Hsiao, Li, Zhang, and Zhu \cite{HLZZ2026} resolved the
complete four-eigenvalue Benjamin--Feir spectrum in finite depth and
derived a three-parameter stability diagram.

The present infinite-depth problem is not a regular limit of that
result: the zero-mode Dirichlet--Neumann symbol changes the long-wave
scale from \(O(\mu)\) to \(O(\sqrt{\mu})\).  We resolve the complete
local Bloch quartet at the Euler level, derive an explicit
Benjamin--Feir discriminant and an exact non-resonant
\((\kappa,b)\)-diagram, identify a bending-induced stability island,
and, away from the drift-degeneracy
\(\breve c_{\kappa,b}=0\), establish the local figure-eight geometry
of the unstable spectrum.

\paragraph{Proof strategy.}
The proof combines the good unknown and an infinite-depth conformal
flattening with Kato's similarity transformation theory.  On the
resulting four-dimensional spectral subspace, we construct a symplectic
and reversible basis and derive a $4\times4$ Hamiltonian normal form.
A pair of structure-preserving conjugations then separates the
Benjamin--Feir block from the long-wave block, from which the complete
small spectrum follows.

\section{The full Benjamin–Feir spectrum of hydroelastic waves}
\paragraph{The infinite-depth hydroelastic problem.}
We first recall the infinite-depth formulation of the two-dimensional
hydroelastic water-wave problem in Zakharov variables. Let
\(\mathbb T:=\mathbb R/2\pi\mathbb Z\). At time \(t\), the fluid occupies the
domain
\[
    \Omega_{\eta}
    :=
    \bigl\{(x,y)\in \mathbb T\times\mathbb R:\ y<\eta(t,x)\bigr\},
\]
whose upper boundary is the graph
\[
    S_t:=\bigl\{(x,y)\in\mathbb T\times\mathbb R:\ y=\eta(t,x)\bigr\}.
\]
The function \(\eta=\eta(t,x)\) denotes the vertical displacement of a thin
elastic cover lying on the free surface. We consider an inviscid,
incompressible, homogeneous, and irrotational fluid of constant density equal
to one. Thus the velocity field is of the form
\[
    \mathbf v=\nabla_{x,y}\Psi ,
\]
where the velocity potential \(\Psi=\Psi(t,x,y)\) is harmonic in
\(\Omega_\eta\). We denote by
\[
    \psi(t,x):=\Psi(t,x,\eta(t,x))
\]
the trace of the velocity potential on the free surface. For a given pair
\((\eta,\psi)\), the potential \(\Psi\) is determined by
\begin{equation}\label{eq:harmonic-extension-infinite}
\left\{
\begin{aligned}
    \Delta \Psi &=0
        &&\text{in }\Omega_{\eta},\\
    \Psi(x,\eta(t,x))&=\psi(t,x)
        &&\text{on }S_t,\\
    \nabla_{x,y}\Psi(x,y)&\to 0
        &&\text{as }y\to-\infty .
\end{aligned}
\right.
\end{equation}
The associated infinite-depth Dirichlet--Neumann operator is defined by
\begin{equation}\label{eq:DN-definition-infinite}
    G(\eta)\psi
    :=
    \Psi_y(x,\eta(x))-\Psi_x(x,\eta(x))\eta_x(x).
\end{equation}
In particular, at the flat surface one has
\[
    G(0)=|D|.
\]

The elastic cover contributes both capillary and bending restoring forces. We
introduce the arclength derivative and the curvature of the graph by
\begin{equation}\label{eq:curvature-arclength-infinite}
    \partial_s
    :=
    \frac{1}{\sqrt{1+\eta_x^2}}\partial_x,
    \qquad
    \sigma(\eta)
    :=
    \partial_x
    \left(
        \frac{\eta_x}{\sqrt{1+\eta_x^2}}
    \right).
\end{equation}
The surface energy is
\begin{equation}\label{eq:surface-energy-infinite}
    \mathcal E(\eta)
    :=
    \kappa
    \int_{\mathbb T}
        \bigl(\sqrt{1+\eta_x^2}-1\bigr)\,\de x
    +
    b
    \int_{\mathbb T}
        \frac12 \sigma(\eta)^2\sqrt{1+\eta_x^2}\,\de x,
\end{equation}
where \(\kappa\geq0\) is the surface-tension coefficient and \(b>0\) is the
bending rigidity. The variational derivative of the surface energy is
\[
    -\kappa\sigma(\eta)
    +
    b\left(
        \partial_s^2\sigma(\eta)
        +
        \frac12\sigma(\eta)^3
    \right).
\]
Consequently, since \(\psi_t=-\nabla_\eta\mathcal H\), the dynamic equation
contains the terms
\[
    +\kappa\sigma(\eta)
    -
    b\left(
        \partial_s^2\sigma(\eta)
        +
        \frac12\sigma(\eta)^3
    \right).
\]
The bending term is the quadratic-curvature hydroelastic contribution appearing
in the membrane model of Toland \cite{To2008}.

After normalizing the gravitational constant to \(g=1\), the equations of
motion reduce to the Craig--Sulem system
\begin{equation}\label{Craig-Sulem formula}
\left\{
\begin{aligned}
    \eta_t
    &=
    G(\eta)\psi,\\
    \psi_t
    &=
    -\eta
    -\frac{\psi_x^2}{2}
    +
    \frac{1}{2(1+\eta_x^2)}
    \left(G(\eta)\psi+\eta_x\psi_x\right)^2
    +
    \kappa\sigma(\eta)
    -
    b\left(
        \partial_s^2\sigma(\eta)
        +
        \frac12\sigma(\eta)^3
    \right).
\end{aligned}
\right.
\end{equation}
Equivalently, \eqref{Craig-Sulem formula} is the Hamiltonian system
\begin{equation}\label{eq:hamiltonian-form-infinite}
    \partial_t
    \begin{bmatrix}
        \eta\\
        \psi
    \end{bmatrix}
    =
    \mathcal J
    \begin{bmatrix}
        \nabla_\eta\mathcal H\\
        \nabla_\psi\mathcal H
    \end{bmatrix},
    \qquad
    \mathcal J
    :=
    \begin{bmatrix}
        0 & \mathrm{Id}\\
        -\mathrm{Id} & 0
    \end{bmatrix},
\end{equation}
where
\begin{equation}\label{eq:hamiltonian-infinite}
    \mathcal H(\eta,\psi)
    :=
    \frac12
    \int_{\mathbb T}
        \psi G(\eta)\psi\,\de x
    +
    \frac12
    \int_{\mathbb T}
        \eta^2\,\de x
    +
    \mathcal E(\eta).
\end{equation}
Here \(\nabla_\eta\mathcal H\) and \(\nabla_\psi\mathcal H\) denote the
\(L^2\)-gradients, namely
\begin{align}\label{defL2gra}
    \de_\eta \mathcal{H}[\tilde{\eta}] = \int_\T \nabla_\eta \mathcal{H}  \, \tilde{\eta} \, \de x, \qquad 
    \de_\psi \mathcal{H} [\tilde{\psi}] = \int_\T \nabla_\psi \mathcal{H}  \, \tilde{\psi} \, \de x.
\end{align}

Finally, the system is reversible with respect to the involution
\begin{equation}\label{rho involution}
    \rho
    \begin{bmatrix}
        \eta(x)\\
        \psi(x)
    \end{bmatrix}
    :=
    \begin{bmatrix}
        \eta(-x)\\
        -\psi(-x)
    \end{bmatrix}.
\end{equation}
Indeed, \(\mathcal H\circ\rho=\mathcal H\), and the corresponding Hamiltonian
vector field anti-commutes with \(\rho\).

\paragraph{Stokes waves.}
In a moving reference frame with constant speed \(c\), the water-wave system
\eqref{Craig-Sulem formula} becomes
 \begin{equation} \label{in the reference frame}
\left\{\begin{aligned}
    \eta_t&=c\eta_x+G(\eta)\psi\\
    \psi_t&=c\psi_x-\eta-\frac{\psi_x^2}{2}+\frac{1}{2(1+\eta^2_x)}\left(G(\eta)\psi+\eta_x\psi_x\right)^2+\kappa\sigma(\eta)-b\left(\partial^2_s \sigma(\eta)+\frac{1}{2}\sigma(\eta)^3\right).
\end{aligned}\right.
\end{equation}
We consider small amplitude {\em Stokes waves solutions}, namely stationary solutions of \eqref{in the reference frame} which we further require to be  $2\pi$-periodic in space.
 The bifurcation of small-amplitude Stokes waves from the trivial solution was first studied for pure gravity water waves by Stokes \cite{stokes}, Levi-Civita \cite{LC}, Nekrasov \cite{Nek}, Struik \cite{Struik}. 
 In our setting, the existence and analyticity of  a bifurcating branch of traveling waves follows from the Crandall--Rabinowitz theorem. We denote by $B(r):=\{x\in\mathbb{R}: |x|<r\}$ the open ball of radius $r$ centered at zero.
 \begin{theorem} [Stokes waves] \label{Thm: Stokes expansion} 
 Let $(\kappa,b)\in(\mathbb R_{\geq0}\times\mathbb R_{>0})\setminus\mathfrak R$, where \(\mathfrak R\) is defined in \eqref{def:infinite-R}. There 
 exist $\e_*:=\e_*(\kappa,b) >0$ and a unique family  of real analytic 
 solutions $(\eta_\e(x), \psi_\e(x), c_\e)$, parameterized by the amplitude $|\e| \leq \e_*$, of 
\begin{equation}\label{travelingWWstokes}
c \, \eta_x+G(\eta)\psi = 0 \, , \quad 
c \, \psi_x -  \eta - \dfrac{\psi_x^2}{2} + 
\dfrac{1}{2(1+\eta_x^2)} \big( G(\eta) \psi + \eta_x \psi_x \big)^2+\kappa\sigma(\eta)-b\left(\partial^2_s \sigma(\eta)+\frac{1}{2}\sigma(\eta)^3\right)
= 0 \, , 
\end{equation}
  such that
 $ \eta_\e (x), \psi_\e (x) $ are $2\pi$-periodic;  $\eta_\e (x) $ is even
and $\psi_\e (x) $ is odd, of the form 
 \begin{equation}\label{exp:Sto}
 \begin{aligned}
  & \eta_\e (x) = \e \cos(x)+\e^2\left(\eta_2^{[0]}+\eta_2^{[2]}\cos(2x)\right) +\mathcal{O}(\e^3), \\ 
  & \psi_\e (x)  =  \e c_{\kappa,b} \sin(x)+\e^2 \psi_2^{[2]}\sin(2x)+\mathcal{O}(\e^3),  \\
  & c_\e = c_{\kappa,b}  +\e^2 c_2+\mathcal{O}(\e^3)\quad \text{where} \quad c_{\kappa,b} := (1+b+\kappa)^{\frac{1}{2}} \, , 
   \end{aligned}
  \end{equation}
and 
\begin{align}\label{expcoef}
 &\eta_{2}^{[0]} = 0, \qquad 
\eta_{2}^{[2]} = \frac{1+b+\kappa}{2\,(1-14\,b-2\kappa)}, \qquad \psi_{2}^{[2]} = -\frac{(1+b+\kappa)^{\frac{3}{2}}}{2\,(14\,b+2\kappa-1)}, 
\\  \label{expc2}
&c_2 = \frac{-88\,b^2 - 6\,b\kappa + 54\,b - 2\kappa^2 - \kappa - 8}{16\,(1+b+\kappa)^{\frac{1}{2}}\,(14\,b+2\kappa-1)}.
\end{align}
More precisely for any  $ \alpha \geq  0 $ and $m > \frac{11}{2}, m + \frac12 \in \N$, there exists $ \e_*>0 $ such that
the map $\e \mapsto (\eta_\e, \psi_\e, c_\e)$ is analytic from $B(\e_*) \to H^{\alpha,m}_{\mathtt{ev}} (\T)\times H^{\alpha,m}_{\mathtt{odd}}(\T)\times \R$, where 
$ H^{\alpha,m}_{\mathtt{ev}}(\T) $, respectively $ H^{\alpha,m}_{\mathtt{odd}}(\T) $, denote the  space of even, respectively odd, 
 real valued $ 2 \pi $-periodic analytic functions
$ u(x) = \sum_{k \in \mathbb{Z}} u_k e^{\im k x} $
such that $ \| u \|_{\alpha,m}^2 := \sum_{k \in \mathbb{Z}} |u_k|^2 \langle k \rangle^{2m} 
e^{2 \alpha |k|} < + \infty$.  
\end{theorem}
\color{black}
The expansions in 
\eqref{exp:Sto}--\eqref{expc2} are proved in Appendix \ref{sec:App2}. 
The condition $(\kappa,b)\notin\mathfrak R$ is used in Lemma \ref{lem:B0} to ensure that the kernel of the linearized operator at the flat surface is one-dimensional. 
Indeed, $(\kappa,b)\notin\mathfrak R$ is equivalent to
\[
\frac{\sqrt{n(1+\kappa n^2+b n^4)}}{n}
\neq
c_{\kappa,b},
\qquad n\geq2.
\]
Thus the fundamental phase speed is attained only at the first
harmonic. This is precisely the non-resonance condition required for
the kernel of the flat traveling-wave problem in the even--odd
subspace to be one-dimensional.

\paragraph{Linearization.}
We linearize system
\eqref{in the reference frame} at the Stokes waves $(\eta_\epsilon(x),\psi_\epsilon(x))$ given in Theorem \ref{Thm: Stokes expansion} and evaluate $c$ at $c_\epsilon$. 
Using the shape derivative formula \cite{LD, LD book} 
$
\mathrm{d}_\eta G(\eta)[\hat \eta][\psi] = - G(\eta)(B\hat \eta) - \pa_x( V \hat \eta), 
$
where
the functions $(V(x),B(x))$ are the horizontal and vertical components of the velocity field $(\Phi_x,\Phi_y)$ at the free surface and are given by
\begin{align} \label{espV}
   V:= V(x)&:=-B(\eta_\epsilon)_x+(\psi_\epsilon)_x,\\ \label{espB}
  B:=  B(x)&:=\frac{G(\eta_\epsilon)\psi_\epsilon+(\psi_\epsilon)_x(\eta_\epsilon)_x}{1+(\eta_\epsilon)^2_x}=\frac{(\psi_\epsilon)_x-c_\epsilon}{1+(\eta_\epsilon)^2_x} (\eta_\epsilon)_x,
\end{align}
one obtains the real, autonomous linearized  system
\begin{equation} \label{first linear eq}
    \begin{aligned}
        \begin{bmatrix}
            \hat{\eta}_t\\
            \hat{\psi}_t
        \end{bmatrix}=\underbrace{
        \left[\begin{array}{c|c} 
	 -G(\eta_\epsilon)B-\partial_x\circ(V-c_\epsilon) & G(\eta_\epsilon)  \\ 
	\hline 
	-1-BV_x-BG(\eta_\epsilon)\circ B+\kappa\,\mathrm{d}\sigma(\eta_\e)-b \,\mathrm{d}F(\eta_\e) & -(V-c_\epsilon)\partial_x+B G(\eta_\epsilon)
\end{array}\right]}_{:=\mathscr A_\epsilon}
        \begin{bmatrix}
            \hat{\eta}\\
            \hat{\psi}
        \end{bmatrix},
    \end{aligned}
\end{equation}
\normalsize
where
\begin{align} \label{esp l}
   F(\eta):=\partial^2_s \sigma(\eta)+\frac{1}{2}\sigma(\eta)^3,\qquad \mathrm{d} F(\eta_\e)=\pa^4_x+\e^2\left(F^{[1]}_2\pa_x+F^{[2]}_2\pa^2_x+F^{[3]}_2\pa^3_x+F^{[4]}_2\pa^4_x\right)+\mathcal{O}(\e^3).
\end{align}
and
\begin{equation} \label{F12 F22 F32 F42}
\begin{aligned}
    &F^{[1]}_2(x):=-5(\eta_1)_x(\eta_1)_{xxxx}-10(\eta_1)_{xx}(\eta_1)_{xxx}=\frac{15}{2}\sin(2x),\\
    &F^{[2]}_2(x):=-\frac{15}{2}(\eta_1)^2_{xx}-10(\eta_1)_x(\eta_1)_{xxx}=\frac{5}{4}-\frac{35}{4}\cos(2x),\\
    &F^{[3]}_2(x):=-10(\eta_1)_x(\eta_1)_{xx}=-5\sin(2x),\qquad F^{[4]}_2(x):=-\frac{5}{2}(\eta_1)^2_{x}=\frac{5}{4}\cos(2x)-\frac{5}{4}.
\end{aligned}
\end{equation}
We remark that the operator $\mathrm{d} F$ is symmetric. In fact, from \eqref{defL2gra} we know
\[
\de \mathcal{F}(\eta)[\tilde{\eta}] = \int_\T F(\eta)  \, \tilde{\eta} \, \de x, \qquad\mathcal{F}(\eta) := \frac12  \int_{\mathbb T}
        \sigma(\eta)^2\sqrt{1+\eta_x^2}\,\de x.
\]
Taking another differential w.r.t. $\eta$ we obtain
\[
\de^2 \mathcal{F}(\eta)[\tilde{\eta},\hat{\eta}] = \int_\T \de F(\eta)[\hat{\eta}]  \, \tilde{\eta} \, \de x.
\]
As a result, 
\[
\int_\T \de F(\eta)[\hat{\eta}]  \, \tilde{\eta} \, \de x = \de^2 \mathcal{F}(\eta)[\tilde{\eta},\hat{\eta}] = \de^2 \mathcal{F}(\eta)[\hat{\eta},\tilde{\eta}] = \int_\T \de F(\eta)[\tilde{\eta}]  \, \hat{\eta} \, \de x.
\]
Similarly, the linearized curvature operator $\mathrm{d}\sigma(\eta_\epsilon)[\hat{\eta}] = \partial_x \left( (1+(\eta_\epsilon)_x^2)^{-3/2} \hat{\eta}_x \right)$ expands as
\begin{align} \label{dSigma expansion}
    \mathrm{d}\sigma(\eta_\epsilon) = \partial_x^2 + \epsilon^2 \left( \Sigma_2^{[1]} \partial_x + \Sigma_2^{[2]} \partial_x^2 \right) + \mathcal{O}(\epsilon^3),
\end{align}
where the coefficients are given by
\begin{equation} \label{Sigma coefficients}
\begin{aligned}
    \Sigma_2^{[1]}(x) := -3(\eta_1)_x(\eta_1)_{xx} = -\frac{3}{2}\sin(2x), \qquad\Sigma_2^{[2]}(x) := -\frac{3}{2}(\eta_1)_x^2 = -\frac{3}{4} + \frac{3}{4}\cos(2x).
\end{aligned}
\end{equation}
As before, $\de \sigma$ is symmetric. The map  $\e \to (V, B)$ is analytic as a map $B(\e_0) \to H^{\alpha, m-1}_{\mathtt{ev}}(\T) \times H^{\alpha, m-1}_{\mathtt{odd}}(\T)$. 
The real system \eqref{first linear eq} is Hamiltonian, i.e. of the form
$\cJ \cA$ with $\cA$ symmetric with respect to the real scalar product of $L^2(\T, \R^2) = L^2(\T, \R) \times L^2(\T, \R)$.
Moreover the linear operator in \eqref{first linear eq} is reversible, namely it anti-commutes with the involution $\rho$ in \eqref{rho involution}.

Next, we conjugate \eqref{first linear eq} by using the time-independent ``good unknown of Alinhac'' linear transformation 
\begin{align}\label{alinhac}
    \begin{bmatrix}
        \hat{\eta}\\
        \hat{\psi}
    \end{bmatrix}:= Z
    \begin{bmatrix}
        u\\
        v
    \end{bmatrix},~~Z=\begin{bmatrix}
        1&0\\
        B&1
    \end{bmatrix},
    ~~Z^{-1}=\begin{bmatrix}
        1&0\\
        -B&1
    \end{bmatrix},
\end{align}
yielding the linear system 
\begin{align} \label{second linear eq}
    \begin{bmatrix}
        u_t\\
        v_t
    \end{bmatrix}=  \widetilde{\cL}_\e 
    \begin{bmatrix}
        u\\
        v
    \end{bmatrix} \ , \qquad 
    \widetilde{\cL}_\e := \begin{bmatrix}
        -\partial_x\circ (V-c_\epsilon)& G(\eta_\epsilon)\\
        -1-(V-c_\epsilon)B_x+\kappa\,\mathrm{d}\sigma(\eta_\e)-b\,\mathrm{d} F(\eta_\e) ~~& -(V-c_\epsilon)\partial_x
    \end{bmatrix} \ , 
\end{align}
which  is Hamiltonian and reversible since the transformation $Z$ is symplectic, $\mathrm{i.e.}$ $Z^T\mathcal{J} Z=\mathcal{J}$, and satisfies $Z\circ \rho=\rho\circ Z$.

Next, 
we perform a conformal change of variables to flatten 
the water surface. 
By \cite[Appendix A]{BBHM}, 
 there exists a diffeomorphism of $\mathbb{T}$,
 $ x\mapsto x+\mathfrak{p}(x)$, with a small $2\pi$-periodic  function $\mathfrak{p}(x)$, 
 and a small constant $\ttf $, such that, by defining the associated composition operator $ (\mathfrak{P}u)(x) := u(x+\mathfrak{p}(x))$, the Dirichlet-Neumann operator writes as \cite[Lemma A.5]{BBHM}
\begin{equation}\label{Gneta}
 G(\eta_\e) = \pa_x \circ \mathfrak{P}^{-1} \circ {\mathfrak H} 
 \circ \mathfrak{P} \, , 
\end{equation}
where $ {\mathfrak H} $ is the Hilbert transform, i.e. the  Fourier multiplier operator
$$
 \mathfrak{H}(e^{\im j x}):= - \im\, \textup{sign}(j) e^{\im j x} \, , 
 \quad  \forall j \in \Z \setminus \{0\} \, , 
 \quad \mathfrak{H}(1) := 0 \, . 
$$
The function $\mathfrak p(x)$ is determined as a fixed point of 
(see  \cite[formula (A.15)]{BBHM})
\begin{equation} \label{def:ttf}
\mathfrak{p}  = \mathfrak{H}[\eta_\e ( x + \mathfrak{p}(x))]  
  \end{equation}
  As proved in \cite{BMV3}, the map $\e \to \mathfrak{p}$ is analytic as a map $B(\e_0) \to H^s_{\mathtt{odd}}(\T) $.
In addition, in  Appendix \ref{sec:App2} we prove the expansion
\begin{equation}
 \label{expfe}
 \begin{aligned}
&  \mathfrak p(x)  = \e  \sin(x) +\e^2 \frac{2-13b-\kappa}{2(1-14b-2\kappa)} \sin(2x)+\mathcal{O}(\e^3) \, ,
   \end{aligned}
 \end{equation}
Under the symplectic and reversibility-preserving change of variables 
\begin{align}\label{LC}
    h=\mathcal{P}\begin{bmatrix}
        u\\
        v
\end{bmatrix},~~\mathcal{P}=\begin{bmatrix}
    (1+\mathfrak{p}_x)\mathfrak{P} & 0\\
    0 & \mathfrak{P}
\end{bmatrix} \ , 
\end{align}
one  transforms the system \eqref{second linear eq} into the linear system $h_t=\mathcal{L}_\epsilon h$ where $\mathcal{L}_\epsilon$ is the Hamiltonian and reversible real operator
\begin{equation} \label{mathcal L e}
\begin{aligned}
    \mathcal{L}_\epsilon:= \mathcal{L}_{\epsilon}(b) :=\mathcal{P} \widetilde{\cL}_\e \mathcal{P}^{-1} = &\begin{bmatrix}
\partial_x\circ(c_{\kappa,b}+p_\epsilon(x)) &  |D| \\
        -(1+a_\epsilon(x))+\kappa\,\tau_\e-b\,\beta_\e & ~~(c_{\kappa,b}+p_\epsilon(x))\partial_x
    \end{bmatrix}\\
    =&\begin{bmatrix}
        0& \mathrm{Id}\\
        -\mathrm{Id}& 0
    \end{bmatrix}\underbrace{\begin{bmatrix}
 (1+a_\epsilon(x))-\kappa\,\tau_\e+b\,\beta_\e& ~~-(c_{\kappa,b}+p_\epsilon(x))\partial_x\\
 \partial_x\circ(c_{\kappa,b}+p_\epsilon(x))& ~~|D|
    \end{bmatrix}}_{:= \mathfrak{B}_\e},
\end{aligned}    
\end{equation}
where the functions $p_\e(x)$ and $a_\e(x)$ are given by
\begin{equation}\label{def:pa}
c_{\kappa,b}+p_\e(x) :=  \displaystyle{\frac{ c_\e-V(x+\mathfrak{p}(x))}{ 1+\mathfrak{p}_x(x)}} \, , \quad 1+a_\e(x):=   \displaystyle{\frac{1+ (V(x + \mathfrak{p}(x)) - c_\e)
 B_x(x + \mathfrak{p}(x))  }{1+\mathfrak{p}_x(x)}} \,, 
\end{equation}
and the operators $\beta_\e:=\mathfrak{P}\circ\mathrm dF(\eta_\e)\circ\mathfrak{P}^{-1}\circ\frac{1}{1+\mathfrak{p}_x}$ and $\tau_\e:=\mathfrak{P}\circ\mathrm d\sigma(\eta_\e)\circ\mathfrak{P}^{-1}\circ\frac{1}{1+\mathfrak{p}_x}$ are given by
\begin{equation} \label{def:Sigma g}
\begin{aligned}
     \beta_\e: =&\frac{1}{1+\mathfrak{p}_x}\left(\widetilde{\pa_x}\right)^4+\e^2 \sum_{j=1}^4 \frac{F^{[j]}_2(x+\mathfrak{p}(x))}{1+\mathfrak{p}_x}\left(\widetilde{\pa_x}\right)^j+\mathcal{O}(\e^3),
\end{aligned}    
\end{equation}

\begin{equation} \label{tau operator}
\begin{aligned}
    \tau_\epsilon &:= \frac{1}{1+\mathfrak{p}_x} \left(\widetilde{\pa_x}\right)^2 + \epsilon^2 \sum_{j=1}^2 \frac{\Sigma_2^{[j]}(x+\mathfrak{p}(x))}{1+\mathfrak{p}_x} \left(\widetilde{\pa_x}\right)^j + \mathcal{O}(\epsilon^3).
\end{aligned}
\end{equation}
where $\widetilde{\pa_x}:=\pa_x\circ\frac{1}{1+\mathfrak{p}_x}$. The operators $\beta_\varepsilon$ and $\tau_\varepsilon$ are also symmetric. To see this we apply a change of variable to obtain 
\[
\int_\T \mathfrak{P} f(y)  g(y) \de y = \int_\T f(z)  \mathfrak{P}^{-1}\left[\frac{g}{1 + \mathfrak{p}_x} \right](z) \de y.
\]
By the analyticity result of the map $\e \mapsto (V, B)$ given above,  the map
$\e \to (p_\e, a_\e)$ is analytic as a map $B(\e_0) \to H^{m - 1}_{\mathtt{ev}}(\T)\times H^{m - 2}_{\mathtt{ev}}(\T)$.
In Appendix \ref{sec:App2}, we derive their Taylor expansions, which are collected in the following lemmas.
\begin{lemma}\label{lem:pa.exp}
The analytic functions $p_\e (x) $ and $a_\e (x) $  in \eqref{def:pa} 
are even in $ x $, and
\begin{equation}\label{SN1}
p_\epsilon (x) = \epsilon p_1 (x) + \epsilon^2 p_2 (x) + \mathcal{O}(\epsilon^3) \, , \qquad
a_\epsilon (x) = \epsilon a_1(x) +\epsilon^2 a_2 (x) + \mathcal{O}(\epsilon^3) \, ,
\end{equation}
where
\begin{align}\label{pino1fd}
    p_1(x) & = p_1^{[1]} \cos(x)\, , \qquad \quad p_1^{[1]} := -2c_{\kappa,b},\\
 \label{pino2fd}
    p_2(x) & =p_2^{[0]}+p_2^{[2]}\cos(2x),\qquad p_2^{[0]} := c_2+c_{\kappa,b},\qquad p_2^{[2]} := \frac{2c_{\kappa,b}^{3}}{14b+2\kappa-1},
\end{align}
and
\begin{align}
a_1(x) \label{aino1fd} &= a_1^{[1]}\cos(x)\, , \qquad \qquad
a_1^{[1]}:= -(2+b+\kappa)\, , \\
a_2(x)& \label{aino2fd} = a_2^{[0]}+a_2^{[2]}\cos(2x)\, ,\quad \, a_2^{[0]}:=\frac{3b+3\kappa + 4}{2}, \qquad a_2^{[2]} := \frac{46b^2+56b\kappa+35b+10\kappa^2+11\kappa+4}{2(14b+2\kappa - 1)}.
\end{align}
Finally, the self-adjoint operator $\beta_\e$ in \eqref{def:Sigma g} expands 
\begin{align} \label{betaeps}
    \beta_\e=\pa^4_x+\e\beta_1+\e^2\beta_2+\mathcal{O}(\e^3),
\end{align}
where
\begin{align} \label{betaj}
    \beta_j:=b_{j,0}(x)+b_{j,1}(x)\pa_x+b_{j,2}(x)\pa^2_x+b_{j,3}(x)\pa^3_x+b_{j,4}(x)\pa^4_x,\qquad j=1,2
\end{align}
and
\begin{align} \label{b110}
    &b_{1,0}=b^{[1]}_{1,0}\cos(x),\qquad b^{[1]}_{1,0}:=-1,\\
    &b_{1,1}=b^{[1]}_{1,1}\sin(x),\qquad  b^{[1]}_{1,1}:=-5,\\
    &b_{1,2}=b^{[1]}_{1,2}\cos(x),\qquad  b^{[1]}_{1,2}:=10,\\
    &b_{1,3}=b^{[1]}_{1,3}\sin(x),\qquad  b^{[1]}_{1,3}:=10,\\
    &b_{1,4}=b^{[1]}_{1,4}\cos(x),\qquad  b^{[1]}_{1,4}:=-5,\\
    &b_{2,0}:=b_{2,0}^{[0]}+b_{2,0}^{[2]}\cos(2x),\qquad b^{[0]}_{2,0}:=\frac{1}{2}, \qquad b^{[2]}_{2,0}:=\frac{18\,b+30\kappa+33}{2\,\left(14\,b+ 2\kappa-1\right)},\\
    &b_{2,1}:=b^{[2]}_{2,1}\sin(2x),\qquad b^{[2]}_{2,1}:=\frac{5\,\left(86\,b+26\kappa+11\right)}{2\,\left(14\,b+ 2\kappa-1\right)},\\
    &b_{2,2}:=b^{[0]}_{2,2}+b^{[2]}_{2,2}\cos(2x),\quad b^{[0]}_{2,2}:=-\frac{25}{4},\quad b^{[2]}_{2,2}:=-\frac{15\,\left(90\,b+22\kappa+5\right)}{4\,\left(14\,b+ 2\kappa-1\right)},\\
    &b_{2,3}:=b^{[2]}_{2,3}\sin(2x),\qquad b^{[2]}_{2,3}:=-\frac{5\,\left(46\,b+10\kappa+1\right)}{14\,b+ 2\kappa-1},\\ \label{b224}
    &b_{2,4}:=b^{[0]}_{2,4}+b^{[2]}_{2,4}\cos(2x),\quad b^{[0]}_{2,4}:=\frac{25}{4},\quad b^{[2]}_{2,4}:=\frac{5\,\left(46\,b+10\kappa+1\right)}{4\,\left(14\,b+ 2\kappa-1\right)}.
\end{align}
\end{lemma}
\begin{lemma}\label{lem:tau.exp}
The self-adjoint operator $\tau_\epsilon$ defined in \eqref{tau operator} expands as
\begin{align} \label{taueps}
    \tau_\epsilon = \partial_x^2 + \epsilon \tau_1 + \epsilon^2 \tau_2 + \mathcal{O}(\epsilon^3),
\end{align}
where
\begin{align} \label{tauj}
   \tau_j := \tau_{j,0}(x) + \tau_{j,1}(x)\partial_x + \tau_{j,2}(x)\partial_x^2, \qquad j = 1, 2
\end{align}
and the coefficients are given by
\begin{align}
    \tau_{1,0} &= \tau_{1,0}^{[1]}\cos(x), \qquad \tau_{1,0}^{[1]} := 1, \label{tau10} \\
    \tau_{1,1} &= \tau_{1,1}^{[1]}\sin(x), \qquad \tau_{1,1}^{[1]} := 3, \label{tau11} \\
    \tau_{1,2} &= \tau_{1,2}^{[1]}\cos(x), \qquad \tau_{1,2}^{[1]} := -3, \label{tau12}
\end{align}
and for the second order:
\begin{align}
    \tau_{2,0} &= \tau_{2,0}^{[0]} + \tau_{2,0}^{[2]}\cos(2x), \qquad \tau_{2,0}^{[0]} := -\frac{1}{2}, \qquad \tau_{2,0}^{[2]} := \frac{6b - 6\kappa - 9}{2(14b + 2\kappa - 1)}, \label{tau20} \\
    \tau_{2,1} &= \tau_{2,1}^{[2]}\sin(2x), \qquad \tau_{2,1}^{[2]} := \frac{9(-6b - 2\kappa - 1)}{2(14b + 2\kappa - 1)}, \label{tau21} \\
    \tau_{2,2} &= \tau_{2,2}^{[0]} + \tau_{2,2}^{[2]}\cos(2x), \qquad \tau_{2,2}^{[0]} := \frac{9}{4}, \qquad \tau_{2,2}^{[2]} := \frac{9(6b + 2\kappa + 1)}{4(14b + 2\kappa - 1)}. \label{tau22}
\end{align}
\end{lemma}
\paragraph{Bloch-Floquet expansions.} 
In the following we regard $\mathcal{L}_\e$ as an operator on $L^2(\R,\mathbb{C}^2)$ instead of on $L^2(\T,\mathbb{C}^2)$. To be more precise, we interpret the $D$ as the multiplier $\xi$ on the Fourier transform, namely the operator $m(D)$ is given by the following formula
\[
m(D) u := (m(\xi) \widehat{u}(\xi))^\vee,
\]
where $\widehat{u}$ denotes the Fourier transform of $u$ and $\vee$ denotes the inverse Fourier transform. The following Bloch transform gives an isometry between $L^2(\R,\mathbb{C}^2)$ and $L^2(Q,L^2(\T,\mathbb{C}^2)) \cong L^2(Q) \otimes L^2(\T,\mathbb{C}^2)$:
\begin{align}
    \label{Blocht}
    \mathcal{U} f(\mu,x) = \sqrt{2\pi} \sum_{k \in \Z} f(x + 2\pi k) e^{-\im \mu (2 \pi k + x)} = \frac{1}{\sqrt{2\pi}} \sum_{k \in \Z} \widehat{f}(k + \mu) e^{\im kx},
\end{align}
where $\mu \in Q := \left[-\frac{1}{2},\frac{1}{2}\right)$, $x \in \T$. Here the second equality in \eqref{Blocht} is given by Poisson summation formula first by smooth functions with compact support then by density (see the isometry below). In fact, we have 
\[
\int_Q \int_\T |\mathcal{U} f(\mu,x)|^2 \de \mu \de x = \frac{1}{2\pi} \int_Q \sum_{k \in \Z} |\widehat{f}(k + \mu)|^2 \de \mu  = \frac{1}{2\pi} \int_\R |\widehat{f}(\mu)|^2 \de \mu = \int_\R |f(x)|^2 \de x.
\]

Now for fixed $\mu \in Q$, we define the following operator $\mathcal{L}_{\mu,\e}:=e^{-\im\mu x}\mathcal{L}_\e e^{\im\mu x}$, namely
\begin{equation} \label{mathcal L mu e}
\begin{aligned}
    \mathcal{L}_{\mu,\e}:=&\begin{bmatrix}
(\partial_x+\im\mu)\circ(c_{\kappa,b}+p_\epsilon(x)) &  |D+\mu| \\
        -(1+a_\epsilon(x))+\kappa\tau_{\mu,\e}-b\beta_{\mu,\e} & ~~(c_{\kappa,b}+p_\epsilon(x))(\partial_x+\im\mu)
    \end{bmatrix}\\
    =&\underbrace{\begin{bmatrix}
        0& \mathrm{Id}\\
        -\mathrm{Id}& 0
    \end{bmatrix}}_{=:\mathcal{J}}\underbrace{\begin{bmatrix}
 (1+a_\epsilon(x))-\kappa\tau_{\mu,\e}+b\beta_{\mu,\e}& ~~-(c_{\kappa,b}+p_\epsilon(x))(\partial_x+\im\mu)\\
 (\partial_x+\im\mu)\circ(c_{\kappa,b}+p_\epsilon(x))& ~~|D+\mu| 
    \end{bmatrix}}_{=:\mathfrak{B}_{\mu,\epsilon}},
\end{aligned}    
\end{equation}
where $\tau_{\mu,\e}:=e^{-\im\mu x}\tau_\e e^{\im\mu x}$ and $\beta_{\mu,\e}:=e^{-\im\mu x}\beta_\e e^{\im\mu x}$ are given by the selfadjoint operators
\begin{equation} \label{tau_mu_expansion}
    \tau_{\mu,\epsilon} = (\partial_x + \im\mu)^2 + \epsilon \sum_{j=0}^2 \tau_{1,j}(x) (\partial_x + \im\mu)^j + \epsilon^2 \sum_{j=0}^2 \tau_{2,j}(x) (\partial_x + \im\mu)^j + \mathcal{O}(\epsilon^3),
\end{equation}
\begin{align}\label{Sigma}
    \beta_{\mu,\e}=(\pa_x+\im \mu)^4+\e \sum_{j=0}^4 b_{1,j}(x)\left(\pa_x+\im\mu\right)^{j}+\e^2 \sum_{j=0}^4 b_{2,j}(x)\left(\pa_x+\im\mu\right)^{j}+\mathcal{O}(\e^3).
\end{align}

In fact, if we regard $\mathcal{L}_{\mu,\e}$ as an operator with domain $Y:=H^4(\mathbb{T},\mathbb{C})\times H^1(\mathbb{T},\mathbb{C})$ in the space $X:=L^2(\mathbb{T},\mathbb{C})\times L^2(\mathbb{T},\mathbb{C})$, equipped with the complex scalar product
\begin{align} \label{complex product}
    (f,g):=\frac{1}{2\pi}\int_0^{2\pi} \left(f_1\overline{g_1}+f_2\overline{g_2} \right)\,\mathrm{d}x,~~\forall~f=\vet{f_1}{f_2},~~g=\vet{g_1}{g_2}\in L^2(\mathbb{T},\mathbb{C}^2),
\end{align}
it is not hard to show that $\mathfrak{B}_{\mu,\e}$ is self-adjoint. By definition of \eqref{Blocht} we can check that (noticing that the operator $\mathcal{L}_\e$ in \eqref{mathcal L e} has $2\pi$-periodic coefficients and if $A=Op(a)$ is a pseudo-differential operator with symbol $a(x,\xi)$, which is $2\pi$-periodic in $x$, then $A_\mu:=e^{-\im\mu x}A e^{\im\mu x}=Op(a(x,\xi+\mu))$)
\[
\mathcal{U}(\mathfrak{B}_\e f)(\mu, \cdot) = \mathfrak{B}_{\mu,\e} \, \mathcal{U}f(\mu, \cdot).
\]
{\color{black}In other words, if we define $B_{\e}$ on $L^2(Q,L^2(\T,\mathbb{C}^2))$ by
\[
(B_{\e}f) (\mu)  = \mathfrak{B}_{\mu,\e} (f(\mu)),
\]
we have $\mathfrak B_\e = \mathcal{U}^{-1}{B}_{\e} \mathcal{U}$. From the expression \eqref{Blocht} we know that $\mathcal{J}$ interwines with $\mathcal{U}$. So we obtain that $\mathcal{L}_\e = \mathcal{U}^{-1}{L}_{\e} \mathcal{U}$ with $L_{\e}$ on $L^2(Q,L^2(\T,\mathbb{C}^2))$ by
\[
(L_{\e}f) (\mu)  = \mathcal{L}_{\mu,\e} (f(\mu)).
\]
Note that  we cannot apply the Bloch--Floquet theory  \cite[(d) of Theorem XIII.85]{RS78} directly since $\mathcal{L}_{\mu,\e}$ is not self-adjoint. Instead, with the same spirit with the Bloch--Floquet theory, we apply \cite[Theorem 3.1]{DMT} with $X = Q$, $E = L^2(\T,\mathbb{C}^2)$, and $D = Y$ here. We actually get
\begin{align*}
    \sigma_{L^2(\mathbb{R},\mathbb{C}^2)}(\mathcal{L}_\e)=\bigcup_{\mu\in[-\frac{1}{2},\frac{1}{2})} \sigma_{L^2(\mathbb{T},\mathbb{C}^2)}(\mathcal{L}_{\mu,\e}).
\end{align*}
}



Since $\mathcal{L}_\epsilon$ is a real operator and
\[
\mathcal{L}_{\mu,\epsilon}
=
e^{-\im\mu x}\mathcal{L}_\epsilon e^{\im\mu x},
\]
one has
\begin{align} \label{bloch-reality-symmetry}
\overline{\mathcal{L}_{\mu,\epsilon}}
=
\mathcal{L}_{-\mu,\epsilon},
\qquad
\sigma(\mathcal{L}_{-\mu,\epsilon})
=
\overline{\sigma(\mathcal{L}_{\mu,\epsilon})}.
\end{align}
Consequently, it is sufficient to study nonnegative Floquet exponents.
Moreover, the Bloch spectrum is $1$-periodic in $\mu$, so we may restrict
to $\mu\in[0,\frac12)$. Results for negative Floquet exponents are then
recovered by complex conjugation.

The complex operator $\mathcal{L}_{\mu,\e}$ in \eqref{mathcal L mu e} is complex Hamiltonian and reversible. Recall that if $\mathcal{L} : Y \to X$ is a complex linear operator, we say that it is 
\begin{itemize}
    \item \textbf{Complex Hamiltonian:} if there exists a self-adjoint operator
    $\mathfrak{B}=\mathfrak{B}^*$, where $\mathfrak{B}^*$ denotes the adjoint,
    with domain $Y$, with respect to the complex scalar product
    \eqref{complex product}, such that
    $\mathcal{L}=\mathcal{J}\mathfrak{B}$.

    \item \textbf{Reversible:} if
    \begin{align}\label{reversible}
        \mathcal{L}\circ\overline{\rho}
        =
        -\overline{\rho}\circ\mathcal{L},
        \qquad
        \text{where}
        \qquad
        \overline{\rho}
        \begin{bmatrix}
            \eta(x)\\
            \psi(x)
        \end{bmatrix}
        :=
        \begin{bmatrix}
            \overline{\eta}(-x)\\
            -\overline{\psi}(-x)
        \end{bmatrix}.
    \end{align}
\end{itemize}

The property \eqref{reversible} for $\mathcal{L}_{\mu,\e}$ follows
because $\mathcal{L}_{\e}$ is a real operator which is reversible with
respect to the involution $\rho$ in \eqref{rho involution}.
Equivalently, since
$\mathcal{J}\circ\overline{\rho}
=-\overline{\rho}\circ\mathcal{J}$,
the self-adjoint operator $\mathfrak{B}_{\mu,\e}$ is
reversibility-preserving, i.e.,
\begin{align}\label{B rho=rho B}
    \mathfrak{B}_{\mu,\e}\circ\overline{\rho}
    =
    \overline{\rho}\circ\mathfrak{B}_{\mu,\e}.
\end{align}

The physical family is continuous for $\mu\in Q$, but it is not
two-sided analytic at $\mu=0$, since the zero Fourier mode of
$|D+\mu|$ is equal to $|\mu|$. On the positive Bloch half-branch,
one has \cite[Section 5.1]{NS}
\begin{equation}\label{D+mu-positive}
    |D+\mu|
    =
    |D|+\mu\bigl(\mathrm{sgn}(D)+\Pi_0\bigr),
    \qquad
    0\leq\mu<\frac12,
\end{equation}
where $\mathrm{sgn}(D)$ is the Fourier multiplier with symbol
\begin{equation}\label{sgn D}
    \mathrm{sgn}(k)=1\quad(k>0),
    \qquad
    \mathrm{sgn}(0)=0,
    \qquad
    \mathrm{sgn}(k)=-1\quad(k<0),
\end{equation}
and
\[
    \Pi_0f
    :=
    \frac{1}{2\pi}\int_{\mathbb T}f(x)\,\mathrm{d}x.
\]
To apply analytic perturbation theory on an open neighborhood of the
origin, we extend the operator family in \eqref{D+mu-positive} to
negative values of $\mu$ by the same right-hand side, namely,
\begin{equation}\label{D+mu}
    |D+\mu|
    :=
    |D|+\mu\bigl(\mathrm{sgn}(D)+\Pi_0\bigr),
    \qquad
    |\mu|<\frac12.
\end{equation}
This extension is affine, and hence analytic, in $\mu$, and it agrees
with the physical multiplier $|D+\mu|$ for $\mu\geq0$. In all the
perturbative arguments below, the notation $|D+\mu|$ refers to this
analytic extension. For $\mu<0$, it should therefore not be confused
with the physical Fourier multiplier having symbol $|k+\mu|$.
Consequently,
\[
    (\mu,\e)\longmapsto
    \mathcal{L}_{\mu,\e}
    \in\mathcal{L}(Y,X)
\]
is analytic for $|\mu|$ and $|\e|$ sufficiently small. The physical
Bloch spectrum for negative Floquet exponents is not obtained from the
negative part of this analytic continuation; it is recovered separately
from the reality symmetry \eqref{bloch-reality-symmetry}.

Our goal is to prove the existence of eigenvalues of $\mathcal{L}_{\mu,\e}$ in \eqref{mathcal L mu e} with non zero real part. We remark that the Hamiltonian structure of $\mathcal{L}_{\mu,\e}$ implies that eigenvalues with non zero real part may arise only from multiple eigenvalues of $\mathcal{L}_{\mu,0}$ (``Krein criterion''), because if $\lambda$ is an eigenvalue of $\mathcal{L}_{\mu,\e}$ then also $-\overline{\lambda}$ is, and the total algebraic multiplicity of the eigenvalues is conserved under small perturbation. We now describe the spectrum of $\mathcal{L}_{\mu,0}$.

\paragraph{The spectrum of $\mathcal{L}_{\mu,0}$.} The spectrum of the Fourier multiplier matrix operator 
\begin{equation}\label{eq:unperturbed-operator}
\mathcal{L}_{\mu,0}
:=
\begin{bmatrix}
c_{\kappa,b}(\partial_x+\im\mu) & |D+\mu|\\
-1+\kappa(\partial_x+\im\mu)^2-b(\partial_x+\im\mu)^4 & c_{\kappa,b}(\partial_x+\im\mu)
\end{bmatrix}.
\end{equation}
For $0\leq\mu<1/2$, the spectrum of the physical unperturbed fiber
consists of the purely imaginary eigenvalues
$\{\lambda_k^\pm(\mu),\, k\in\mathbb Z\}$, where
\begin{equation}\label{eq:unperturbed-dispersion}
\lambda_k^\pm(\mu)
:=
\im\Big(
c_{\kappa,b}(\pm k+\mu)\mp \sqrt{(1+\kappa(k\pm\mu)^2+b(k\pm\mu)^4)\,|k\pm\mu|}
\Big).
\end{equation}
For $(\kappa,b)\notin \mathfrak R$ (cf.\ \eqref{def:infinite-R}), the real operator $\mathcal{L}_{0,0}$ possesses the eigenvalue $0$
with algebraic multiplicity $4$,
\[
\lambda_0^+(0)=\lambda_0^-(0)=\lambda_1^+(0)=\lambda_1^-(0)=0,
\]
and geometric multiplicity $3$. A real basis of the kernel of $\mathcal{L}_{0,0}$ is
\begin{equation}\label{eq:eigenfunc-L00-kernel}
f_1^+
:=\vet{(1+b+\kappa)^{-1/4}\cos (x)}{(1+b+\kappa)^{1/4}\sin(x)},
\qquad
f_1^-
:=\vet{-(1+b+\kappa)^{-1/4}\sin(x)}{(1+b+\kappa)^{1/4}\cos(x)},
\qquad
f_0^-:=\vet{0}{1},
\end{equation}
together with the generalized eigenvector
\begin{equation}\label{eq:generalized-eigenfunc-f0plus}
f_0^+:=\vet{1}{0},
\qquad
\mathcal{L}_{0,0}f_0^+=-f_0^-.
\end{equation}

Furthermore, $0$ is an isolated eigenvalue of $\mathcal{L}_{0,0}$, namely
\begin{equation}\label{eq:sigma-decomposition-L00}
\sigma(\mathcal{L}_{0,0})=\sigma'(\mathcal{L}_{0,0})\cup \sigma''(\mathcal{L}_{0,0}),
\qquad
\sigma'(\mathcal{L}_{0,0})=\{0\},
\end{equation}
and
\[
\sigma''(\mathcal{L}_{0,0})
=
\{\lambda_k^\sigma(0):\, k\neq 0,1,\ \sigma=\pm\}.
\]
\color{black}
The generalized kernel of the periodic fiber at $\mu=0$ is described
in Proposition~\ref{prop:periodic-generalized-kernel} below.

\color{black}
By Kato's perturbation theory for any $\mu,\e\neq 0$ sufficiently small, the perturbed spectrum $\sigma(\mathcal{L}_{\mu,\e})$ admits a disjoint decomposition as 
\begin{align} \label{disjoint decomposition of spectrum}
    \sigma(\mathcal{L}_{\mu,\e})=\sigma'(\mathcal{L}_{\mu,\e})\cup \sigma''(\mathcal{L}_{\mu,\e}),
\end{align}
where $\sigma'(\mathcal{L}_{\mu,\e})$ consists of $4$ eigenvalues close to $0$. We denote by $\mathcal{V}_{\mu,\e}$ the spectral subspace associated with $\sigma'(\mathcal{L}_{\mu,\e})$, which has dimension $4$ and it is invariant by $\mathcal{L}_{\mu,\e}$. Our goal is to prove that, for $\e$ small, for values of the Floquet exponent $\mu$ in an interval of order $\e$, the $4\times 4$ matrix which represents the operator $\mathcal{L}_{\mu,\e}:\mathcal{V}_{\mu,\e}\rightarrow \mathcal{V}_{\mu,\e}$ possesses a pair of eigenvalues close to zero with opposite non zero real parts. 

Before stating our main result, let us introduce a notation that we shall use throughout the paper.

\begin{itemize}
\item
\textbf{Notation:} We write $r(\mu,\epsilon)
=
\mathcal O\bigl(
\mu^{m_1}\epsilon^{n_1},\ldots,
\mu^{m_p}\epsilon^{n_p}
\bigr)$, where $m_j,n_j\in\mathbb N_0$ and $m_j+n_j\geq1$, if $r$ is
analytic and
\[
\|r(\mu,\epsilon)\|_X
\leq
C\sum_{j=1}^p
|\mu|^{m_j}|\epsilon|^{n_j}
\]
for $(\mu,\epsilon)$ sufficiently small.
\end{itemize}

Our main spectral result is the following one:
\begin{theorem}[Complete Benjamin--Feir spectrum]\label{Complete BF thm}
Let $(\kappa,b)\in(\mathbb R_{\geq0}\times\mathbb R_{>0})\setminus\mathfrak{R}$, where $\mathfrak{R}$ is defined in \eqref{def:infinite-R}.
There exist $\epsilon_0,\mu_0>0$ such that, for any $0<\mu<\mu_0$ and $0\leq\epsilon<\epsilon_0$, there exists a
symplectic and reversibility-preserving change of coordinates on $\mathcal{V}_{\mu,\epsilon}$, equivalently a
symplectic and reversible basis $\widetilde{\mathcal{G}}_{\mu,\epsilon}$ of $\mathcal{V}_{\mu,\epsilon}$, in which
the operator $\mathcal{L}_{\mu,\epsilon}:\mathcal{V}_{\mu,\epsilon}\to\mathcal{V}_{\mu,\epsilon}$ is represented by a
$4\times4$ block-diagonal matrix of the form
\begin{equation}\label{US-diag}
\begin{pmatrix}
U & 0\\
0 & S
\end{pmatrix},
\end{equation}
where $U$ and $S$ are $2\times2$ matrices with identical diagonal entries, of the form
\begin{equation}\label{U-S-infinite}
\begin{aligned}
U &=
\begin{pmatrix}
\im\,\frac{1}{2}\breve{c}_{\kappa,b}\,\mu
+\im\,\mathcal{O}(\mu\epsilon^2,\mu^2\epsilon,\mu^3)
&
-\,\frac{e_{22}}{8}\,\mu^2+\mathcal{O}(\mu^2\epsilon,\mu^3)
\\[1.2ex]
-\,e_{11}\,\epsilon^2+\frac{e_{22}}{8}\,\mu^2
+\mathcal{O}(\epsilon^3,\mu\epsilon^2,\mu^2\epsilon,\mu^3)
&
\im\,\frac{1}{2}\breve{c}_{\kappa,b}\,\mu
+\im\,\mathcal{O}(\mu\epsilon^2,\mu^2\epsilon,\mu^3)
\end{pmatrix},
\\
S &=
\begin{pmatrix}
\im c_{\kappa,b}\mu+\im\,\mathcal{O}(\mu\epsilon^2,\mu^2\epsilon,\mu^3)
&
\mu+\mathcal{O}(\mu^2\epsilon,\mu^3)
\\[1.2ex]
-1-\kappa\mu^2
+\mathcal{O}(\epsilon^3,\mu\epsilon^2,\mu^2\epsilon,\mu^3)
&
\im c_{\kappa,b}\mu+\im\,\mathcal{O}(\mu\epsilon^2,\mu^2\epsilon,\mu^3)
\end{pmatrix}.
\end{aligned}
\end{equation}
Here $\breve{c}_{\kappa,b}:=2c_{\kappa,b}-e_{12}$, the constants $e_{11}$, $e_{12}$, $e_{22}$ are defined in
Proposition~\ref{lem:matrix-representation-G}, and $c_{\kappa,b}=(1+b+\kappa)^{1/2}$.
The basis $\widetilde{\mathcal{G}}_{\mu,\epsilon}$ is obtained from the basis $\mathcal{G}$ in
Lemma~\ref{lem:basis-G} by the block-decoupling transformations constructed in Section~\ref{sec:BD}; in the original basis
$\mathcal{G}$ the matrix of $\mathscr{L}_{\mu,\epsilon}$ has a nonzero off-diagonal block $F$ as in \eqref{4.19}.

The eigenvalues of $U$ are
\begin{equation}\label{lambda1-infinite}
\lambda_1^\pm(\mu,\epsilon)
=
\im\,\frac{1}{2}\breve{c}_{\kappa,b}\,\mu
+\im\,\mathcal{O}(\mu\epsilon^2,\mu^2\epsilon,\mu^3)
\pm
\frac{\mu}{8}\sqrt{\Delta_{\mathrm{BF}}(\kappa,b;\mu,\epsilon)},
\end{equation}
where $\Delta_{\mathrm{BF}}$ is the Benjamin--Feir discriminant defined in
\eqref{BF-discriminant-block}--\eqref{BFDF} below.
The pair $\lambda_1^\pm$ has nonzero real parts if and only if
\[
\Delta_{\mathrm{BF}}(\kappa,b;\mu,\epsilon)>0.
\]

The eigenvalues of $S$ are a purely imaginary pair
\begin{equation}\label{lambda0-infinite}
\lambda_0^\pm(\mu,\epsilon)
=
\im c_{\kappa,b}\mu
+\im\,\mathcal{O}(\mu\epsilon^2,\mu^2\epsilon,\mu^3)
\pm
\im\sqrt{\mu}\,(1+\mathcal{O}(\epsilon,\mu)).
\end{equation}
For $\epsilon=0$, the eigenvalues $\lambda_1^\pm(\mu,0)$ and $\lambda_0^\pm(\mu,0)$ agree with the expansions obtained from the unperturbed dispersion relation \eqref{eq:unperturbed-dispersion}.
\end{theorem}

\begin{corollary}[Benjamin--Feir instability criterion]\label{cor:BF-index-infinite}
Let $(\kappa,b)\in(\mathbb R_{\geq0}\times\mathbb R_{>0})\setminus\mathfrak R$.
For positive Floquet exponents, consider the sideband scaling
\[
    \mu=\epsilon\nu,\qquad 0<\nu=\mathcal O(1).
\]
Then the Benjamin--Feir discriminant satisfies
\begin{equation}\label{BFDF-sideband-corollary}
\Delta_{\mathrm{BF}}(\kappa,b;\epsilon\nu,\epsilon)
=
\epsilon^2 e_{22}(\kappa,b)
\Bigl(8e_{11}(\kappa,b)-e_{22}(\kappa,b)\nu^2\Bigr)
+\mathcal O(\epsilon^3),
\end{equation}
locally uniformly for bounded positive $\nu$. Define
\[
    \operatorname{Ind}_\infty(\kappa,b)
    :=8e_{11}(\kappa,b)e_{22}(\kappa,b).
\]
If $\operatorname{Ind}_\infty(\kappa,b)>0$, then for all sufficiently
small $\epsilon>0$ there is a nonempty interval
\begin{equation}\label{nu-star-corollary}
    0<\nu<\nu_*(\kappa,b;\epsilon),
    \qquad
    \nu_*(\kappa,b;\epsilon)
    =
    \frac{\sqrt{8e_{22}(\kappa,b)e_{11}(\kappa,b)}}
    {|e_{22}(\kappa,b)|}\,(1+\mathcal O(\epsilon)),
\end{equation}
on which the pair $\lambda_1^\pm(\epsilon\nu,\epsilon)$ has opposite
nonzero real parts. By the symmetry
\[
\sigma(\mathcal{L}_{-\mu,\epsilon})
=
\overline{\sigma(\mathcal{L}_{\mu,\epsilon})},
\]
the same conclusion holds for the corresponding negative Floquet
exponents. Hence the full unstable sideband interval is
\[
0<|\nu|<\nu_*(\kappa,b;\epsilon).
\]
If $\operatorname{Ind}_\infty(\kappa,b)<0$, then the four small
eigenvalues remain purely imaginary in the corresponding small-amplitude
sideband regime. On the transition set
$\operatorname{Ind}_\infty(\kappa,b)=0$, the leading quadratic criterion
degenerates and higher-order terms in the discriminant are required.
\end{corollary}

\begin{corollary}[Figure-eight geometry]
\label{cor:figure-eight}
Assume that
\[
\operatorname{Ind}_{\infty}(\kappa,b)>0,
\qquad
\breve c_{\kappa,b}\neq 0.
\]
Then, for all sufficiently small $\epsilon>0$, there exists
\[
\mu_*(\epsilon)
=
\epsilon\frac{\sqrt{\operatorname{Ind}_{\infty}(\kappa,b)}}
{|e_{22}(\kappa,b)|}
\big(1+O(\epsilon)\big)
\]
such that the curves
\[
\Gamma_\epsilon^+
:=
\bigcup_{\sigma=\pm}
\left\{
\lambda_1^\sigma(\mu,\epsilon):
0\leq \mu\leq\mu_*(\epsilon)
\right\},
\qquad
\Gamma_\epsilon^-:=\overline{\Gamma_\epsilon^+},
\]
are simple closed curves satisfying
\[
\Gamma_\epsilon^+\cap\Gamma_\epsilon^-=\{0\}.
\]
They intersect transversely at the origin. Hence
$\Gamma_\epsilon^+\cup\Gamma_\epsilon^-$ is a local figure-eight curve (see Figure~\ref{fig:infinite-bf-figure-eight}).
\end{corollary}
\begin{proof}
Set $\mu=\epsilon\nu$. By Corollary~\ref{cor:BF-index-infinite},
the discriminant has a unique positive zero
\[
\nu_*(\epsilon)
=
\frac{\sqrt{\operatorname{Ind}_{\infty}(\kappa,b)}}
{|e_{22}(\kappa,b)|}
+O(\epsilon),
\]
and is positive for $0<\nu<\nu_*(\epsilon)$. Write
\[
\lambda_1^\pm(\mu,\epsilon)
=
\im\alpha(\mu,\epsilon)\pm\beta(\mu,\epsilon),
\]
where
\[
\alpha(\mu,\epsilon)
=
\frac12\breve c_{\kappa,b}\mu
+O(\mu\epsilon^2,\mu^2\epsilon,\mu^3),
\qquad
\beta(\mu,\epsilon)
=
\frac{\mu}{8}\sqrt{\Delta_{\mathrm BF}(\kappa,b;\mu,\epsilon)}.
\]
Thus $\beta>0$ in $(0,\mu_*(\epsilon))$ and vanishes at the
endpoints. Moreover,
\[
\partial_\mu\alpha
=
\frac12\breve c_{\kappa,b}+O(\epsilon^2)\neq0,
\]
so $\alpha$ is strictly monotone. Therefore the two positive-Floquet
branches form a simple closed curve lying, except for the origin, in
one open half-plane. The negative-Floquet branches form its complex
conjugate, and the two loops meet only at the origin. Finally,
\[
\partial_\mu\lambda_1^\pm(0,\epsilon)
=
\frac{\im}{2}\breve c_{\kappa,b}
\pm
\frac{\epsilon}{8}
\sqrt{\operatorname{Ind}_{\infty}(\kappa,b)}
+O(\epsilon^2),
\]
which gives two distinct tangent lines at the origin.
\end{proof}

\begin{remark}
At $\epsilon=0$, and away from the possible degeneracy $e_{22}=0$, \eqref{lambda1-infinite} gives
\[
\lambda_1^\pm(\mu,0)
=
\im\,\frac{1}{2}\breve{c}_{\kappa,b}\mu
\pm
\im\,\frac{|e_{22}|}{8}\,\mu^2
+\mathcal{O}(\mu^3),
\]
in agreement with the Krein splitting of the eigenvalue branches emanating from the fourfold zero eigenvalue of
$\mathcal{L}_{0,0}$.
If $e_{22}=0$, the displayed $\mu^2$ splitting vanishes and the next-order terms in the unperturbed expansion must be kept.
\end{remark}

\subsection{Exact non-resonant phase diagram}\label{subsec:phase-diagram}

The coefficients entering the Benjamin--Feir index are
\begin{equation}\label{eq:e11-e22-phase}
\begin{aligned}
e_{11}(\kappa,b)
&=
-\frac{P_{11}(\kappa,b)}
{8\sqrt{1+\kappa+b}\,R_2(\kappa,b)},
\qquad
e_{22}(\kappa,b)
&=
-\frac{P_{22}(\kappa,b)}
{(1+\kappa+b)^{3/2}},
\end{aligned}
\end{equation}
where
\begin{equation}\label{eq:P11-P22-R2}
\begin{aligned}
P_{11}(\kappa,b)
&:=
88b^2+(6\kappa-54)b+2\kappa^2+\kappa+8,\\
P_{22}(\kappa,b)
&:=
15b^2+(22\kappa+30)b+3\kappa^2+6\kappa-1,\\
R_2(\kappa,b)
&:=
14b+2\kappa-1.
\end{aligned}
\end{equation}
Consequently,
\begin{equation}\label{eq:sign-index-factorization}
\operatorname{sgn}\operatorname{Ind}_\infty(\kappa,b)
=
\operatorname{sgn}
\bigl(
P_{11}(\kappa,b)P_{22}(\kappa,b)R_2(\kappa,b)
\bigr)
\end{equation}
for every admissible parameter with
\(e_{11}(\kappa,b)e_{22}(\kappa,b)\neq0\).

Set
\[
\kappa_{22}
:=
-1+\frac{2}{\sqrt3},
\qquad
\kappa_{11}
:=
\frac{-125+30\sqrt{22}}{167},
\]
and define
\begin{align}
b_{22}(\kappa)
&:=
-1-\frac{11}{15}\kappa
+\frac{2}{15}\sqrt{19\kappa^2+60\kappa+60},
&&
0\leq\kappa\leq\kappa_{22},
\label{eq:b22-curve}\\
b_{\pm}(\kappa)
&:=
\frac{
27-3\kappa
\pm\sqrt{25-250\kappa-167\kappa^2}
}{88},
&&
0\leq\kappa\leq\kappa_{11},
\label{eq:bpm-curves}\\
b_R(\kappa)
&:=
\frac{1-2\kappa}{14},
&&
0\leq\kappa<\frac12.
\label{eq:bR-curve}
\end{align}

\begin{proposition}[Exact leading-order phase diagram]
\label{prop:exact-phase-diagram}
In the admissible set \(\mathcal P_\infty\), away from the regular
transition curves \(P_{11}=0\) and \(P_{22}=0\), the modulationally
stable set is the disjoint union
\begin{equation}\label{eq:stable-set-exact}
\begin{aligned}
\mathcal S_{\mathrm{low}}
&:=
\Bigl\{
0\leq\kappa<\kappa_{22},
\quad
b_{22}(\kappa)<b<b_R(\kappa)
\Bigr\}\\
&\hspace{2.5em}
\cup
\Bigl\{
\kappa_{22}\leq\kappa<\tfrac12,
\quad
0<b<b_R(\kappa)
\Bigr\},\\
\mathcal S_{\mathrm{island}}
&:=
\Bigl\{
0\leq\kappa<\kappa_{11},
\quad
b_-(\kappa)<b<b_+(\kappa)
\Bigr\},
\end{aligned}
\end{equation}
intersected with \(\mathcal P_\infty\).
On these sets
\(\operatorname{Ind}_\infty(\kappa,b)<0\), and the four small Bloch
eigenvalues remain purely imaginary in the sideband regime.
At every other non-resonant parameter away from
\(P_{11}P_{22}=0\), one has
\(\operatorname{Ind}_\infty(\kappa,b)>0\) and a nonempty interval of
unstable sideband exponents.

Moreover, in the physical quadrant, \(P_{22}=0\) is the graph
\(b=b_{22}(\kappa)\) for \(0\leq\kappa<\kappa_{22}\), while
\(P_{11}=0\) consists of the two graphs \(b=b_\pm(\kappa)\) for
\(0\leq\kappa<\kappa_{11}\). Their relative positions are
\begin{equation}\label{eq:curve-ordering}
0
<
b_{22}(\kappa)
<
b_R(\kappa)
<
b_-(\kappa)
<
b_+(\kappa),
\qquad
0\leq\kappa<\kappa_{11}.
\end{equation}
At \(\kappa=\kappa_{11}\), the two roots of \(P_{11}\) coalesce:
\(b_-(\kappa_{11})=b_+(\kappa_{11})\).
\end{proposition}

\begin{proof}
The formulas \eqref{eq:b22-curve} and \eqref{eq:bpm-curves} follow
from the quadratic formula applied to \(P_{22}\) and \(P_{11}\),
respectively, viewed as polynomials in \(b\).
The positive root of \(P_{22}(\kappa,\cdot)\) exists precisely for
\(0\leq\kappa<\kappa_{22}\).
Since \(P_{22}(\kappa,0)<0\) in this interval and the leading
coefficient is positive, one has
\[
P_{22}(\kappa,b)<0
\quad\text{for}\quad
0<b<b_{22}(\kappa),
\qquad
P_{22}(\kappa,b)>0
\quad\text{for}\quad
b>b_{22}(\kappa).
\]
For \(\kappa\geq\kappa_{22}\), the polynomial
\(P_{22}(\kappa,b)\) is positive for every \(b>0\).

Likewise, \(P_{11}(\kappa,\cdot)\) has two distinct positive roots
precisely for \(0\leq\kappa<\kappa_{11}\), and
\[
P_{11}(\kappa,b)<0
\quad\Longleftrightarrow\quad
b_-(\kappa)<b<b_+(\kappa).
\]
At \(\kappa=\kappa_{11}\) the two roots coincide, while for
\(\kappa>\kappa_{11}\) one has
\(P_{11}(\kappa,b)>0\) for all \(b>0\).

On the second-harmonic resonance line \(b=b_R(\kappa)\), one
computes
\[
P_{11}\bigl(\kappa,b_R(\kappa)\bigr)
=
\frac{9(4\kappa+5)^2}{49}
>0,
\qquad
P_{22}\bigl(\kappa,b_R(\kappa)\bigr)
=
\frac{32\kappa^2+584\kappa+239}{196}
>0.
\]
At \(\kappa=0\),
\[
0
<
b_{22}(0)
<
b_R(0)
<
b_-(0)
=
\frac14
<
b_+(0)
=
\frac4{11}.
\]
The curves involved are continuous on their respective domains.
Since neither \(P_{22}(\kappa,b_R(\kappa))\) nor
\(P_{11}(\kappa,b_R(\kappa))\) vanishes, the roots of \(P_{22}\) and
\(P_{11}\) cannot cross the resonance graph \(b=b_R(\kappa)\).
Moreover, \(b_-(\kappa)<b_+(\kappa)\) for
\(0\leq\kappa<\kappa_{11}\), with equality only at
\(\kappa=\kappa_{11}\).
This proves \eqref{eq:curve-ordering}.

Finally, \(R_2(\kappa,b)<0\) below \(b=b_R(\kappa)\) and
\(R_2(\kappa,b)>0\) above it.
Combining this observation with the sign information for
\(P_{11}\) and \(P_{22}\), and using
\eqref{eq:sign-index-factorization}, yields precisely the two stable
components in \eqref{eq:stable-set-exact}; the index is positive at
all remaining non-resonant parameters away from
\(P_{11}P_{22}=0\).
\end{proof}

\section{Perturbative Approach to the Separated Eigenvalues}

In this section, we analyze the splitting of the eigenvalues of $\mathcal{L}_{\mu, \epsilon}$ close to 0 for small values of $\mu$ and $\epsilon$, using Kato’s similarity transformation theory \cite[I-§4-6, II-§4]{Kato1966} and \cite{BMV1, BMV3, BMV_ed}. To this end, it is convenient to rewrite the operator $\mathcal{L}_{\mu, \epsilon}$ in \eqref{mathcal L mu e} as
\begin{equation} \label{mathcal L= ichmu+mathscr L}
\mathcal{L}_{\mu, \epsilon} = \im  c_{\kappa,b} \mu + \mathscr{L}_{\mu, \epsilon}, \quad \mu > 0,
\end{equation}
where, using also \eqref{D+mu}, $\mathscr{L}_{\mu, \epsilon}$ is the Hamiltonian operator
\begin{equation} \label{mathscr L mu e}
\mathscr{L}_{\mu, \epsilon} = \mathcal{J} \mathcal{B}_{\mu, \epsilon},
\end{equation}
with $\mathcal{B}_{\mu, \epsilon}$ the self-adjoint operator
\begin{equation} \label{mathcal B mu e}
\mathcal{B}_{\mu, \epsilon} := \begin{bmatrix} 1 + a_{\epsilon}(x)-\kappa\,\tau_{\mu,\e}+b\beta_{\mu,\e} & -( c_{\kappa,b} + p_{\epsilon}(x)) \partial_x - \im \mu p_{\epsilon}(x) \\ \partial_x \circ( c_{\kappa,b} + p_{\epsilon}(x)) + \im \mu p_{\epsilon}(x) & |D + \mu| \end{bmatrix} \ ,
\quad \beta_{\mu,\e} \mbox{ in } \eqref{Sigma} \ . 
\end{equation}
In addition  $\mathscr{L}_{\mu, \epsilon}$   is also complex-reversible, namely it satisfies, by \eqref{reversible},
\begin{equation} \label{mathscr rho=-rho mathscr}
\mathscr{L}_{\mu, \epsilon} \circ \bar{\rho} = - \bar{\rho} \circ \mathscr{L}_{\mu, \epsilon},
\end{equation}
whereas $\mathcal{B}_{\mu, \epsilon}$ is reversibility-preserving, i.e. fulfills \eqref{B rho=rho B}. Note also that $\mathcal{B}_{0, \epsilon}$ is a real operator.

The scalar operator $\im  c_{\kappa,b} \mu \equiv \im  c_{\kappa,b} \mu \,\text{Id}$ just translates the spectrum of $\mathcal{L}_{\mu, \epsilon}$ along the imaginary axis of the quantity $\im  c_{\kappa,b} \mu$, that is, in view of \eqref{mathcal L= ichmu+mathscr L},
\begin{equation}
\sigma(\mathcal{L}_{\mu, \epsilon}) = \im  c_{\kappa,b} \mu + \sigma(\mathscr{L}_{\mu, \epsilon}).
\end{equation}
Thus in the sequel we focus on studying the spectrum of $\mathscr{L}_{\mu, \epsilon}$.

Note also that
$\mathscr L_{0,\epsilon}=\mathcal L_{0,\epsilon}$
for all sufficiently small $|\epsilon|$. At $\epsilon=0$, the
eigenvalue zero has algebraic multiplicity four and geometric
multiplicity three, with generalized kernel
\[
\operatorname{span}
\{f_1^+,f_1^-,f_0^+,f_0^-\}.
\]
Moreover, the spectrum is separated as in
\eqref{eq:sigma-decomposition-L00}.

For sufficiently small $|\epsilon|$,
Proposition~\ref{prop:periodic-generalized-kernel} identifies
$\mathcal V_{0,\epsilon}$ with the four-dimensional generalized
kernel of $\mathscr L_{0,\epsilon}$ and gives
\[
\left(
\mathscr L_{0,\epsilon}
\big|_{\mathcal V_{0,\epsilon}}
\right)^2=0.
\]
If, in addition,
\[
\partial_\epsilon c_\epsilon\neq0,
\]
then
\[
\ker\left(
\mathscr L_{0,\epsilon}
\big|_{\mathcal V_{0,\epsilon}}
\right)
=
\operatorname{span}_{\mathbb C}
\{U_1(\epsilon),U_2(\epsilon)\},
\]
and hence the eigenvalue zero has geometric multiplicity two.
In particular, since
\[
\partial_\epsilon c_\epsilon
=
2c_2\epsilon+\mathcal O(\epsilon^2),
\]
this conclusion holds for all sufficiently small
$\epsilon\neq0$ whenever
$e_{11}(\kappa,b)=2c_2\neq0$.

We remark that, in view of \eqref{D+mu} and the definitions of
$\tau_{\mu,\epsilon}$ and $\beta_{\mu,\epsilon}$, the operator
$\mathscr{L}_{\mu,\epsilon}$ depends analytically on $\mu$.  The operator $\mathscr{L}_{\mu, \epsilon}: Y \subset X \to X$ has domain $Y := H^4(\mathbb{T},\mathbb{C})\times H^1(\mathbb{T},\mathbb{C})$ and range $X := L^2(\mathbb{T},\mathbb{C})\times L^2(\mathbb{T},\mathbb{C})$.

\begin{lemma} \label{kato thm}
Let $\Gamma$ be a closed, counterclockwise-oriented curve around $0$ in the complex plane separating $\sigma' (\mathscr{L}_{0,0}) = \{0\}$ and the other part of the spectrum $\sigma'' (\mathscr{L}_{0,0})$ in \eqref{eq:sigma-decomposition-L00}. There exist $\mu_0,\,\epsilon_0>0$ such that for any $(\mu, \epsilon) \in B(\mu_0) \times B(\epsilon_0)$ the following statements hold:

\begin{enumerate}
    \item \textit{The curve $\Gamma$ belongs to the resolvent set of the operator $\mathscr{L}_{\mu,\epsilon} : Y \subset X \to X$ defined in \eqref{mathscr L mu e}.}
    \item \textit{The operators}
    \begin{equation} \label{Projection P mu e}
        P_{\mu,\epsilon} := - \frac{1}{2\pi \im} \oint_\Gamma (\mathscr{L}_{\mu,\epsilon} - \lambda)^{-1} \ d\lambda : X \to Y
    \end{equation}
    \textit{are well-defined projectors commuting with $\mathscr{L}_{\mu,\epsilon}$, i.e., $P_{\mu,\epsilon}^2 = P_{\mu,\epsilon}$ and $P_{\mu,\epsilon} \mathscr{L}_{\mu,\epsilon} = \mathscr{L}_{\mu,\epsilon} P_{\mu,\epsilon}$. The map $(\mu, \epsilon) \mapsto P_{\mu,\epsilon}$ is analytic from $B(\mu_0) \times B(\epsilon_0)$ to $\mathcal{L}(X,Y)$.}
    \item \textit{The domain $Y$ of the operator $\mathscr{L}_{\mu,\epsilon}$ decomposes as the direct sum}
    \begin{equation} \label{Y=V+ker P}
        Y = \mathcal{V}_{\mu,\epsilon} \oplus \ker(P_{\mu,\epsilon}), \quad \mathcal{V}_{\mu,\epsilon} := \operatorname{Rg}(P_{\mu,\epsilon}) = \ker(\operatorname{Id} - P_{\mu,\epsilon}),
    \end{equation}
    \textit{of closed invariant subspaces, namely $\mathscr{L}_{\mu,\epsilon} : \mathcal{V}_{\mu,\epsilon} \to \mathcal{V}_{\mu,\epsilon}$, $\mathscr{L}_{\mu,\epsilon} : \ker(P_{\mu,\epsilon}) \to \ker(P_{\mu,\epsilon})$. Moreover}
    \begin{equation} \label{spectrum separated by Gamma}
    \begin{aligned}
        \sigma(\mathscr{L}_{\mu,\epsilon}) \cap \{z \in \mathbb{C} \text{ inside } \Gamma\} &= \sigma(\mathscr{L}_{\mu,\epsilon} |_{\mathcal{V}_{\mu,\epsilon}}) = \sigma'(\mathscr{L}_{\mu,\epsilon}), \\
        \sigma(\mathscr{L}_{\mu,\epsilon}) \cap \{z \in \mathbb{C} \text{ outside } \Gamma\} &= \sigma(\mathscr{L}_{\mu,\epsilon} |_{\ker(P_{\mu,\epsilon})}) = \sigma''(\mathscr{L}_{\mu,\epsilon}).
    \end{aligned}
    \end{equation}
    \item \textit{The projectors $P_{\mu,\epsilon}$ are similar to each other; the transformation operators}
    \begin{equation} \label{U transformation operators}
        U_{\mu,\epsilon} := (\operatorname{Id} - (P_{\mu,\epsilon} - P_{0,0})^2)^{-1/2} \big[P_{\mu,\epsilon} P_{0,0} + (\operatorname{Id} - P_{\mu,\epsilon})(\operatorname{Id} - P_{0,0}) \big]
    \end{equation}
    \textit{are bounded and invertible in $Y$ and in $X$, with inverse}
    \begin{equation} \label{U inverse}
        U_{\mu,\epsilon}^{-1} = \big[P_{0,0} P_{\mu,\epsilon} + (\operatorname{Id} - P_{0,0})(\operatorname{Id} - P_{\mu,\epsilon})\big](\operatorname{Id} - (P_{\mu,\epsilon} - P_{0,0})^2)^{-1/2},
    \end{equation}
    \textit{and $U_{\mu,\epsilon} P_{0,0} U_{\mu,\epsilon}^{-1} = P_{\mu,\epsilon}$ as well as $U_{\mu,\epsilon}^{-1} P_{\mu,\epsilon} U_{\mu,\epsilon} = P_{0,0}$. The map $(\mu, \epsilon) \mapsto U_{\mu,\epsilon}$ is analytic from $B(\mu_0) \times B(\epsilon_0)$ to $\mathcal{L}(Y)$.}
    \item \textit{The subspaces $\mathcal{V}_{\mu,\epsilon} = \operatorname{Rg}(P_{\mu,\epsilon})$ are isomorphic to each other: $\mathcal{V}_{\mu,\epsilon} = U_{\mu,\epsilon} \mathcal{V}_{0,0}$. In particular $\dim \mathcal{V}_{\mu,\epsilon} = \dim \mathcal{V}_{0,0} = 4$, for any $(\mu, \epsilon) \in B(\mu_0) \times B(\epsilon_0)$.}
\end{enumerate}
\end{lemma}

The proof of Lemma \ref{kato thm} is similar to the one of \cite[Lemma 3.1]{BMV1} and we skip it.  Recalling \eqref{mathscr L mu e}-\eqref{mathscr rho=-rho mathscr}, the Hamiltonian and reversible nature of the operator $\mathscr{L}_{\mu,\e}$ imply additional algebraic properties for spectral projectors $P_{\mu,\e}$ and the transformation operators $U_{\mu,\e}$ as follows.  

\begin{lemma} \label{properties of U and P}
For any $(\mu, \epsilon) \in B(\mu_0) \times B(\epsilon_0)$, the following holds true:

\begin{itemize}
    \item[(i)] The projectors $P_{\mu,\epsilon}$ defined in \eqref{Projection P mu e} are skew-Hamiltonian, namely $\mathcal{J} P_{\mu,\epsilon} = P_{\mu,\epsilon}^* \mathcal{J}$, and reversibility preserving, i.e. $\bar{\rho} P_{\mu,\epsilon} = P_{\mu,\epsilon} \bar{\rho}$.
    
    \item[(ii)] The transformation operators $U_{\mu,\epsilon}$ in \eqref{U transformation operators} are symplectic, namely $U_{\mu,\epsilon}^* \mathcal{J} U_{\mu,\epsilon} = \mathcal{J}$, and reversibility preserving.
    
    \item[(iii)] $P_{0,\epsilon}$ and $U_{0,\epsilon}$ are real operators, i.e. $\bar{P}_{0,\epsilon} = P_{0,\epsilon}$ and $\bar{U}_{0,\epsilon} = U_{0,\epsilon}$.
\end{itemize}
    
\end{lemma} 

See \cite[Lemma 3.2]{BMV1} for details. By the previous lemma, the linear involution $\bar{\rho}$ commutes with the spectral projectors $P_{\mu,\epsilon}$ and then $\bar{\rho}$ leaves invariant the subspace $\mathcal{V}_{\mu,\epsilon} = \mathrm{Rg}(P_{\mu,\epsilon})$.

\color{black}
We next identify the generalized kernel of the periodic fiber at
$\mu=0$.  The construction is the analogue of
\cite[Theorem~4.1]{NS}, with the capillary and bending terms included in
the stationary hydroelastic equations.  The proof is given in
Appendix~\ref{app:periodic-generalized-kernel}.

For a scalar periodic function $f$, we write
\[
\langle f\rangle_x
:=
\frac{1}{2\pi}\int_{\mathbb T}f(x)\,\mathrm dx.
\]
Let $\mathfrak P_\epsilon$ be the composition operator introduced in
\eqref{Gneta}, and define the linear change of variables
\begin{equation}\label{eq:T-epsilon-generalized-kernel}
\mathscr T_\epsilon
\begin{bmatrix}h_1\\ h_2\end{bmatrix}
:=
\begin{bmatrix}
(1+\mathfrak p_{\epsilon,x})\,
\mathfrak P_\epsilon h_1
\\[1mm]
\mathfrak P_\epsilon\bigl(h_2-B_\epsilon h_1\bigr)
\end{bmatrix}.
\end{equation}
Thus $\mathscr T_\epsilon=\mathcal P_\epsilon Z_\epsilon^{-1}$ is the
composition of the good-unknown and flattening transformations used in
\eqref{alinhac} and \eqref{LC}.

\begin{proposition}[Generalized kernel of the periodic fiber]
\label{prop:periodic-generalized-kernel}
Let $(\kappa,b)\in\mathcal P_\infty$.  For all sufficiently small
$|\epsilon|$, define
\begin{equation}\label{eq:U1-U2tilde-def}
U_1(\epsilon)
:=
\begin{bmatrix}0\\1\end{bmatrix},
\qquad
\widetilde U_2(\epsilon)
:=
\mathscr T_\epsilon
\begin{bmatrix}
\partial_x\eta_\epsilon\\
\partial_x\psi_\epsilon
\end{bmatrix},
\end{equation}
and
\begin{equation}\label{eq:U3-U4-def}
U_3(\epsilon)
:=
\mathscr T_\epsilon
\begin{bmatrix}
\partial_a\eta_a|_{a=\epsilon}\\
\partial_a\psi_a|_{a=\epsilon}
\end{bmatrix},
\qquad
U_4(\epsilon)
:=
\mathscr T_\epsilon
\begin{bmatrix}1\\0\end{bmatrix},
\end{equation}
where $a\mapsto(\eta_a,\psi_a,c_a)$ is the Stokes branch of
Theorem~\ref{Thm: Stokes expansion}.

For $\epsilon\neq0$, set
\begin{equation}\label{eq:normalized-U2}
U_2(\epsilon)
:=
\frac{1}{\epsilon}\widetilde U_2(\epsilon)
-
\left\langle
\frac{1}{\epsilon}
\bigl(\widetilde U_2(\epsilon)\bigr)_2
\right\rangle_x U_1(\epsilon).
\end{equation}
Then $U_2$ extends real-analytically to $\epsilon=0$, and
\begin{equation}\label{eq:U2-leading-value}
U_2(0)
=
\begin{bmatrix}
-\sin (x)\\ c_{\kappa,b}\cos (x)
\end{bmatrix}
=
c_{\kappa,b}^{1/2}f_1^-.
\end{equation}
Moreover,
\begin{equation}\label{eq:generalized-kernel-relations}
\begin{aligned}
\mathscr L_{0,\epsilon}U_1&=0,
&\qquad
\mathscr L_{0,\epsilon}\widetilde U_2&=0,
&\qquad
\mathscr L_{0,\epsilon}U_2&=0,
\\
\mathscr L_{0,\epsilon}U_3
&=-\partial_\epsilon c_\epsilon\,\widetilde U_2,
&\qquad
\mathscr L_{0,\epsilon}U_4&=-U_1.
\end{aligned}
\end{equation}
The vectors $U_j(\epsilon)$ are real and satisfy
\begin{equation}\label{eq:Uj-reversibility}
\overline\rho U_j=-U_j,
\quad j=1,2,
\qquad
\overline\rho U_j=U_j,
\quad j=3,4,
\end{equation}
as well as
\begin{equation}\label{eq:U2-zero-average}
\left\langle\bigl(U_2(\epsilon)\bigr)_2\right\rangle_x=0.
\end{equation}
Finally,
\begin{equation}\label{eq:V-generalized-kernel-span}
\mathcal V_{0,\epsilon}
=
\operatorname{span}_{\mathbb C}
\{U_1(\epsilon),U_2(\epsilon),U_3(\epsilon),U_4(\epsilon)\},
\end{equation}
and hence
\begin{equation}\label{eq:L-square-zero-on-V}
\left(
\mathscr L_{0,\epsilon}\big|_{\mathcal V_{0,\epsilon}}
\right)^2=0.
\end{equation}
In particular, zero is the only eigenvalue of
$\mathscr L_{0,\epsilon}|_{\mathcal V_{0,\epsilon}}$ and has algebraic
multiplicity four. If
$\partial_\epsilon c_\epsilon\neq0$, then
\[
\ker\left(
\mathscr L_{0,\epsilon}
\big|_{\mathcal V_{0,\epsilon}}
\right)
=
\operatorname{span}_{\mathbb C}
\{U_1(\epsilon),U_2(\epsilon)\},
\]
so that zero has geometric multiplicity two.
\end{proposition}

\color{black}

\paragraph{Symplectic and reversible basis of $\mathcal{V}_{\mu,\epsilon}$.} It is convenient to represent the Hamiltonian and reversible operator $\mathscr{L}_{\mu,\epsilon}: \mathcal{V}_{\mu,\epsilon} \to \mathcal{V}_{\mu,\epsilon}$ in a basis which is symplectic and reversible, according to the following definition:

\begin{definition}[Symplectic and reversible basis]  \label{Symplectic and reversible basis}
    A basis $\mathsf{F} := \{ \mathsf{f}_1^+, \mathsf{f}_1^-, \mathsf{f}_0^+, \mathsf{f}_0^- \}$ of $\mathcal{V}_{\mu,\epsilon}$ is \textit{symplectic} if, for any $k, k' = 0,1$,
\begin{equation} \label{basis is symplectic}
    \begin{aligned}
    (\mathcal{J} \mathsf{f}_k^\mp, \mathsf{f}_k^\pm) = \pm 1, ~~(\mathcal{J} \mathsf{f}_k^\sigma, \mathsf{f}_k^\sigma) = 0, \quad \forall \,\sigma = \pm; \\
    \text{if } k \neq k', \text{ then } (\mathcal{J} \mathsf{f}_k^\sigma, \mathsf{f}_{k'}^{\sigma'}) = 0, \quad \forall\, \sigma, \sigma' = \pm.
\end{aligned}
\end{equation}   

This is \textit{reversible} if
\begin{equation} \label{basis is reversible}
    \begin{aligned}
    \bar{\rho} \mathsf{f}_1^+ &= \mathsf{f}_1^+,  ~\bar{\rho} \mathsf{f}_1^- = -\mathsf{f}_1^-,~\bar{\rho} \mathsf{f}_0^+ = \mathsf{f}_0^+,  ~\bar{\rho} \mathsf{f}_0^- = -\mathsf{f}_0^-, \\
    &\text{i.e. } \bar{\rho} \mathsf{f}_k^\sigma = \sigma \mathsf{f}_k^\sigma, \quad \forall\, \sigma = \pm, k = 0,1.
\end{aligned}
\end{equation}
\end{definition} 

We use the following notation along the paper: we denote by $\mathrm{even}(x)$ a real $2\pi$-periodic function which is even in $x$, and by $\mathrm{odd}(x)$ a real $2\pi$-periodic function which is odd in $x$.

\begin{remark}[Parity structure of the reversible basis \eqref{basis is reversible}]
The elements of a reversible basis \( \mathsf{F}=\{\mathsf{f}^+_1,\mathsf{f}^-_1,\mathsf{f}^+_0,\mathsf{f}^-_0\} \) enjoys specific parity properties. Specifically,
\begin{align} \label{Parity structure}
\mathsf{f}_k^+(x) = 
\begin{bmatrix}
\mathit{even}(x) + \im\,\mathit{odd}(x) \\
\mathit{odd}(x) + \im\,\mathit{even}(x)
\end{bmatrix}, \quad
\mathsf{f}_k^-(x) = 
\begin{bmatrix}
\mathit{odd}(x) + \im\,\mathit{even}(x) \\
\mathit{even}(x) + \im\,\mathit{odd}(x)
\end{bmatrix}.
\end{align}
This structure follows from the reversibility of the problem, specifically from the involution \( \overline{\rho} \) defined in equation \eqref{reversible}, which implies that the real and imaginary parts of each component satisfy definite parity conditions.
\end{remark}

\begin{remark}[Symplectic expansion using the basis in \eqref{basis is symplectic}]
We can express any vector \( \mathsf{f} \in \mathcal{V}_{\mu,\epsilon} \) as a linear combination of the symplectic basis:
\begin{align} \label{remark f expansion}
    \mathsf{f} = \alpha_1^+ \mathsf{f}_1^+ + \alpha_1^- \mathsf{f}_1^- + \alpha_0^+ \mathsf{f}_0^+ + \alpha_0^- \mathsf{f}_0^-,
\end{align}
for suitable coefficients \( \alpha_k^\sigma \in \mathbb{C} \). These coefficients are computed by applying the symplectic form \( \mathcal{J} \), taking \( L^2 \)-scalar products with the basis elements, and using the symplecticity \eqref{basis is symplectic}. Therefore, we may rewrite \eqref{remark f expansion} as
\begin{align} \label{expasion of f by using symplectic basis}
    \mathsf{f}=-(\mathcal{J}\mathsf{f},\mathsf{f}^-_1)\mathsf{f}^+_1+(\mathcal{J}\mathsf{f},\mathsf{f}^+_1)\mathsf{f}^-_1-(\mathcal{J}\mathsf{f},\mathsf{f}^-_0)\mathsf{f}^+_0+(\mathcal{J}\mathsf{f},\mathsf{f}^+_0)\mathsf{f}^-_0.
\end{align}
\end{remark}

\section[Matrix Representation of the operator on the finite-dimensional space]{Matrix Representation of $\mathscr{L}_{\mu,\e}$ on $\mathcal{V}_{\mu,\e}$}
Using the transformation operator $U_{\mu,\e}$ in \eqref{U transformation operators}, we construct the basis of $\mathcal{V}_{\mu,\e}$
\begin{equation} \label{F basis set and f}
    \begin{aligned}
        \mathcal{F}&:=\{f^+_1(\mu,\e),f^-_1(\mu,\e),f^+_0(\mu,\e),f^-_0(\mu,\e)\}, \qquad 
        f^\sigma_k(\mu,\e):=U_{\mu,\e} f^{\sigma}_k,~~\sigma=\pm,~k=0,1,
    \end{aligned}
\end{equation}
where 
\begin{align} \label{eigenfunc of mathcall L00 2}
    f^+_1:=\vet{(1+b+\kappa)^{-1/4}\cos(x)}{(1+b+\kappa)^{1/4}\sin(x)}, ~~f^-_1:=\vet{-(1+b+\kappa)^{-1/4}\sin(x)}{(1+b+\kappa)^{1/4}\cos(x)}, ~~f^+_0=\vet{1}{0},~~f^-_0=\vet{0}{1},
\end{align}
form a basis of $\mathcal{V}_{0,0}=\mathrm{Rg}(P_{0,0})$, cf. \eqref{eigenfunc of mathcall L00 2}. Note that the real valued vectors $\{f^\pm_1,f^\pm_0\}$ form a symplectic and reversible basis for $\mathcal{V}_{0,0}$, according to Definition \ref{Symplectic and reversible basis}. Then, by Lemma \ref{kato thm} and Lemma \ref{properties of U and P} we deduce that:
\begin{lemma} \label{F is symplectic and reversible}
    The basis $\mathcal{F}$ of $\mathcal{V}_{\mu,\e}$ defined in \eqref{F basis set and f}, is symplectic and reversible, i.e. satisfies \eqref{basis is symplectic} and \eqref{basis is reversible}. Each map $(\mu,\e) \mapsto f^\sigma_k(\mu,\e)$ is analytic as a map $B(\mu_0)\times B(\mu_0)\rightarrow H^4(\mathbb{T})\times H^1(\mathbb{T})$.
\end{lemma}
\begin{proof}
By Lemma \ref{properties of U and P}-(ii), the operators $U_{\mu,\epsilon}$ are symplectic and reversibility preserving, i.e. 
$U_{\mu,\epsilon}^* \mathcal{J} U_{\mu,\epsilon} = \mathcal{J}$ and $U_{\mu,\epsilon}\circ\overline{\rho}=\overline{\rho}\circ U_{\mu,\epsilon}$. Recall \eqref{F basis set and f}, i.e. $f_k^\sigma(\mu,\epsilon):=U_{\mu,\epsilon} f_k^\sigma, \sigma=\pm, k=0,1.$
Then for any $k,k',\sigma,\sigma'$,
\begin{align*}
(\mathcal{J} f_k^\sigma(\mu,\epsilon), f_{k'}^{\sigma'}(\mu,\epsilon))
= (\mathcal{J} f_k^\sigma, f_{k'}^{\sigma'}),
\end{align*}
so the symplectic relations \eqref{basis is symplectic} are preserved. Moreover,
\begin{align*}
\overline{\rho} f_k^\sigma(\mu,\epsilon)
= U_{\mu,\epsilon}\circ \overline{\rho} f_k^\sigma
= \sigma f_k^\sigma(\mu,\epsilon),
\end{align*}
which shows that the reversibility conditions \eqref{basis is reversible} hold as well. 
Finally, the analyticity of $f_k^\sigma(\mu,\epsilon)$ follows from the analyticity of 
$U_{\mu,\epsilon}$ (Lemma \ref{kato thm}). 
\end{proof}

We then expand the vectors $f^\sigma_k(\mu,\e)$ in $(\mu,\e)$. We denote by $\mathrm{even}_0(x)$ a real, even, $2\pi$-periodic function with zero space average. 

Throughout this lemma, the symbols $\mathcal O(\cdot)$ denote remainders in
\[
Y=H^4(\mathbb T,\mathbb C)\times H^1(\mathbb T,\mathbb C)
\]
with the indicated parity. The terms $\mathcal O(\mu^2)$ and
$\mathcal O(\epsilon^2)$ denote pure remainders obtained by restricting to
$\epsilon=0$ and $\mu=0$, respectively. Thus
$\mathcal O(\mu^2)=\mu^2 r_\mu(\mu)$ and
$\mathcal O(\epsilon^2)=\epsilon^2 r_\epsilon(\epsilon)$, with analytic
remainders valued in the displayed parity subspaces. The notation
$\mathcal O(\mu^2\epsilon,\mu\epsilon^2)$ collects all remaining higher mixed
terms after the explicit $\mu\epsilon$ contribution, namely terms of the form
$\mu^2\epsilon r_{21}(\mu,\epsilon)+\mu\epsilon^2 r_{12}(\mu,\epsilon)$.

\begin{lemma}[Expansion of the basis $\mathcal{F}$] \label{lem:4.2}
For small values of $(\mu,\epsilon)$, the symplectic and reversible basis $\mathcal{F} = \{f_1^+(\mu,\epsilon), f_1^-(\mu,\epsilon), f_0^+(\mu,\epsilon), f_0^-(\mu,\epsilon)\}$ defined in \eqref{F basis set and f} has the following expansions:

\begin{equation} \label{expansion_f1+}
\begin{aligned}
    f_1^+(\mu,\epsilon)
    &=
    \vet{\delta_{\kappa,b}^{-1} \cos(x)}
        {\delta_{\kappa,b} \sin(x)}
    + \im \frac{\mu}{4} \gamma_{\kappa,b}
    \vet{\delta_{\kappa,b}^{-1} \sin(x)}
        {\delta_{\kappa,b} \cos(x)}
    + \epsilon
    \vet{\alpha_{\kappa,b} \cos(2x)}
        {\beta_{\kappa,b} \sin(2x)}
    \\
    &\quad
    + \mathcal{O}(\mu^2)
    \vet{\mathit{even}_0(x) + \im \mathit{odd}(x)}
        {\mathit{odd}(x) + \im \mathit{even}_0(x)}
    + \mathcal{O}(\epsilon^2)
    \vet{\mathit{even}_0(x)}
        {\mathit{odd}(x)}
    \\
    &\quad
    + \im \mu\epsilon
    \vet{\mathit{odd}(x)}
        {\mathit{even}(x)}
    + \mathcal{O}(\mu^2\epsilon,\mu\epsilon^2).
\end{aligned}
\end{equation}

\begin{equation} \label{expansion_f1-}
\begin{aligned}
    f_1^-(\mu,\epsilon)
    &=
    \vet{-\delta_{\kappa,b}^{-1} \sin(x)}
        {\delta_{\kappa,b} \cos(x)}
    + \im \frac{\mu}{4} \gamma_{\kappa,b}
    \vet{\delta_{\kappa,b}^{-1} \cos(x)}
        {-\delta_{\kappa,b} \sin(x)}
    + \epsilon
    \vet{-\alpha_{\kappa,b} \sin(2x)}
        {\beta_{\kappa,b} \cos(2x)}
    \\
    &\quad
    + \mathcal{O}(\mu^2)
    \vet{\mathit{odd}(x) + \im \mathit{even}_0(x)}
        {\mathit{even}_0(x) + \im \mathit{odd}(x)}
    + \mathcal{O}(\epsilon^2)
    \vet{\mathit{odd}(x)}
        {\mathit{even}(x)}
    \\
    &\quad
    + \im \mu\epsilon
    \vet{\mathit{even}(x)}
        {\mathit{odd}(x)}
    + \mathcal{O}(\mu^2\epsilon,\mu\epsilon^2).
\end{aligned}
\end{equation}

\begin{equation} \label{expansion_f0+}
\begin{aligned}
    f_0^+(\mu,\epsilon)
    &=
    \vet{1}{0}
    + \epsilon \delta_{\kappa,b}
    \vet{\delta_{\kappa,b}^{-1} \cos(x)}
        {-\delta_{\kappa,b} \sin(x)}
    + \mathcal{O}(\epsilon^2)
    \vet{\mathit{even}_0(x)}
        {\mathit{odd}(x)}
    \\
    &\quad
    +  \frac{\mu\epsilon}{4}
    \vet{\delta_{\kappa,b}^{-4} \cos(x)}
        {-\delta_{\kappa,b}^{-2} \sin(x)}
    + \im \mu\epsilon
    \vet{\mathit{odd}(x)}
        {\mathit{even}_0(x)}
    + \mathcal{O}(\mu^2\epsilon,\mu\epsilon^2).
\end{aligned}
\end{equation}

\begin{equation} \label{expansion_f0-}
\begin{aligned}
   f_0^-(\mu,\epsilon)
   &=
    \vet{0}{1}
    + \frac{\mu\epsilon}{2}
    \vet{\delta_{\kappa,b}^{-2}\sin x}
        {\cos x}
    + \im \mu\epsilon
    \vet{\mathit{even}_0(x)}
        {\mathit{odd}(x)}
    + \mathcal{O}(\mu^2\epsilon,\mu\epsilon^2).
\end{aligned}
\end{equation}

where the remainders $\mathcal{O}(\cdot)$ are vectors in $Y = H^4(\mathbb{T}, \mathbb{C}) \times H^1(\mathbb{T}, \mathbb{C})$, $c_{\kappa,b} = \sqrt{1+b+\kappa}$, and the hydroelastic expansion constants are given by:
\begin{equation} \label{hydro_constants}
    \alpha_{\kappa,b} = -\frac{2-13b-\kappa}{14b+2\kappa-1} c_{\kappa,b}^{-1/2}, \quad
    \beta_{\kappa,b} = -\frac{1+b+\kappa}{14b+2\kappa-1} c_{\kappa,b}^{1/2},\quad
    \gamma_{\kappa,b} = \frac{1-\kappa-3b}{1+b+\kappa}, \quad
    \delta_{\kappa,b} = c_{\kappa,b}^{1/2}.
\end{equation}
For $\mu=0$, the basis $\{f_k^\pm(0,\epsilon), \ k=0,1\}$ is purely real and satisfies:
\begin{equation} \label{basis_mu_0}
    f_1^+(0,\epsilon) = \vet{\mathit{even}_0(x)}{\mathit{odd}(x)}, \quad
    f_1^-(0,\epsilon) = \vet{\mathit{odd}(x)}{\mathit{even}(x)},
\end{equation}
\begin{equation}\label{eq:f_0_pm_0_epsilon}
    f_0^+(0,\epsilon) = \vet{1}{0} + \vet{\mathit{even}_0(x)}{\mathit{odd}(x)}, \quad
    f_0^-(0,\epsilon) = \vet{0}{1}.
\end{equation}
\end{lemma}
\begin{proof}
    The long computations are given in Appendix \ref{secA1}.
\end{proof}

\paragraph{Second basis of $\mathcal V_{\mu,\epsilon}$.}
We now construct from the basis $\mathcal F$ in \eqref{F basis set and f} another symplectic and reversible basis of
$\mathcal V_{\mu,\epsilon}$, with an additional property that will be crucial in the expansion of the matrix
representing $L_{\mu,\epsilon}$.

Recall that, by Lemma~\ref{lem:4.2},
for $\mu=0$ the basis $\{f_k^\sigma(0,\epsilon)\}_{k=0,1,\ \sigma=\pm}$ is real and satisfies
\[
f_k^+(0,\epsilon)=
\begin{bmatrix}
\mathit{even}(x)\\
\mathit{odd}(x)
\end{bmatrix},
\qquad
f_k^-(0,\epsilon)=
\begin{bmatrix}
\mathit{odd}(x)\\
\mathit{even}(x)
\end{bmatrix}.
\]
In particular, the second component of $f_1^-(0,\epsilon)$ is an even function whose space average
is not necessarily zero. Thus we introduce the new basis
\begin{equation}\label{4.8}
\mathcal G:=\mathcal{G}(\mu,\e)=\{g_1^+(\mu,\epsilon),g_1^-(\mu,\epsilon),g_0^+(\mu,\epsilon),g_0^-(\mu,\epsilon)\},
\end{equation}
defined by
\begin{equation}\label{4.9}
g_1^+(\mu,\epsilon):=f_1^+(\mu,\epsilon),\qquad
g_1^-(\mu,\epsilon):=f_1^-(\mu,\epsilon)-n(\mu,\epsilon)\,f_0^-(\mu,\epsilon),
\end{equation}
\begin{equation}\label{4.10}
g_0^+(\mu,\epsilon):=f_0^+(\mu,\epsilon)+n(\mu,\epsilon)\,f_1^+(\mu,\epsilon),\qquad
g_0^-(\mu,\epsilon):=f_0^-(\mu,\epsilon),
\end{equation}
where
\begin{equation}\label{4.11}
n(\mu,\epsilon):=
\frac{(f_1^-(\mu,\epsilon),f_0^-(\mu,\epsilon))}
{\|f_0^-(\mu,\epsilon)\|^2}.
\end{equation}
We also set
\[
\nu_{\kappa,b}:=\frac{b+\kappa}{4c_{\kappa,b}^{3/2}}.
\]

We first note that \(n(\mu,\epsilon)\) is real. Indeed, for the \(L^2\)
inner product and the anti-linear reverser \(\bar\rho\), one has
\[
(f,g)=\overline{(\bar\rho f,\bar\rho g)}.
\]
Since the basis \(\mathcal F\) is reversible, Lemma~\ref{F is symplectic and reversible}
gives
\[
\bar\rho f_1^-(\mu,\epsilon)=-f_1^-(\mu,\epsilon),
\qquad
\bar\rho f_0^-(\mu,\epsilon)=-f_0^-(\mu,\epsilon).
\]
Therefore
\[
\begin{aligned}
(f_1^-(\mu,\epsilon),f_0^-(\mu,\epsilon))=
\overline{
(\bar\rho f_1^-(\mu,\epsilon),
 \bar\rho f_0^-(\mu,\epsilon))
}=
\overline{
(-f_1^-(\mu,\epsilon),
 -f_0^-(\mu,\epsilon))
}=
\overline{
(f_1^-(\mu,\epsilon),f_0^-(\mu,\epsilon))
}.
\end{aligned}
\]
Hence \((f_1^-(\mu,\epsilon),f_0^-(\mu,\epsilon))\) is real. Since
\(\|f_0^-(\mu,\epsilon)\|^2>0\) is real, we conclude that
\[
n(\mu,\epsilon)=\overline{n(\mu,\epsilon)}.
\]
Thus \(n(\mu,\epsilon)\in\mathbb R\). Since the basis $F$ is reversible and symplectic, and $n(\mu,\epsilon)$ is real, the new basis
$\mathcal G$ is again symplectic and reversible.

\begin{lemma}\label{lem:second-basis}
The basis $\mathcal G$ in \eqref{4.8} is symplectic and reversible, namely it satisfies
\eqref{basis is symplectic} and \eqref{basis is reversible}. Moreover, each map
\[
(\mu,\epsilon)\mapsto g_k^\sigma(\mu,\epsilon),\qquad \sigma=\pm,\ k=0,1,
\]
is analytic from $B(\mu_0)\times B(\epsilon_0)$ to $H^4(\mathbb T)\times H^1(\mathbb T)$.
\end{lemma}

\begin{proof}
The vectors $g_k^\sigma(\mu,\epsilon)$ satisfy \eqref{basis is symplectic} and \eqref{basis is reversible} because the vectors
$f_k^\sigma(\mu,\epsilon)$ satisfy the same properties by Lemma~\ref{F is symplectic and reversible}, and
$n(\mu,\epsilon)$ is real. The analyticity of $g_k^\sigma(\mu,\epsilon)$ follows from the
corresponding analyticity of the basis $F$.
\end{proof}

We now introduce the finite-dimensional matrix representation of
$\mathscr{L}_{\mu,\epsilon}$ in the basis $\mathcal G$.

\begin{lemma}\label{matrix representation mathsf L}
Let $\mathcal{G}$ be the symplectic and reversible basis of
$\mathcal V_{\mu,\epsilon}$ defined in \eqref{4.8}.
Then the matrix representing the restriction
\[
\mathscr L_{\mu,\epsilon}
\big|_{\mathcal V_{\mu,\epsilon}}
:
\mathcal V_{\mu,\epsilon}
\longrightarrow
\mathcal V_{\mu,\epsilon}
\]
with respect to the ordered basis $\mathcal G$ is
\begin{equation}\label{J4Bmuepsilon}
\mathsf L_{\mu,\epsilon}
:=
\left[
\mathscr L_{\mu,\epsilon}
\big|_{\mathcal V_{\mu,\epsilon}}
\right]_{\mathcal G}
=
\mathsf J_4\mathsf B_{\mu,\epsilon},
\qquad
\mathsf J_4
:=
\begin{pmatrix}
\mathsf J_2&0\\
0&\mathsf J_2
\end{pmatrix},
\qquad
\mathsf J_2
:=
\begin{pmatrix}
0&1\\
-1&0
\end{pmatrix},
\end{equation}
where
\begin{equation}\label{Bmuepsilon matrix representation}
\mathsf B_{\mu,\epsilon}
:=
\left(
\big(
\mathcal B_{\mu,\epsilon}\mathsf{g}_j,
\mathsf{g}_i
\big)
\right)_{1\leq i,j\leq4},
\qquad
(\mathsf{g}_1,\mathsf{g}_2,\mathsf{g}_3,\mathsf{g}_4)
:=
(g_1^+,g_1^-,g_0^+,g_0^-).
\end{equation}
The matrix $\mathsf B_{\mu,\epsilon}$ is self-adjoint:
\[
\mathsf B_{\mu,\epsilon}
=
\mathsf B_{\mu,\epsilon}^*.
\]
Moreover, for any $k,k'\in\{0,1\}$ and
$\sigma,\sigma'\in\{+,-\}$,
\begin{equation}\label{B are alternatively real or imaginary}
\big(
\mathcal B_{\mu,\epsilon}g_k^\sigma,
g_{k'}^{\sigma'}
\big)
\in
\begin{cases}
\mathbb R,
&\sigma=\sigma',\\[1mm]
\im\mathbb R,
&\sigma=-\sigma'.
\end{cases}
\end{equation}
In particular, $\mathsf B_{\mu,\epsilon}$ is
reversibility-preserving and
$\mathsf L_{\mu,\epsilon}$ is Hamiltonian and reversible.
\end{lemma}

\begin{proof}
Since $\mathcal G$ is symplectic, every
$h\in\mathcal V_{\mu,\epsilon}$ admits the expansion
\[
\begin{aligned}
h=
-(\mathcal Jh,g_1^-)\,g_1^+
+(\mathcal Jh,g_1^+)\,g_1^- -(\mathcal Jh,g_0^-)\,g_0^+
+(\mathcal Jh,g_0^+)\,g_0^-.
\end{aligned}
\]
Applying this identity to
$h=\mathscr L_{\mu,\epsilon}g_k^\sigma$ and using
\[
\mathscr L_{\mu,\epsilon}
=
\mathcal J\mathcal B_{\mu,\epsilon},
\qquad
\mathcal J^2=-\operatorname{Id},
\]
we obtain
\[
\begin{aligned}
\mathscr L_{\mu,\epsilon}g_k^\sigma
=\big(
\mathcal B_{\mu,\epsilon}g_k^\sigma,
g_1^-
\big)g_1^+
-
\big(
\mathcal B_{\mu,\epsilon}g_k^\sigma,
g_1^+
\big)g_1^- +
\big(
\mathcal B_{\mu,\epsilon}g_k^\sigma,
g_0^-
\big)g_0^+
-
\big(
\mathcal B_{\mu,\epsilon}g_k^\sigma,
g_0^+
\big)g_0^-.
\end{aligned}
\]
This proves
\[
\mathsf L_{\mu,\epsilon}
=
\mathsf J_4\mathsf B_{\mu,\epsilon}.
\]
The self-adjointness of $\mathsf B_{\mu,\epsilon}$ follows from that
of $\mathcal B_{\mu,\epsilon}$.

Finally, since
\[
(f,g)
=
\overline{
(\overline\rho f,\overline\rho g)
},
\qquad
\mathcal B_{\mu,\epsilon}\overline\rho
=
\overline\rho\mathcal B_{\mu,\epsilon},
\]
and
\[
\overline\rho g_k^\sigma
=
\sigma g_k^\sigma,
\]
we have
\[
\begin{aligned}
\big(
\mathcal B_{\mu,\epsilon}g_k^\sigma,
g_{k'}^{\sigma'}
\big)
&=
\overline{
\big(
\overline\rho\mathcal B_{\mu,\epsilon}g_k^\sigma,
\overline\rho g_{k'}^{\sigma'}
\big)
}=
\sigma\sigma'
\overline{
\big(
\mathcal B_{\mu,\epsilon}g_k^\sigma,
g_{k'}^{\sigma'}
\big)
}.
\end{aligned}
\]
Hence the entry is real when $\sigma=\sigma'$ and purely imaginary
when $\sigma=-\sigma'$, proving
\eqref{B are alternatively real or imaginary}.
\end{proof}

Next we provide the expansion of the new basis $\mathcal G$.

\begin{lemma}[New expansion of the basis $\mathcal{G}$]\label{lem:basis-G}
For small values of $(\mu, \epsilon)$, the basis $\mathcal{G}$ defined in \eqref{4.8} has the expansion:

\begin{equation}
\begin{aligned}
g_1^+(\mu,\epsilon)
&=
\vet{\delta_{\kappa,b}^{-1}\cos(x)}
    {\delta_{\kappa,b}\sin(x)}
+ \im\frac{\mu}{4}\gamma_{\kappa,b}
\vet{\delta_{\kappa,b}^{-1}\sin(x)}
    {\delta_{\kappa,b}\cos(x)}
+ \epsilon
\vet{\alpha_{\kappa,b}\cos(2x)}
    {\beta_{\kappa,b}\sin(2x)}
\\
&\quad
+ \mathcal{O}(\epsilon^2)
\vet{\mathit{even}_0(x)}
    {\mathit{odd}(x)}
+ \mathcal{O}(\mu^2)
\vet{\mathit{even}_0(x)+\im\,\mathit{odd}(x)}
    {\mathit{odd}(x)+\im\,\mathit{even}_0(x)}
\\
&\quad
+ \im\mu\epsilon
\vet{\mathit{odd}(x)}
    {\mathit{even}(x)}
+ \mathcal{O}(\mu^2\epsilon,\mu\epsilon^2),
\label{eq:g1_plus}
\end{aligned}
\end{equation}

\begin{equation}
\begin{aligned}
g_1^-(\mu,\epsilon)
&=
\vet{-\delta_{\kappa,b}^{-1}\sin(x)}
    {\delta_{\kappa,b}\cos(x)}
+ \im\frac{\mu}{4}\gamma_{\kappa,b}
\vet{\delta_{\kappa,b}^{-1}\cos(x)}
    {-\delta_{\kappa,b}\sin(x)}
+ \epsilon
\vet{-\alpha_{\kappa,b}\sin(2x)}
    {\beta_{\kappa,b}\cos(2x)}
\\
&\quad
+ \mathcal{O}(\epsilon^2)
\vet{\mathit{odd}(x)}
    {\mathit{even}_0(x)}
+ \mathcal{O}(\mu^2)
\vet{\mathit{odd}(x)+\im\,\mathit{even}_0(x)}
    {\mathit{even}_0(x)+\im\,\mathit{odd}(x)}
\\
&\quad
- \mu\epsilon\,\nu_{\kappa,b}
\vet{0}{1}
+ \im\mu\epsilon
\vet{\mathit{even}(x)}
    {\mathit{odd}(x)}
+ \mathcal{O}(\mu^2\epsilon,\mu\epsilon^2),
\label{eq:g1_minus}
\end{aligned}
\end{equation}

\begin{equation}
\begin{aligned}
g_0^+(\mu,\epsilon)
&=
\vet{1}{0}
+ \epsilon\delta_{\kappa,b}
\vet{\delta_{\kappa,b}^{-1}\cos(x)}
    {-\delta_{\kappa,b}\sin(x)}
+ \frac{\mu\epsilon}{4}
\vet{\delta_{\kappa,b}^{-4}\cos(x)}
    {-\delta_{\kappa,b}^{-2}\sin(x)}
\\
&\quad
+ \mu\epsilon\,\nu_{\kappa,b}
\vet{\delta_{\kappa,b}^{-1}\cos(x)}
    {\delta_{\kappa,b}\sin(x)}
+ \mathcal{O}(\epsilon^2)
\vet{\mathit{even}_0(x)}
    {\mathit{odd}(x)}
\\
&\quad
+ \im\mu\epsilon
\vet{\mathit{odd}(x)}
    {\mathit{even}_0(x)}
+ \mathcal{O}(\mu^2\epsilon,\mu\epsilon^2),
\label{eq:g0_plus}
\end{aligned}
\end{equation}

\begin{align}
g_0^-(\mu,\epsilon)
&=
\vet{0}{1}
+ \frac{\mu\epsilon}{2}
\vet{\delta_{\kappa,b}^{-2}\sin (x)}
    {\cos (x)}
+ \im\mu\epsilon
\vet{\mathit{even}_0(x)}
    {\mathit{odd}(x)}
+ \mathcal{O}(\mu^2\epsilon,\mu\epsilon^2).
\label{eq:g0_minus}
\end{align}
In particular, at $\mu=0$, the basis
$\{g_1^+(0,\epsilon),g_1^-(0,\epsilon),g_0^+(0,\epsilon),g_0^-(0,\epsilon)\}$
is real and satisfies
\begin{equation}
g_1^+(0,\epsilon)
=
\vet{\mathit{even}_0(x)}{\mathit{odd}(x)},
\qquad
g_1^-(0,\epsilon)
=
\vet{\mathit{odd}(x)}{\mathit{even}_0(x)},
\end{equation}
\begin{equation}
g_0^+(0,\epsilon)
=
\vet{1}{0}
+
\vet{\mathit{even}_0(x)}{\mathit{odd}(x)},
\qquad
g_0^-(0,\epsilon)
=
\vet{0}{1}.
\label{eq:g_mu_zero}
\end{equation}
Moreover, for any $\epsilon$,
\begin{equation}
\int_{\mathbb T} (g_1^-(0,\epsilon))_2\,dx=0.
\label{eq:zero_average}
\end{equation}
\end{lemma}

\begin{proof}
Since \(f_0^-(0,\epsilon)=\vet{0}{1}\), the definition of \(n\) gives
\[
g_1^-(0,\epsilon)
=
f_1^-(0,\epsilon)
-
\frac{\left(f_1^-(0,\epsilon),f_0^-(0,\epsilon)\right)}
     {\|f_0^-(0,\e)\|^2}f_0^-(0,\epsilon).
\]
Thus the second component of \(g_1^-(0,\epsilon)\) has zero average; together
with Lemma~\ref{lem:4.2}, this yields
\[
g_1^-(0,\epsilon)
=
\vet{\mathit{odd}(x)}{\mathit{even}_0(x)}.
\]
Using the expansions in Lemma~\ref{lem:4.2}, Fourier orthogonality and parity,
\begin{align*}
\bigl(f_1^-(\mu,\epsilon),f_0^-(\mu,\epsilon)\bigr)
&=
\mu\epsilon
\left(
\vet{-c_{\kappa,b}^{-1/2}\sin (x)}
      {c_{\kappa,b}^{1/2}\cos (x)},
\vet{\frac{\sin (x)}{2c_{\kappa,b}}}
      {\frac{\cos (x)}{2}}
\right)
+\mathcal O(\epsilon^2,\mu^2\epsilon,\mu\epsilon^2)=
\frac{c_{\kappa,b}^2-1}{4c_{\kappa,b}^{3/2}}\,
\mu\epsilon
+\mathcal O(\epsilon^2,\mu^2\epsilon,\mu\epsilon^2).
\end{align*}
Moreover,
\[
\|f_0^-(\mu,\epsilon)\|^2
=
1+\mathcal O(\mu^2\epsilon,\mu\epsilon^2).
\]
Since \(c_{\kappa,b}^2=1+b+\kappa\), it follows that
\[
n(\mu,\epsilon)
=
\nu_{\kappa,b}\mu\epsilon
+\mathcal O(\epsilon^2,\mu^2\epsilon,\mu\epsilon^2),
\qquad
\nu_{\kappa,b}:=
\frac{b+\kappa}{4c_{\kappa,b}^{3/2}}.
\]

Substituting this into
\[
g_1^+=f_1^+,\qquad
g_0^-=f_0^-,\qquad
g_1^-=f_1^--nf_0^-,\qquad
g_0^+=f_0^++nf_1^+,
\]
and using Lemma~\ref{lem:4.2} gives all the claimed expansions. In particular,
the leading corrections are
\[
-\nu_{\kappa,b}\mu\epsilon\vet{0}{1}
\quad\text{in }g_1^-,
\qquad
\nu_{\kappa,b}\mu\epsilon
\vet{c_{\kappa,b}^{-1/2}\cos (x)}
      {c_{\kappa,b}^{1/2}\sin (x)}
\quad\text{in }g_0^+.
\]
The reality at \(\mu=0\) follows from Lemma~\ref{lem:4.2} and
\(n(0,\epsilon)\in\mathbb R\).
\end{proof}

\color{black}
\begin{lemma}[Identification of the odd kernel direction]
\label{lem:G-odd-kernel-direction}
For every sufficiently small $|\epsilon|$, there exists a nonzero real
number $q_\epsilon$ such that
\begin{equation}\label{eq:U2-q-g1minus}
U_2(\epsilon)
=
q_\epsilon\,g_1^-(0,\epsilon).
\end{equation}
Moreover, $q_\epsilon$ depends real-analytically on $\epsilon$ and
$q_0=c_{\kappa,b}^{1/2}$.  Consequently,
\begin{equation}\label{eq:gminus-kernel-L}
\mathscr L_{0,\epsilon}g_1^-(0,\epsilon)=0,
\qquad
\mathscr L_{0,\epsilon}g_0^-(0,\epsilon)=0,
\end{equation}
and
\begin{equation}\label{eq:gminus-kernel-B}
\mathcal B_\epsilon g_1^-(0,\epsilon)=0,
\qquad
\mathcal B_\epsilon g_0^-(0,\epsilon)=0.
\end{equation}
\end{lemma}

\begin{proof}
By Proposition~\ref{prop:periodic-generalized-kernel}, the vector
$U_2(\epsilon)$ is a real element of $\mathcal V_{0,\epsilon}$ and
satisfies
\[
\overline\rho U_2(\epsilon)=-U_2(\epsilon).
\]
At $\mu=0$, Lemma~\ref{lem:basis-G} implies that
$\mathcal G(0,\epsilon)$ is a real basis of $\mathcal V_{0,\epsilon}$ and
\[
\overline\rho g_k^+(0,\epsilon)=g_k^+(0,\epsilon),
\qquad
\overline\rho g_k^-(0,\epsilon)=-g_k^-(0,\epsilon),
\qquad k=0,1.
\]
We first record the following real-linear decomposition:
\begin{equation}\label{eq:real-antireversible-subspace}
\left\{
 h\in\mathcal V_{0,\epsilon}:
 \overline h=h,
 \ \overline\rho h=-h
\right\}
=
\operatorname{span}_{\mathbb R}
\{g_1^-(0,\epsilon),g_0^-(0,\epsilon)\}.
\end{equation}
Indeed, write a real vector $h\in\mathcal V_{0,\epsilon}$ uniquely as
\[
h
=a_1g_1^++b_1g_1^-+a_0g_0^++b_0g_0^-.
\]
Since both $h$ and the basis vectors are real, uniqueness of the expansion
implies $a_1,b_1,a_0,b_0\in\mathbb R$.  Applying $\overline\rho$ and using
$\overline\rho h=-h$ yields $a_1=a_0=0$, proving
\eqref{eq:real-antireversible-subspace}.

It follows that there exist real numbers $q_\epsilon,r_\epsilon$ such that
\begin{equation}\label{eq:U2-minus-decomposition}
U_2(\epsilon)
=
q_\epsilon g_1^-(0,\epsilon)
+
r_\epsilon g_0^-(0,\epsilon).
\end{equation}
By Lemma~\ref{lem:basis-G},
\[
\left\langle
\bigl(g_1^-(0,\epsilon)\bigr)_2
\right\rangle_x=0,
\qquad
 g_0^-(0,\epsilon)=\begin{bmatrix}0\\1\end{bmatrix},
\qquad
\left\langle
\bigl(g_0^-(0,\epsilon)\bigr)_2
\right\rangle_x=1.
\]
Taking the spatial average of the second component of
\eqref{eq:U2-minus-decomposition} and using
\eqref{eq:U2-zero-average}, we obtain $r_\epsilon=0$.  Hence
\eqref{eq:U2-q-g1minus} holds.

The vector $U_2(\epsilon)$ is nonzero for small $|\epsilon|$, because
\eqref{eq:U2-leading-value} gives
$U_2(0)=c_{\kappa,b}^{1/2}f_1^-\neq0$.  Thus $q_\epsilon\neq0$.
Since both $U_2(\epsilon)$ and $g_1^-(0,\epsilon)$ depend
real-analytically on $\epsilon$, one may write
\[
q_\epsilon
=
\frac{
\bigl(U_2(\epsilon),g_1^-(0,\epsilon)\bigr)
}{
\|g_1^-(0,\epsilon)\|^2
},
\]
which proves the real-analytic dependence.  At $\epsilon=0$,
\eqref{eq:U2-leading-value} and
$g_1^-(0,0)=f_1^-$ give $q_0=c_{\kappa,b}^{1/2}$.

By \eqref{eq:generalized-kernel-relations} and
\eqref{eq:U2-q-g1minus},
\[
\mathscr L_{0,\epsilon}g_1^-(0,\epsilon)=0.
\]
Furthermore, Lemma~\ref{lem:basis-G} gives
$g_0^-(0,\epsilon)=U_1$, so
$\mathscr L_{0,\epsilon}g_0^-(0,\epsilon)=0$.
Finally,
\[
\mathscr L_{0,\epsilon}=\mathcal J\mathcal B_\epsilon
\]
and $\mathcal J$ is invertible; therefore
\eqref{eq:gminus-kernel-B} follows.
\end{proof}

\color{black}
We now compute the self-adjoint matrix
$\mathsf B_{\mu,\epsilon}$ defined in
\eqref{Bmuepsilon matrix representation}. We write
\begin{equation}\label{4.19}
\mathsf B_{\mu,\epsilon}
=
\begin{pmatrix}
E&F\\
F^*&G
\end{pmatrix},
\qquad
E=E^*,
\qquad
G=G^*.
\end{equation}
It is useful to distinguish the operator on the full function space from its
matrix on $\mathcal V_{\mu,\epsilon}$. We decompose
\begin{equation}\label{eq:operator-decomposition-B}
\mathcal B_{\mu,\epsilon}
=
\mathcal B_{\epsilon}
+\mathcal B^{\flat}_{\mu}
+\mathcal B^{\sharp}_{\mu,\epsilon}
+\mathcal B^s_{\mu,\epsilon},
\end{equation}
where
\begin{align}
\mathcal B_{\epsilon}
&:=
\begin{pmatrix}
1+a_\epsilon-\kappa\tau_\epsilon+b\beta_\epsilon
&-(c_{\kappa,b}+p_\epsilon)\partial_x\\
\partial_x\circ(c_{\kappa,b}+p_\epsilon)&|D|
\end{pmatrix},
\label{B epsilon}\\[1mm]
\mathcal B^{\flat}_{\mu}
&:=
\begin{pmatrix}
0&0\\
0&\mu\bigl(\sgn(D)+\Pi_0\bigr)
\end{pmatrix},
\label{B b}\\[1mm]
\mathcal B^{\sharp}_{\mu,\epsilon}
&:=
\begin{pmatrix}
0&-\im\mu p_\epsilon\\
\im\mu p_\epsilon&0
\end{pmatrix},
\qquad
\mathcal B^s_{\mu,\epsilon}
:=
\begin{pmatrix}
\mathcal S_{\mu,\epsilon}&0\\
0&0
\end{pmatrix}.
\label{B sharp}
\end{align}
Here
\begin{equation}\label{eq:S-mu-epsilon-expansion}
\mathcal S_{\mu,\epsilon}
=
\sum_{j=1}^4\frac{\mu^j}{j!}\mathcal S_{0,\epsilon}^{(j)},
\end{equation}
with
\[
\mathcal S_{0,\epsilon}^{(1)}
=-\kappa\tau_\epsilon^{(1)}+b\beta_\epsilon^{(1)},
\qquad
\mathcal S_{0,\epsilon}^{(2)}
=-\kappa\tau_\epsilon^{(2)}+b\beta_\epsilon^{(2)},\qquad
\mathcal S_{0,\epsilon}^{(3)}=b\beta_\epsilon^{(3)},
\qquad
\mathcal S_{0,\epsilon}^{(4)}=b\beta_\epsilon^{(4)}.
\]

For $q\in\{\epsilon,\flat,\sharp,s\}$, let $\mathsf{B}^q$ denote the matrix of
$\mathcal B^q$ in the basis $\mathcal G(\mu,\epsilon)$, namely
\begin{equation}\label{eq:matrix-components-Bq}
(\mathsf B^q)_{ij}
:=
\bigl(\mathcal B^q \mathsf g_j(\mu,\epsilon),\mathsf g_i(\mu,\epsilon)\bigr),
\qquad 1\leq i,j\leq4.
\end{equation}
Thus
\begin{equation}\label{eq:matrix-decomposition-B}
\mathsf B_{\mu,\epsilon}
=\mathsf B^\epsilon+\mathsf B^\flat+\mathsf B^\sharp+\mathsf B^s.
\end{equation}

To state the estimates economically, let $\mathcal R_{\mu,\epsilon}$ denote
the class of self-adjoint, reversibility-preserving matrices
$\mathsf R=(\mathsf R_{ij})_{1\leq i,j\leq4}$ of the form
\begin{equation}\label{eq:remainder-class}
\mathsf R=
\begin{pmatrix}
\mathcal O_0&\im\mathcal O_1&\mathcal O_0&\im\mathcal O_2\\
-\im\mathcal O_1&\mathcal O_3&\im\mathcal O_2&\mathcal O_3\\
\mathcal O_0&-\im\mathcal O_2&\mathcal O_0&\im\mathcal O_1\\
-\im\mathcal O_2&\mathcal O_3&-\im\mathcal O_1&\mathcal O_3
\end{pmatrix},
\end{equation}
where
\begin{align}
\mathcal O_0&=\mathcal O(\epsilon^3,\mu\epsilon^2,\mu^2\epsilon,\mu^3),\qquad \mathcal O_1=\mathcal O(\mu\epsilon^2,\mu^2\epsilon,\mu^3),
\label{eq:remainder-orders-1}\\
\mathcal O_2&=\mathcal O(\mu\epsilon,\mu^3),\qquad\mathcal O_3=\mathcal O(\mu^2\epsilon,\mu^3).
\label{eq:remainder-orders-2}
\end{align}
Different occurrences of $\mathcal O_j$ may denote different real-analytic
functions satisfying the corresponding estimates.

We first state the resulting reduced matrix. Its proof is given after the four
components in \eqref{eq:matrix-decomposition-B} have been estimated.

\begin{proposition}[Expansion of the reduced matrix]
\label{lem:matrix-representation-G}
For sufficiently small $(\mu,\epsilon)$,
\begin{equation}\label{eq:reduced-matrix-expansion}
\mathsf B_{\mu,\epsilon}
=
\begin{pmatrix}
e_{11}\epsilon^2-\dfrac{e_{22}}8\mu^2
&\im\dfrac{e_{12}}2\mu&0&0\\[2mm]
-\im\dfrac{e_{12}}2\mu
&-\dfrac{e_{22}}8\mu^2&0&0\\[2mm]
0&0&1+\kappa\mu^2&0\\
0&0&0&\mu
\end{pmatrix}
+\mathsf R_{\mu,\epsilon},
\qquad
\mathsf R_{\mu,\epsilon}\in\mathcal R_{\mu,\epsilon},
\end{equation}
where
\begin{align}
e_{11}
&=-\frac{88b^2+6b\kappa-54b+2\kappa^2+\kappa+8}
{8c_{\kappa,b}(14b+2\kappa-1)},
\label{eq:e11-matrix}\\
e_{12}
&=\frac{1+3\kappa+5b}{c_{\kappa,b}},
\label{eq:e12-matrix}\\
e_{22}
&=-\frac{15b^2+3\kappa^2+22b\kappa+30b+6\kappa-1}
{c_{\kappa,b}^3}.
\label{eq:e22-matrix}
\end{align}
Equivalently, the blocks $E,F,G$ in \eqref{4.19} have the expansions encoded
by \eqref{eq:reduced-matrix-expansion}--\eqref{eq:remainder-orders-2}.
\end{proposition}

\begin{lemma}[The $\mathsf B^\epsilon$-contribution]
\label{Expansion of B epsilon}
Set
\begin{equation}\label{eq:zeta-kappa-b}
\zeta_{\kappa,b}:=\frac18c_{\kappa,b}\gamma_{\kappa,b}^2.
\end{equation}
Then
\begin{equation}\label{eq:B-epsilon-contribution}
\mathsf B^\epsilon
=
\begin{pmatrix}
e_{11}\epsilon^2+\zeta_{\kappa,b}\mu^2&0&0&0\\
0&\zeta_{\kappa,b}\mu^2&0&0\\
0&0&1&0\\
0&0&0&0
\end{pmatrix}
+\mathsf R^\epsilon_{\mu,\epsilon},
\qquad
\mathsf R^\epsilon_{\mu,\epsilon}\in\mathcal R_{\mu,\epsilon}.
\end{equation}
\end{lemma}

\begin{proof}
We divide the proof into the expansion at $\mu=0$ and the dependence on the
Floquet parameter.

\smallskip
\noindent
\emph{Step 1: the matrix at $\mu=0$.}
At $\mu=0$, both the operator $\mathcal B_\epsilon$ and the basis
$\mathcal G(0,\epsilon)$ are real. Since $\mathcal B_\epsilon$ is
self-adjoint and reversibility-preserving, the entries with $i+j$ odd vanish.
\color{black}
By Lemma~\ref{lem:G-odd-kernel-direction},
\[
\mathcal B_\epsilon g_1^-(0,\epsilon)
=
\mathcal B_\epsilon g_0^-(0,\epsilon)
=0.
\]
\color{black}
Hence the second
and fourth rows and columns vanish, and
\begin{equation}\label{eq:Be0-structure}
\mathsf B^\epsilon(0)
=
\begin{pmatrix}
a(\epsilon)&0&\alpha(\epsilon)&0\\
0&0&0&0\\
\alpha(\epsilon)&0&\chi(\epsilon)&0\\
0&0&0&0
\end{pmatrix}
\end{equation}
for real-analytic scalar functions $a,\chi,\alpha$.

We use the expansions
\begin{equation}\label{eq:B-epsilon-operator-expansion}
\mathcal B_\epsilon
=\mathcal B_0+\epsilon\mathcal B_1+\epsilon^2\mathcal B_2
+\mathcal O(\epsilon^3),
\end{equation}
where
\[
\mathcal B_0
=
\begin{bmatrix}
1-\kappa\partial_x^2+b\partial_x^4&-c_{\kappa,b}\partial_x\\
c_{\kappa,b}\partial_x&|D|
\end{bmatrix},
\]
\[
\mathcal B_1
=
\begin{bmatrix}
a_1-\kappa\tau_1+b\beta_1&-p_1\partial_x\\
\partial_x\circ p_1&0
\end{bmatrix},
\qquad
\mathcal B_2
=
\begin{bmatrix}
a_2-\kappa\tau_2+b\beta_2&-p_2\partial_x\\
\partial_x\circ p_2&0
\end{bmatrix}.
\]
By Lemma~\ref{lem:basis-G},
\begin{align}
g_1^+(0,\epsilon)
&=f_1^++\epsilon g_{11}^++\epsilon^2g_{12}^+
+\mathcal O(\epsilon^3),
\label{eq:g1plus-epsilon-expansion}\\
g_0^+(0,\epsilon)
&=f_0^++\epsilon g_{01}^++\epsilon^2g_{02}^+
+\mathcal O(\epsilon^3),
\label{eq:g0plus-epsilon-expansion}
\end{align}
where
\[
f_1^+
=
\begin{bmatrix}
c_{\kappa,b}^{-1/2}\cos (x)\\
c_{\kappa,b}^{1/2}\sin (x)
\end{bmatrix},
\qquad
f_0^+=\begin{bmatrix}1\\0\end{bmatrix},\qquad
g_{11}^+
=
\begin{bmatrix}
\alpha_{\kappa,b}\cos(2x)\\
\beta_{\kappa,b}\sin(2x)
\end{bmatrix},
\qquad
g_{01}^+
=
\begin{bmatrix}
\cos (x)\\-c_{\kappa,b}\sin (x)
\end{bmatrix},
\]
and the first components of $g_{12}^+$ and $g_{02}^+$ have zero spatial
average.

We first compute $a(\epsilon)$. Expanding the quadratic form and using
$\mathcal B_0f_1^+=0$, we obtain
\begin{align*}
a(\epsilon)
={}&\epsilon(\mathcal B_1f_1^+,f_1^+)+\epsilon^2\Bigl[
(\mathcal B_2f_1^+,f_1^+)
+2(\mathcal B_1f_1^+,g_{11}^+)
+(\mathcal B_0g_{11}^+,g_{11}^+)
\Bigr]
+\mathcal O(\epsilon^3).
\end{align*}
The coefficient of $\epsilon$ vanishes: the vector
$\mathcal B_1f_1^+$ contains only the Fourier modes $0$ and $2$, whereas
$f_1^+$ contains only the first harmonics. Therefore
\begin{equation}\label{eq:a-epsilon-expansion}
a(\epsilon)=e_{11}\epsilon^2+\mathcal O(\epsilon^3),
\end{equation}
where
\begin{equation}\label{eq:e11-quadratic-form}
e_{11}
=(\mathcal B_2f_1^+,f_1^+)
+2(\mathcal B_1f_1^+,g_{11}^+)
+(\mathcal B_0g_{11}^+,g_{11}^+).
\end{equation}
Substitution of the coefficients of $a_j,p_j,\tau_j,\beta_j$ gives
\eqref{eq:e11-matrix}. No non-vanishing assumption on $e_{11}$ is used.

We next compute $\chi(\epsilon)$. Since
$\mathcal B_0f_0^+=f_0^+$ and the first components of
$g_{01}^+,g_{02}^+$ have zero average, self-adjointness gives
\begin{align*}
\chi(\epsilon)
={}&1+\epsilon(\mathcal B_1f_0^+,f_0^+)+\epsilon^2\Bigl[
(\mathcal B_2f_0^+,f_0^+)
+2(\mathcal B_1f_0^+,g_{01}^+)
+(\mathcal B_0g_{01}^+,g_{01}^+)
\Bigr]
+\mathcal O(\epsilon^3).
\end{align*}
The linear coefficient vanishes because the first component of
$\mathcal B_1f_0^+$ has zero average. Furthermore,
\[
\mathcal B_1f_0^+
=
\begin{bmatrix}
-2c_{\kappa,b}^2\cos (x)\\
2c_{\kappa,b}\sin (x)
\end{bmatrix},
\qquad
\mathcal B_0g_{01}^+
=
\begin{bmatrix}
2c_{\kappa,b}^2\cos (x)\\
-2c_{\kappa,b}\sin (x)
\end{bmatrix},
\]
and the zero Fourier coefficients of $a_2,\tau_2,\beta_2$ give
\[
(\mathcal B_2f_0^+,f_0^+)=2c_{\kappa,b}^2.
\]
Consequently,
\[
(\mathcal B_2f_0^+,f_0^+)
+2(\mathcal B_1f_0^+,g_{01}^+)
+(\mathcal B_0g_{01}^+,g_{01}^+)
=2c_{\kappa,b}^2-4c_{\kappa,b}^2+2c_{\kappa,b}^2=0,
\]
and hence
\begin{equation}\label{eq:c-epsilon-expansion}
\chi(\epsilon)=1+\mathcal O(\epsilon^3).
\end{equation}

Finally, consider
\[
\alpha(\epsilon)
=(\mathcal B_\epsilon g_1^+(0,\epsilon),g_0^+(0,\epsilon)).
\]
Using \eqref{eq:B-epsilon-operator-expansion}--\eqref{eq:g0plus-epsilon-expansion},
$\mathcal B_0f_1^+=0$, and $\mathcal B_0f_0^+=f_0^+$, we find
\begin{equation}\label{eq:alpha-epsilon-expansion-preliminary}
\alpha(\epsilon)
=\epsilon\alpha_1+\epsilon^2\alpha_2+\mathcal O(\epsilon^3),
\end{equation}
where
\[
\alpha_1
=(\mathcal B_1f_1^+,f_0^+)
+(\mathcal B_0g_{11}^+,f_0^+),
\]
and
\begin{align*}
\alpha_2
={}&(\mathcal B_2f_1^+,f_0^+)
+(\mathcal B_1g_{11}^+,f_0^+)
+(\mathcal B_1f_1^+,g_{01}^+)+(\mathcal B_0g_{12}^+,f_0^+)
+(\mathcal B_0g_{11}^+,g_{01}^+).
\end{align*}
The second term in $\alpha_1$ vanishes because $g_{11}^+$ contains only the
second harmonics. For the first term, the zero Fourier coefficient of the
first component of $\mathcal B_1f_1^+$ is
\[
c_{\kappa,b}^{-1/2}
\left[-\frac{2+b+\kappa}{2}-\frac{\kappa}{2}
-\frac{b}{2}+c_{\kappa,b}^2\right]=0,
\]
because $c_{\kappa,b}^2=1+b+\kappa$. Thus $\alpha_1=0$.

We now show that $\alpha_2=0$. Indeed,
\[
(\mathcal B_0g_{12}^+,f_0^+)=0
\]
because the first component of $g_{12}^+$ has zero average. Moreover,
\[
(\mathcal B_2f_1^+,f_0^+)=0,
\qquad
(\mathcal B_1g_{11}^+,f_0^+)=0:
\]
the coefficients of $\mathcal B_2$ contain only the modes $0$ and $2$, so
$\mathcal B_2f_1^+$ contains only the first and third harmonics, while the
coefficients of $\mathcal B_1$ contain only first harmonics, so
$\mathcal B_1g_{11}^+$ also contains only the first and third harmonics.
Furthermore,
\[
(\mathcal B_0g_{11}^+,g_{01}^+)=0
\]
by orthogonality of the second and first harmonics, and
\[
(\mathcal B_1f_1^+,g_{01}^+)=0
\]
because $\mathcal B_1f_1^+$ contains only the modes $0$ and $2$, whereas
$g_{01}^+$ contains only the first harmonics. Consequently,
\begin{equation}\label{eq:alpha-epsilon-expansion}
\alpha(\epsilon)=\mathcal O(\epsilon^3).
\end{equation}
Combining \eqref{eq:Be0-structure}, \eqref{eq:a-epsilon-expansion},
\eqref{eq:c-epsilon-expansion}, and \eqref{eq:alpha-epsilon-expansion}, we get
\begin{equation}\label{eq:B-epsilon-at-mu-zero}
\mathsf B^\epsilon(0)
=
\begin{pmatrix}
e_{11}\epsilon^2&0&0&0\\
0&0&0&0\\
0&0&1&0\\
0&0&0&0
\end{pmatrix}
+\mathsf R^\epsilon_{0,\epsilon},
\qquad
\mathsf R^\epsilon_{0,\epsilon}\in\mathcal R_{0,\epsilon}.
\end{equation}

\smallskip
\noindent
\emph{Step 2: dependence on $\mu$.}
The operator $\mathcal B_\epsilon$ is independent of $\mu$; all
$\mu$-dependence of $\mathsf B^\epsilon$ comes from the basis. Differentiating
\eqref{eq:matrix-components-Bq}, we obtain
\begin{equation}\label{eq:first-mu-derivative-B-epsilon}
\partial_\mu(\mathsf B^\epsilon)_{ij}(0,\epsilon)
=
(\mathcal B_\epsilon\dot g_j(0,\epsilon),g_i(0,\epsilon))
+(\mathcal B_\epsilon g_j(0,\epsilon),\dot g_i(0,\epsilon)).
\end{equation}
Since $g_1^-(0,\epsilon),g_0^-(0,\epsilon)\in\ker\mathcal B_\epsilon$,
the terms in \eqref{eq:first-mu-derivative-B-epsilon} involving these vectors
vanish by self-adjointness. Substituting the expansions of
$\dot g_j(0,\epsilon)$ from Lemma~\ref{lem:basis-G}, and using parity and the
zero-average property of $g_1^-(0,\epsilon)$, gives
\begin{equation}\label{eq:first-mu-derivative-remainder}
\mu\,\partial_\mu \mathsf B^\epsilon(0,\epsilon)
\in\mathcal R_{\mu,\epsilon}.
\end{equation}
The largest possible linear contribution is the $(2,3)$-entry and its
adjoint. More precisely,
\[
\partial_\mu(\mathsf B^\epsilon)_{23}(0,\epsilon)
=(\mathcal B_\epsilon g_0^+(0,\epsilon),\dot g_1^-(0,\epsilon))
=\im\mathcal O(\epsilon),
\]
whereas the remaining entries satisfy the stronger bounds prescribed by
$\mathcal O_0,\mathcal O_1,$ and $\mathcal O_3$ in
\eqref{eq:remainder-orders-1}--\eqref{eq:remainder-orders-2}.

It remains to compute the quadratic term. At $(\mu,\epsilon)=(0,0)$,
\[
\dot g_0^+(0,0)=\dot g_0^-(0,0)=0,
\qquad
\ddot g_0^+(0,0)=\ddot g_0^-(0,0)=0,
\]
while
\[
\dot g_1^+(0,0)
=\im\frac{\gamma_{\kappa,b}}4
\begin{bmatrix}
c_{\kappa,b}^{-1/2}\sin (x)\\
c_{\kappa,b}^{1/2}\cos (x)
\end{bmatrix},
\qquad\dot g_1^-(0,0)
=\im\frac{\gamma_{\kappa,b}}4
\begin{bmatrix}
c_{\kappa,b}^{-1/2}\cos (x)\\
-c_{\kappa,b}^{1/2}\sin (x)
\end{bmatrix}.
\]
For $i,j\in\{1,2\}$, the terms containing $\ddot g_i$ or $\ddot g_j$ in
$\partial_\mu^2(\mathsf B^\epsilon)_{ij}(0,0)$ vanish because
$\mathcal B_0f_1^\pm=0$. Thus only the quadratic forms involving the first
$\mu$-derivatives remain. Parity gives the vanishing of the off-diagonal
entry, and a direct computation gives
\[
(\mathcal B_0\dot g_1^+(0,0),\dot g_1^+(0,0))
=(\mathcal B_0\dot g_1^-(0,0),\dot g_1^-(0,0))
=\zeta_{\kappa,b}.
\]
Consequently,
\begin{equation}\label{eq:second-mu-derivative-B-epsilon}
\frac{\mu^2}{2}\partial_\mu^2 \mathsf B^\epsilon(0,0)
=
\begin{pmatrix}
\zeta_{\kappa,b}\mu^2&0&0&0\\
0&\zeta_{\kappa,b}\mu^2&0&0\\
0&0&0&0\\
0&0&0&0
\end{pmatrix}.
\end{equation}
By analyticity, replacing $\epsilon=0$ by small $\epsilon$ changes the
quadratic coefficient by $\mathcal O(\epsilon)$, hence contributes only
$\mathcal O(\mu^2\epsilon)$ to the remainder. Combining
\eqref{eq:B-epsilon-at-mu-zero},
\eqref{eq:first-mu-derivative-remainder}, and
\eqref{eq:second-mu-derivative-B-epsilon} proves
\eqref{eq:B-epsilon-contribution}.
\end{proof}

\begin{lemma}[The $\mathsf B^\flat$-contribution]\label{lem:Bflat}
One has
\begin{equation}\label{eq:B-flat-contribution}
\mathsf B^\flat
=
\begin{pmatrix}
-\dfrac{\gamma_{\kappa,b}c_{\kappa,b}}4\mu^2
&\im\dfrac{c_{\kappa,b}}2\mu&0&0\\[2mm]
-\im\dfrac{c_{\kappa,b}}2\mu
&-\dfrac{\gamma_{\kappa,b}c_{\kappa,b}}4\mu^2&0&0\\[2mm]
0&0&0&0\\
0&0&0&\mu
\end{pmatrix}
+\mathsf R^\flat_{\mu,\epsilon},
\qquad
\mathsf R^\flat_{\mu,\epsilon}\in\mathcal R_{\mu,\epsilon}.
\end{equation}
\end{lemma}

\begin{proof}
Set $T:=\sgn(D)+\Pi_0$. Since
$\mathcal B^\flat_\mu=\operatorname{diag}(0,\mu T)$, only the second
components of the basis vectors enter. Using
\[
T\sin(kx)=-\im\cos(kx),
\qquad
T\cos(kx)=\im\sin(kx),
\qquad
T1=1,
\]
together with Lemma~\ref{lem:basis-G}, we obtain
\begin{align*}
(\mathsf B^\flat)_{12}
&=\im\frac{c_{\kappa,b}}2\mu
+\im\mathcal O(\mu\epsilon^2,\mu^2\epsilon,\mu^3),\\
(\mathsf B^\flat)_{11}=(\mathsf B^\flat)_{22}
&=-\frac{\gamma_{\kappa,b}c_{\kappa,b}}4\mu^2
+\mathcal O(\mu^2\epsilon,\mu^3),\\
(\mathsf B^\flat)_{44}
&=\mu+\mathcal O(\mu^2\epsilon,\mu^3).
\end{align*}
For example, the leading term of $(\mathsf B^\flat)_{12}$ is
\[
\left(\im\mu c_{\kappa,b}^{1/2}\sin (x),
  c_{\kappa,b}^{1/2}\sin (x)\right)
=\im\frac{c_{\kappa,b}}2\mu.
\]
The two order-$\mu^2$ cross terms between the leading first harmonic and the
$\mu$-correction of $g_1^\pm$ give the diagonal coefficient. All remaining
entries satisfy the bounds in \eqref{eq:remainder-class}: the $(2,3)$-entry
and its adjoint may be of order $\mu\epsilon$, while the sharper estimates for
the other entries follow from parity and the zero-average property of
$g_1^-(0,\epsilon)$. This proves \eqref{eq:B-flat-contribution}.
\end{proof}

\begin{lemma}[The $\mathsf B^\sharp$-contribution]\label{lem:Bsharp}
The matrix $\mathsf B^\sharp$ belongs to $\mathcal R_{\mu,\epsilon}$.
\end{lemma}

\begin{proof}
Since $p_\epsilon=-2c_{\kappa,b}\epsilon\cos (x)+\mathcal O(\epsilon^2)$,
one has $\mathcal B^\sharp_{\mu,\epsilon}=\mathcal O(\mu\epsilon)$. Freezing
the basis at $\mu=0$ gives
\[
(\mathcal B^\sharp_{\mu,\epsilon}g_j(\mu,\epsilon),g_i(\mu,\epsilon))
=
(\mathcal B^\sharp_{\mu,\epsilon}g_j(0,\epsilon),g_i(0,\epsilon))
+\mathcal O(\mu^2\epsilon).
\]
The frozen basis is real. Hence the diagonal quadratic forms vanish, and
parity gives the exact vanishing of the $(1,3)$- and $(2,4)$-entries. The
order-$\mu\epsilon$ contributions to the $(1,2)$- and $(3,4)$-entries vanish
by Fourier orthogonality, so their first possible contributions are
$\mathcal O(\mu\epsilon^2)$. The $(1,4)$- and $(2,3)$-entries may be of order
$\mu\epsilon$, exactly as allowed by $\mathcal O_2$. Therefore
$\mathsf B^\sharp\in\mathcal R_{\mu,\epsilon}$.
\end{proof}

\begin{lemma}[The $\mathsf B^s$-contribution]\label{lem:Bs}
Set
\begin{equation}\label{eq:sigma-s-kappa-b}
\sigma^s_{\kappa,b}
:=\frac{\kappa+6b+\gamma_{\kappa,b}(\kappa+2b)}{2c_{\kappa,b}}.
\end{equation}
Then
\begin{equation}\label{eq:B-s-contribution}
\mathsf B^s
=
\begin{pmatrix}
\sigma^s_{\kappa,b}\mu^2
&\im\dfrac{\kappa+2b}{c_{\kappa,b}}\mu&0&0\\[2mm]
-\im\dfrac{\kappa+2b}{c_{\kappa,b}}\mu
&\sigma^s_{\kappa,b}\mu^2&0&0\\[2mm]
0&0&\kappa\mu^2&0\\
0&0&0&0
\end{pmatrix}
+\mathsf R^s_{\mu,\epsilon},
\qquad
\mathsf R^s_{\mu,\epsilon}\in\mathcal R_{\mu,\epsilon}.
\end{equation}
\end{lemma}

\begin{proof}
Write $g_j=(u_j,v_j)^{\mathsf T}$. Since
$\mathcal B^s_{\mu,\epsilon}$ acts only on the first component,
\[
(\mathsf B^s)_{ij}=(\mathcal S_{\mu,\epsilon}u_j,u_i).
\]
At $\epsilon=0$,
\[
\mathcal S_{0,0}^{(1)}=-2\im\kappa\partial_x+4\im b\partial_x^3,
\qquad
\mathcal S_{0,0}^{(2)}=2\kappa-12b\partial_x^2.
\]
Moreover,
\[
u_1^0=c_{\kappa,b}^{-1/2}\cos (x),
\quad
u_2^0=-c_{\kappa,b}^{-1/2}\sin (x),
\quad
u_3^0=1,
\quad
u_4^0=0,
\]
and
\[
\dot u_1^0
=\im\frac{\gamma_{\kappa,b}}4c_{\kappa,b}^{-1/2}\sin (x),
\qquad
\dot u_2^0
=\im\frac{\gamma_{\kappa,b}}4c_{\kappa,b}^{-1/2}\cos (x),
\qquad
\dot u_3^0=\dot u_4^0=0.
\]
Expanding both the operator and the basis in $\mu$ gives
\begin{align*}
(\mathsf B^s)_{12}
&=\im\frac{\kappa+2b}{c_{\kappa,b}}\mu
+\im\mathcal O(\mu\epsilon^2,\mu^2\epsilon,\mu^3),\\
(\mathsf B^s)_{11}=(\mathsf B^s)_{22}
&=\sigma^s_{\kappa,b}\mu^2
+\mathcal O(\mu^2\epsilon,\mu^3),\\
(\mathsf B^s)_{33}
&=\kappa\mu^2+\mathcal O(\mu^2\epsilon,\mu^3).
\end{align*}
Indeed,
\[
(\mathcal S_{0,0}^{(1)}u_2^0,u_1^0)
=\im\frac{\kappa+2b}{c_{\kappa,b}},
\]
while
\begin{align*}
\sigma^s_{\kappa,b}
={}&\frac12(\mathcal S_{0,0}^{(2)}u_1^0,u_1^0)
+(\mathcal S_{0,0}^{(1)}\dot u_1^0,u_1^0)
+(\mathcal S_{0,0}^{(1)}u_1^0,\dot u_1^0).
\end{align*}
The same identity holds with $u_1^0$ replaced by $u_2^0$, and the
$(3,3)$-coefficient follows from
$\mathcal S_{0,0}^{(1)}1=0$ and
$\mathcal S_{0,0}^{(2)}1=2\kappa$.

It remains to estimate the omitted entries. The linear-in-$\mu$ $(1,3)$-entry
vanishes for every $\epsilon$, because
$\mathcal S_{0,\epsilon}^{(1)}$ maps even functions into $\im$ times odd
functions. The $(2,3)$-entry may be of order $\mu\epsilon$ and is absorbed by
$\mathcal O_2$. Finally, $u_4(\mu,\epsilon)=\mathcal O(\mu\epsilon)$, so every
entry in the fourth row or column is
$\mathcal O(\mu^2\epsilon,\mu^3)$. Self-adjointness and reversibility yield the
remaining bounds. This proves \eqref{eq:B-s-contribution}.
\end{proof}

\begin{proof}[Proof of Proposition~\ref{lem:matrix-representation-G}]
Summing \eqref{eq:B-epsilon-contribution},
\eqref{eq:B-flat-contribution}, Lemma~\ref{lem:Bsharp}, and
\eqref{eq:B-s-contribution}, we obtain
\eqref{eq:reduced-matrix-expansion}, except for the identification of the two
combined scalar coefficients. Since $c_{\kappa,b}^2=1+\kappa+b$,
\[
\frac{c_{\kappa,b}}2+
\frac{\kappa+2b}{c_{\kappa,b}}
=
\frac{1+3\kappa+5b}{2c_{\kappa,b}}
=\frac{e_{12}}2.
\]
Furthermore,
\[
\zeta_{\kappa,b}
-\frac{\gamma_{\kappa,b}c_{\kappa,b}}4
+\sigma^s_{\kappa,b}
=-\frac{e_{22}}8.
\]
Indeed, substituting the definitions of $\gamma_{\kappa,b}$,
$\zeta_{\kappa,b}$, and $\sigma^s_{\kappa,b}$ gives
\[
-\frac{e_{22}}8
=
\frac{15b^2+3\kappa^2+22b\kappa+30b+6\kappa-1}
{8c_{\kappa,b}^3}.
\]
Finally, the sum of finitely many matrices in
$\mathcal R_{\mu,\epsilon}$ still belongs to
$\mathcal R_{\mu,\epsilon}$. This proves the proposition.
\end{proof}

\section{Block-Decoupling and Proof of the Main Result}\label{sec:BD}

In this section we complete the proof of Theorem~\ref{Complete BF thm} by block-diagonalizing the
$4\times4$ Hamiltonian and reversible matrix representing $\mathscr{L}_{\mu,\epsilon}$ on
$\mathcal{V}_{\mu,\epsilon}$.
Unlike the finite-depth case, no singular $\mu^{1/2}$--rescaling is needed: in infinite depth the long-wave
block already has the scaling $G_{11}\sim 1$ and $G_{22}\sim\mu$ by Proposition~\ref{lem:matrix-representation-G}.
We work for $0<\mu<\mu_0$, $0<\epsilon<\epsilon_0$, and all transformations below are
symplectic and reversibility-preserving.

Recall from \eqref{J4Bmuepsilon} and Proposition~\ref{lem:matrix-representation-G} that, in the basis
$\mathcal{G}$,
\[
\mathsf{L}_{\mu,\epsilon}=\mathsf J_4 \mathsf B_{\mu,\epsilon},
\qquad
\mathsf B_{\mu,\epsilon}=
\begin{pmatrix}
E & F\\
F^* & G
\end{pmatrix}.
\]
The matrix $\mathsf{L}_{\mu,\epsilon}$ is Hamiltonian and reversible, with block form
\[
\mathsf{L}_{\mu,\epsilon}
=
\mathsf J_4
\begin{pmatrix}
E & F\\
F^* & G
\end{pmatrix}
=
\begin{pmatrix}
\mathsf J_2E & \mathsf J_2F\\
\mathsf J_2F^* & \mathsf J_2G
\end{pmatrix},
\]
where
\[
E=
\begin{pmatrix}
E_{11} & \im E_{12}\\
-\im E_{12} & E_{22}
\end{pmatrix},
\qquad
F=
\begin{pmatrix}
F_{11} & \im F_{12}\\
\im F_{21} & F_{22}
\end{pmatrix},
\qquad
G=
\begin{pmatrix}
G_{11} & \im G_{12}\\
-\im G_{12} & G_{22}
\end{pmatrix}.
\]
The entries are real analytic in $(\mu,\epsilon)$ and satisfy the expansions of
Proposition~\ref{lem:matrix-representation-G}. In particular
\[
G_{11}=1+\kappa\mu^2+\mathcal{O}(\epsilon^3,\mu\epsilon^2,\mu^2\epsilon,\mu^3),\qquad
G_{22}=\mu+\mathcal{O}(\mu^2\epsilon,\mu^3),
\]
\[
G_{12}-E_{12}=-\frac{e_{12}}{2}\mu+\mathcal{O}(\mu\epsilon^2,\mu^2\epsilon,\mu^3),
\]
\[
E_{11}=e_{11}\epsilon^2-\frac{e_{22}}{8}\mu^2+\mathcal{O}(\epsilon^3,\mu\epsilon^2,\mu^2\epsilon,\mu^3),
\qquad
E_{22}=-\frac{e_{22}}{8}\mu^2+\mathcal{O}(\mu^2\epsilon,\mu^3),
\]
and
\[
F_{11}=\mathcal{O}(\epsilon^3,\mu\epsilon^2,\mu^2\epsilon,\mu^3),
\qquad F_{12},F_{21}=\mathcal{O}(\mu\epsilon,\mu^3),
\qquad
F_{22}=\mathcal{O}(\mu^2\epsilon,\mu^3).
\]
The first step below removes the entry $F_{11}$, which carries no factor $\mu$.

\subsection{First step: removal of \(F_{11}\)}

Let
\[
\mathsf Q:=
\begin{pmatrix}
1&0\\
0&0
\end{pmatrix},
\qquad
\mathsf P:=
\begin{pmatrix}
0&0\\
0&1
\end{pmatrix},
\]
and define
\[
m:= -\frac{F_{11}}{G_{11}}.
\]
Since \(G_{11}=1+\kappa\mu^2+\mathcal{O}(\epsilon^3,\mu\epsilon^2,\mu^2\epsilon,\mu^3)\), the function \(m\) is real analytic and satisfies
\[
m=\mathcal{O}(\epsilon^3,\mu\epsilon^2,\mu^2\epsilon,\mu^3).
\]
Set
\[
\widetilde{Y}:=I_4+m
\begin{pmatrix}
0&-\mathsf P\\
\mathsf Q&0
\end{pmatrix}.
\]
Then $\widetilde{Y}$ is symplectic and reversibility-preserving. Indeed, the matrix
\[
\begin{pmatrix}
0&-\mathsf P\\
\mathsf Q&0
\end{pmatrix}
\]
is Hamiltonian and reversibility-preserving in the present real coordinates,
and the special choice of the rank-one projectors \(\mathsf P,\mathsf Q\) gives directly
\[
[\widetilde{Y}]^*\mathsf J_4\widetilde{Y}=\mathsf J_4.
\]

We conjugate
\[ \mathsf{L}_{\mu,\epsilon}^{(1)}
:=
[\widetilde{Y}]^{-1}\mathsf{L}_{\mu,\epsilon}\widetilde{Y}
=
\mathsf J_4 \mathsf B_{\mu,\epsilon}^{(1)},
\qquad
\mathsf B_{\mu,\epsilon}^{(1)}:=[\widetilde{Y}]^* \mathsf B_{\mu,\epsilon}\widetilde{Y}.
\]
Writing
\[
\mathsf B_{\mu,\epsilon}^{(1)}
=
\begin{pmatrix}
E^{(1)} & F^{(1)}\\
[F^{(1)}]^* & G^{(1)}
\end{pmatrix},
\]
a direct multiplication gives
\[
E^{(1)}
=
E+m(\mathsf QF^*+F \mathsf Q)+m^2 \mathsf QG \mathsf Q,\qquad G^{(1)}
=
G-m(\mathsf PF+F^*\mathsf P)+m^2 \mathsf PE \mathsf P,
\]
and
\[
F^{(1)}
=
F+m(\mathsf QG-E \mathsf P)-m^2 \mathsf QF^*\mathsf P.
\]
Equivalently,
\[
E^{(1)}
=
E+
\begin{pmatrix}
2mF_{11}+m^2G_{11} & -\im m F_{21}\\
\im mF_{21} & 0
\end{pmatrix},\qquad G^{(1)}
=
G+
\begin{pmatrix}
0 & \im mF_{21}\\
-\im mF_{21} & -2mF_{22}+m^2E_{22}
\end{pmatrix},
\]
and
\[
F^{(1)}
=
\begin{pmatrix}
F_{11}+mG_{11}
&
\im(F_{12}+mG_{12}-mE_{12}+m^2F_{21})\\
\im F_{21}
&
F_{22}-mE_{22}
\end{pmatrix}.
\]
By the definition of \(m\), the \((1,1)\)-entry vanishes:
\[
F^{(1)}_{11}=F_{11}+mG_{11}=0.
\]
Moreover the corrections to \(E\) and \(G\) are of higher order, so \(E^{(1)}\)
and \(G^{(1)}\) have the same leading expansions as \(E\) and \(G\). Finally
\[
F^{(1)}
=
\begin{pmatrix}
0 & \im F^{(1)}_{12}\\
\im F^{(1)}_{21} & F^{(1)}_{22}
\end{pmatrix},
\]
with
\[
F^{(1)}_{12}=\mathcal{O}(\mu\epsilon,\mu^3),\qquad F^{(1)}_{21}=\mathcal{O}(\mu\epsilon,\mu^3),\qquad F^{(1)}_{22}=\mathcal{O}(\mu^2\epsilon,\mu^3).
\]
Thus after the first step every remaining off-diagonal entry carries the
necessary factor $\mu$.

\subsection{Second step: a homological equation}

We now remove the new off-diagonal block up to higher order terms. Write
\[
D^{(1)}
:=
\begin{pmatrix}
D^{(1)}_1&0\\
0&D^{(1)}_0
\end{pmatrix}
:=
\begin{pmatrix}
\mathsf J_2E^{(1)}&0\\
0&\mathsf J_2G^{(1)}
\end{pmatrix},\qquad R^{(1)}
:=
\begin{pmatrix}
0&\mathsf J_2F^{(1)}\\
\mathsf J_2[F^{(1)}]^*&0
\end{pmatrix}.
\]
We look for a Hamiltonian and reversibility-preserving matrix
\[
S^{(1)}
=
\mathsf J_4
\begin{pmatrix}
0&\Theta\\
\Theta^*&0
\end{pmatrix},
\qquad
\Theta:=\mathsf J_2 \widetilde{X},
\]
where
\[
\widetilde{X}=
\begin{pmatrix}
x_{11}&ix_{12}\\
ix_{21}&x_{22}
\end{pmatrix},
\qquad x_{ij}\in\mathbb R,
\]
such that
\begin{equation}\label{S1-def}
[S^{(1)},D^{(1)}]+R^{(1)}=0.
\end{equation}
This equation is equivalent to the Sylvester equation
\[
D^{(1)}_1 \widetilde{X}-\widetilde{X}D^{(1)}_0=-\mathsf J_2F^{(1)}.
\]
Writing
\[
a:=G^{(1)}_{12}-E^{(1)}_{12},
\qquad
b:=G^{(1)}_{11},
\qquad
c:=E^{(1)}_{22},
\qquad d:=G^{(1)}_{22},
\qquad
e:=E^{(1)}_{11},
\]
the Sylvester equation is equivalent to the real linear system
\[
\mathsf A
\begin{pmatrix}
x_{11}\\
x_{12}\\
x_{21}\\
x_{22}
\end{pmatrix}
=
\begin{pmatrix}
-F^{(1)}_{21}\\
F^{(1)}_{22}\\
-F^{(1)}_{11}\\
F^{(1)}_{12}
\end{pmatrix},\qquad \text{where}\qquad\mathsf A
=
\begin{pmatrix}
a&b&c&0\\
d&a&0&-c\\
e&0&a&-b\\
0&-e&-d&a
\end{pmatrix}.
\]
Since \(F^{(1)}_{11}=0\), the right-hand side is
\[
\begin{pmatrix}
-F^{(1)}_{21}\\
F^{(1)}_{22}\\
0\\
F^{(1)}_{12}
\end{pmatrix}.
\]

The determinant of \(\mathsf A\) is
\[
\det \mathsf A
=
a^4-2a^2(bd+ce)+(bd-ce)^2.
\]
Using
\[
a=-\frac{e_{12}}{2}\mu+\mathcal{O}(\mu\epsilon^2,\mu^2\epsilon,\mu^3),
\qquad
b=1+\kappa\mu^2+\mathcal{O}(\epsilon^3,\mu\epsilon^2,\mu^2\epsilon,\mu^3),
\qquad
d=\mu+\mathcal{O}(\mu^2\epsilon,\mu^3),
\]
\[
c=-\frac{e_{22}}{8}\mu^2+\mathcal{O}(\mu^2\epsilon,\mu^3),
\qquad
e=e_{11}\epsilon^2-\frac{e_{22}}{8}\mu^2+\mathcal{O}(\epsilon^3,\mu\epsilon^2,\mu^2\epsilon,\mu^3),
\]
we obtain
\[
\det\mathsf A
=
\mu^2(1+\mathcal{O}(\epsilon^2,\mu)).
\]
Hence \(\mathsf A\) is invertible for \(0<|\mu|\ll1\). Moreover the explicit
inverse has the structure
\[
\mathsf A^{-1}
=
\frac1\mu
\begin{pmatrix}
\mathcal{O}(\mu)&1+\mathcal{O}(\epsilon^2,\mu)&\mathcal{O}(\mu^2)&\mathcal{O}(\mu^2)\\
\mathcal{O}(\mu)&\mathcal{O}(\mu)&\mathcal{O}(\mu^3)&\mathcal{O}(\mu^2)\\
\mathcal{O}(\epsilon^2,\mu^2)&\mathcal{O}(\epsilon^2,\mu^2)&\mathcal{O}(\mu)&-1+\mathcal{O}(\epsilon,\mu)\\
\mathcal{O}(\mu\epsilon^2,\mu^3)&\mathcal{O}(\epsilon^2,\mu^2)&\mathcal{O}(\mu)&\mathcal{O}(\mu)
\end{pmatrix}.
\]
Therefore, since
\[
F^{(1)}_{12},F^{(1)}_{21}
=
\mathcal{O}(\mu\epsilon,\mu^3),
\qquad F^{(1)}_{22}
=
\mathcal{O}(\mu^2\epsilon,\mu^3),
\]
we get
\[
x_{11}=\mathcal{O}(\mu\epsilon,\mu^2),
\qquad x_{12}=\mathcal{O}(\mu\epsilon,\mu^3),
\qquad
x_{21}=\mathcal{O}(\epsilon,\mu^2),
\qquad
x_{22}=\mathcal{O}(\mu\epsilon,\mu^3).
\]
In particular
\[
\widetilde{X}=\mathcal{O}(\epsilon,\mu^2).
\]

Notice that no singular factor \(\mu^{-1}\) acts on an \(\mathcal{O}(\epsilon^3)\) term:
that term has already been removed through the exact cancellation
\(F^{(1)}_{11}=0\). This is the essential reason for performing the first
block-decoupling step before solving the Sylvester equation.

We now define
\[
\mathsf{L}_{\mu,\epsilon}^{(2)}
:=
e^{S^{(1)}}\mathsf{L}_{\mu,\epsilon}^{(1)}e^{-S^{(1)}}.
\]
Since \(S^{(1)}\) is Hamiltonian and reversibility-preserving, the matrix
\(\mathsf{L}_{\mu,\epsilon}^{(2)}\) is again Hamiltonian and reversible.

Using the Lie expansion and the homological equation
\[
[S^{(1)},D^{(1)}]+R^{(1)}=0,
\]
we get
\[
\mathsf{L}_{\mu,\epsilon}^{(2)}
=
D^{(1)}
+
\frac12[S^{(1)},R^{(1)}]
+
\frac12
\int_0^1
(1-\tau^2)
e^{\tau S^{(1)}}
\operatorname{ad}_{S^{(1)}}^2(R^{(1)})
e^{-\tau S^{(1)}}\,d\tau.
\]
The commutator \(\frac12[S^{(1)},R^{(1)}]\) is block-diagonal. More precisely,
\[
\frac12[S^{(1)},R^{(1)}]
=
\begin{pmatrix}
\mathsf J_2\widetilde E&0\\
0&\mathsf J_2\widetilde G
\end{pmatrix},
\]
where
\[
\widetilde E
=
\operatorname{Sym}\bigl(\mathsf J_2 \widetilde{X} \mathsf J_2[F^{(1)}]^*\bigr),
\qquad
\widetilde G
=
\operatorname{Sym}\bigl([\widetilde{X}]^*F^{(1)}\bigr).
\]
Here
\[
\operatorname{Sym}(A):=\frac12(A+A^*).
\]

We now estimate the entries componentwise. Since \(F^{(1)}_{11}=0\), a direct
multiplication gives
\[
\mathsf J_2\widetilde{X}\mathsf J_2[F^{(1)}]^*
=
\begin{pmatrix}
x_{21}F^{(1)}_{12}
&
\im\bigl(x_{22}F^{(1)}_{21}+x_{21}F^{(1)}_{22}\bigr)
\\
\im x_{11}F^{(1)}_{12}
&
x_{12}F^{(1)}_{21}-x_{11}F^{(1)}_{22}
\end{pmatrix}.
\]
Therefore
\[
\widetilde E_{11}
=
\mathcal{O}(\mu\epsilon^2,\mu^3\epsilon,\mu^5),
\qquad\widetilde E_{12}
=
\mathcal{O}(\mu^2\epsilon^2,\mu^3\epsilon,\mu^5),
\qquad
\widetilde E_{22}
=
\mathcal{O}(\mu^2\epsilon^2,\mu^4\epsilon,\mu^5,\mu^6).
\]
These estimates give the required control of the diagonal correction generated
by the first Lie transform.

Similarly,
\[
[\widetilde{X}]^*F^{(1)}
=
\begin{pmatrix}
x_{21}F^{(1)}_{21}
&
\im\bigl(x_{11}F^{(1)}_{12}-x_{21}F^{(1)}_{22}\bigr)
\\
\im x_{22}F^{(1)}_{21}
&
x_{22}F^{(1)}_{22}+x_{12}F^{(1)}_{12}
\end{pmatrix},
\]
and hence
\[
\widetilde G_{11}
=
\mathcal{O}(\mu\epsilon^2,\mu^3\epsilon,\mu^5),
\qquad\widetilde G_{12}
=
\mathcal{O}(\mu^2\epsilon^2,\mu^3\epsilon,\mu^5),
\qquad
\widetilde G_{22}
=
\mathcal{O}(\mu^2\epsilon^2,\mu^4\epsilon,\mu^5).
\]
Thus \(E^{(1)}+\widetilde E\) and \(G^{(1)}+\widetilde G\) have the same
leading form as \(E\) and \(G\).

It remains to estimate the off-diagonal part generated by the integral
remainder. Write the real parameters $p,r,s,u$ and $q,t$ for the entries of
\[
\widetilde E=
\begin{pmatrix}
p&\im q\\
-\im q&r
\end{pmatrix},
\qquad
\widetilde G=
\begin{pmatrix}
s&\im t\\
-\im t&u
\end{pmatrix}.
\]
From the componentwise bounds above,
\[
p,s=\mathcal{O}(\mu\epsilon^2,\mu^3\epsilon,\mu^5),
\qquad
q,t=\mathcal{O}(\mu^2\epsilon^2,\mu^3\epsilon,\mu^5),
\qquad
r,u=\mathcal{O}(\mu^2\epsilon^2,\mu^4\epsilon,\mu^5).
\]
Since
\[
\widehat F
=
2\bigl(\Theta \mathsf J_2\widetilde G-\widetilde E \mathsf J_2\Theta\bigr),
\qquad
\Theta=\mathsf J_2\widetilde{X},
\]
each entry of $\widehat F$ is a finite bilinear combination of the $p,q,r,s,t,u$
and the $x_{ij}$. Using
\[
x_{11}=\mathcal{O}(\mu\epsilon,\mu^2),\qquad
x_{12},x_{22}=\mathcal{O}(\mu\epsilon,\mu^3),
\qquad
x_{21}=\mathcal{O}(\epsilon,\mu^2),
\]
we obtain
\[
\widehat F_{11},\widehat F_{12},\widehat F_{21}
=
\mathcal{O}(\mu^2\epsilon^3,\mu^3\epsilon^2,\mu^5\epsilon,\mu^7),
\qquad
\widehat F_{22}
=
\mathcal{O}(\mu^3\epsilon^3,\mu^4\epsilon^2,\mu^5\epsilon,\mu^7).
\]
Since $e^{\tau S^{(1)}}=I+\mathcal{O}(\epsilon,\mu^2)$, the same estimates hold for the
off-diagonal block produced by the integral remainder. Consequently
\[
\mathsf{L}_{\mu,\epsilon}^{(2)}
=
\mathsf J_4
\begin{pmatrix}
E^{(2)}&F^{(2)}\\
[F^{(2)}]^*&G^{(2)}
\end{pmatrix},
\]
where $E^{(2)}$ and $G^{(2)}$ have the same leading expansions as $E$
and $G$, and
\begin{align}\label{F2-infinite}
F^{(2)}_{11},F^{(2)}_{12},F^{(2)}_{21}
=
\mathcal{O}(\mu^2\epsilon^3,\mu^3\epsilon^2,\mu^5\epsilon,\mu^7),\qquad F^{(2)}_{22}=
\mathcal{O}(\mu^3\epsilon^3,\mu^4\epsilon^2,\mu^5\epsilon,\mu^7).
\end{align}
In particular
\[
F^{(2)}=\mu^2\,\mathcal{O}(\epsilon^3,\mu\epsilon^2,\mu^3\epsilon,\mu^5),
\]
and the remaining coupling is two orders smaller in $\epsilon$ than
$F^{(1)}$.

\subsection{Complete block-decoupling}

We now completely remove the remaining off-diagonal block. Write
\[
\mathsf{L}_{\mu,\epsilon}^{(2)}
=
D^{(2)}+R^{(2)},
\]
where
\[
D^{(2)}
=
\begin{pmatrix}
\mathsf J_2E^{(2)}&0\\
0&\mathsf J_2G^{(2)}
\end{pmatrix},
\qquad
R^{(2)}
=
\begin{pmatrix}
0&\mathsf J_2F^{(2)}\\
\mathsf J_2[F^{(2)}]^*&0
\end{pmatrix}.
\]
We look for a Hamiltonian and reversibility-preserving matrix
\[
S^{(2)}
=
\mathsf J_4
\begin{pmatrix}
0&\Theta^{(2)}\\
[\Theta^{(2)}]^*&0
\end{pmatrix},
\qquad
\Theta^{(2)}=\mathsf J_2X^{(2)},
\]
such that
\[
e^{\mu S^{(2)}}\mathsf{L}_{\mu,\epsilon}^{(2)}e^{-\mu S^{(2)}}
\]
is block-diagonal. Equivalently,
\[
\Pi_{\mathrm{off}}
\left(
e^{\mu S^{(2)}}
(D^{(2)}+R^{(2)})
e^{-\mu S^{(2)}}
\right)
=0,
\]
where \(\Pi_{\mathrm{off}}\) denotes the projection onto block-off-diagonal
matrices.

Expanding the left-hand side gives
\[
R^{(2)}+\mu[S^{(2)},D^{(2)}]
+\mu^2\mathcal R(S^{(2)})=0,
\]
where \(\mathcal R(S^{(2)})\) is analytic in \((\mu,\epsilon,S^{(2)})\) and at
least quadratic in \(S^{(2)}\). Since \([S^{(2)},D^{(2)}]\) is block-off-diagonal,
this is equivalent to the nonlinear homological equation
\[
[S^{(2)},D^{(2)}]
=
-\frac1\mu R^{(2)}-\mu\,\mathcal R(S^{(2)}),
\]
and therefore to the real Sylvester system
\[
\mathsf A^{(2)}\vec x
=
\vec f(\mu,\epsilon,\vec x),
\qquad
\vec f(\mu,\epsilon,\vec x)=\mu\vec\nu(\mu,\epsilon)+\mu^2\vec g(\mu,\epsilon,\vec x),
\]
where \(\mathsf A^{(2)}\) has the same form as \(\mathsf A\), with the entries
of \(E^{(2)}\) and \(G^{(2)}\), and
\[
\mu^2\vec\nu(\mu,\epsilon)
:=
\begin{pmatrix}
-F^{(2)}_{21}\\
F^{(2)}_{22}\\
-F^{(2)}_{11}\\
F^{(2)}_{12}
\end{pmatrix},
\qquad
\mu^2\vec g=-\mu^2\,\mathcal R(S^{(2)}).
\]
Since \(E^{(2)}\) and \(G^{(2)}\) have the same leading expansions as \(E\) and
\(G\), one still has
\[
\det \mathsf A^{(2)}
=
\mu^2(1+\mathcal{O}(\epsilon^2,\mu)),
\qquad
[\mathsf A^{(2)}]^{-1}
=
\frac1\mu \mathsf B(\mu,\epsilon),
\]
where \(\mathsf B\) is analytic and bounded for small \((\mu,\epsilon)\).
By \eqref{F2-infinite},
\[
F^{(2)}=\mu^2\,\mathcal{O}(\epsilon^3,\mu\epsilon^2,\mu^3\epsilon,\mu^5),
\]
so
\[
\vec\nu(\mu,\epsilon)
=
\mathcal{O}(\epsilon^3,\mu\epsilon^2,\mu^3\epsilon,\mu^5),
\]
and no term of the form \(\epsilon^k/\mu\) appears on the right-hand side.
The fixed point equation \(\mathsf A^{(2)}\vec x=\vec f\) is equivalent to
\[
\vec x
=
\mathsf B(\mu,\epsilon)\vec\nu(\mu,\epsilon)
+\mu\,\mathsf B(\mu,\epsilon)\vec g(\mu,\epsilon,\vec x),
\]
and hence
\[
X^{(2)}
=
\mathcal{O}(\epsilon^3,\mu\epsilon^2,\mu^3\epsilon,\mu^5).
\]
The third component of \(\mu\vec\nu\) is
\(-F^{(2)}_{11}/\mu=\mathcal{O}(\mu\epsilon^3,\mu^2\epsilon^2,\mu^4\epsilon,\mu^6)\),
and is therefore coupled through the non-singular coefficient
\(G^{(2)}_{11}=1+\kappa\mu^2+\mathcal{O}(\epsilon^3,\mu\epsilon^2,\mu^2\epsilon,\mu^3)\), exactly as in the first Sylvester
estimate.

Thus this equation has a unique small analytic solution by the implicit function theorem. The
corresponding matrix \(S^{(2)}\) is Hamiltonian and reversibility-preserving.

Consequently
\[
\mathsf{L}_{\mu,\epsilon}^{(3)}
:=
e^{\mu S^{(2)}}
\mathsf{L}_{\mu,\epsilon}^{(2)}
e^{-\mu S^{(2)}}
=
D^{(2)}+P^{(2)},
\]
where \(P^{(2)}\) is block-diagonal, Hamiltonian and reversible. Moreover
\[
P^{(2)}
=
\mathcal{O}(\mu F^{(2)}X^{(2)})
\]
is of higher order than the corrections already estimated above. Therefore
\[
\mathsf{L}_{\mu,\epsilon}^{(3)}
=
\begin{pmatrix}
\mathsf J_2E^{(3)}&0\\
0& \mathsf J_2G^{(3)}
\end{pmatrix},
\]
where \(E^{(3)}\) and \(G^{(3)}\) are self-adjoint and reversibility-preserving
and have the same leading expansions as \(E\) and \(G\).

We have therefore conjugated the original Hamiltonian and reversible matrix
\(\mathsf{L}_{\mu,\epsilon}\) to a block-diagonal Hamiltonian and reversible
matrix
\[
\mathsf{L}_{\mu,\epsilon}^{(3)}
=
\begin{pmatrix}
\mathsf J_2E^{(3)}&0\\
0&\mathsf J_2G^{(3)}
\end{pmatrix}.
\]
This finishes the symplectic and reversible block-decoupling construction.  We now prove
Theorem~\ref{Complete BF thm} by extracting the two spectral pairs from the two diagonal blocks.
\begin{proof}[Proof of Theorem~\ref{Complete BF thm}]
The operator $\mathcal{L}_{\mu,\epsilon}$ on $\mathcal{V}_{\mu,\epsilon}$ is represented by
\[
\im c_{\kappa,b}\mu\,I_4+
\exp(\mu S^{(2)})\exp(S^{(1)})\,[\widetilde{Y}]^{-1}\,\mathsf{L}_{\mu,\epsilon}\,\widetilde{Y}\,
\exp(-S^{(1)})\exp(-\mu S^{(2)})
=
\im c_{\kappa,b}\mu\,I_4+
\begin{pmatrix}
\mathsf J_2 E^{(3)} & 0\\
0 & \mathsf J_2 G^{(3)}
\end{pmatrix},
\]
where $\widetilde{Y}$, $S^{(1)}$, and $S^{(2)}$ are the symplectic reversible transformations constructed above,
and $E^{(3)}$, $G^{(3)}$ have the same leading expansions as $E$, $G$.
Setting $U:=\im c_{\kappa,b}\mu\,I_2+\mathsf J_2 E^{(3)}$ and $S:=\im c_{\kappa,b}\mu\,I_2+\mathsf J_2 G^{(3)}$, a direct multiplication with
$\mathsf J_2$ yields \eqref{U-S-infinite}.

For the Benjamin--Feir pair, write
\[
U=
\begin{pmatrix}
\im\alpha&\beta\\
\gamma&\im\alpha
\end{pmatrix},
\]
where
\[
\alpha
=
c_{\kappa,b}\mu-E^{(3)}_{12}
=
\frac12\breve{c}_{\kappa,b}\mu
+
\mathcal{O}(\mu\epsilon^2,\mu^2\epsilon,\mu^3),
\]
\[
\beta
=
E^{(3)}_{22}
=
-\frac{e_{22}}{8}\mu^2+\mathcal{O}(\mu^2\epsilon,\mu^3),
\]
and
\[
\gamma
=
-E^{(3)}_{11}
=
-e_{11}\epsilon^2+\frac{e_{22}}{8}\mu^2
+
\mathcal{O}(\epsilon^3,\mu\epsilon^2,\mu^2\epsilon,\mu^3).
\]
Define the Benjamin--Feir discriminant by
\begin{equation}\label{BF-discriminant-block}
\Delta_{\mathrm{BF}}(\kappa,b;\mu,\epsilon)
:=
-64\,\frac{E^{(3)}_{11}E^{(3)}_{22}}{\mu^2}.
\end{equation}
Then $\beta\gamma=\frac{\mu^2}{64}\Delta_{\mathrm{BF}}(\kappa,b;\mu,\epsilon)$, and the leading expansions give
\begin{equation}\label{BFDF}
\Delta_{\mathrm{BF}}(\kappa,b;\mu,\epsilon)
=8e_{22}e_{11}\epsilon^2-e_{22}^2\mu^2
+\mathcal{O}(\epsilon^3,\mu\epsilon^2,\mu^2\epsilon,\mu^3).
\end{equation}
Since $\lambda_1^\pm=\im\alpha\pm\sqrt{\beta\gamma}$, we obtain
\eqref{lambda1-infinite} with
\[
\lambda_1^\pm
=
\im \frac12\breve{c}_{\kappa,b}\mu
+\im\,\mathcal{O}(\mu\epsilon^2,\mu^2\epsilon,\mu^3)
\pm
\frac{\mu}{8}\sqrt{\Delta_{\mathrm{BF}}(\kappa,b;\mu,\epsilon)}.
\]
In particular $\operatorname{Re}\lambda_1^\pm\neq0$ if and only if $\Delta_{\mathrm{BF}}>0$.

For the long-wave block, $S$ has the form
\[
S
=
\begin{pmatrix}
\im c_{\kappa,b}\mu+\im \delta_0&G^{(3)}_{22}\\
-G^{(3)}_{11}&\im c_{\kappa,b}\mu+\im\delta_0
\end{pmatrix},
\qquad
\delta_0=\mathcal{O}(\mu\epsilon^2,\mu^2\epsilon,\mu^3),
\]
where
\[
G^{(3)}_{11}
=
1+\kappa\mu^2+\mathcal{O}(\epsilon^3,\mu\epsilon^2,\mu^2\epsilon,\mu^3),
\qquad
G^{(3)}_{22}
=
\mu+\mathcal{O}(\mu^2\epsilon,\mu^3).
\]
Thus
\[
\det(S-\lambda I)
=
(\lambda-\im c_{\kappa,b}\mu-\im \delta_0)^2
+
G^{(3)}_{11}G^{(3)}_{22},
\]
with
\[
G^{(3)}_{11}G^{(3)}_{22}
=
\mu+\mathcal{O}(\mu\epsilon^2,\mu^2\epsilon,\mu^3),
\]
which is positive for $\mu,\e$ sufficiently small. Equivalently, after absorbing the harmless higher-order terms into the final
remainder,
\[
\lambda_0^\pm
=
\im c_{\kappa,b}\mu+\im\,\mathcal{O}(\mu\epsilon^2,\mu^2\epsilon,\mu^3)
\pm \im\sqrt{\mu}\,(1+\mathcal{O}(\epsilon,\mu)),
\]
which is \eqref{lambda0-infinite}.
\end{proof}

\appendix

\section{Existence of Stokes waves}
In this appendix we provide a short proof of the existence of Stokes waves by using the Crandall--Rabinowitz bifurcation theorem (see for example \cite[Theorem 3.1]{BMV2}).

 \begin{theorem}[Existence of Stokes waves]
Given $\alpha\geq 0$, $m + \frac12 \in \N$, $m > \frac{11}{2}$ and $(\kappa,b) \in (\R_{\geq 0} \times\R_{> 0}) \setminus \fR$, with  $\fR$  in \eqref{def:infinite-R}. There exist
$\e_*=\e_*(\alpha,m,\kappa,b)>0$ and a unique family of solutions to \eqref{travelingWWstokes}
\[
(\eta_\e(x),\psi_\e(x),c_\e)\in H^{\alpha,m}_{\mathtt{ev}} (\T)\times H^{\alpha,m}_{\mathtt{odd}}(\T)\times \R,
\qquad |\e|<\e_*
\]
such that
\[
c_0:=\sqrt{1+\kappa+b},
\]
where  $ H^{\alpha,m}_{\mathtt{ev}}(\T) $ (resp. $ H^{\alpha,m}_{\mathtt{odd}}(\T) $) denotes the space of even (resp. odd) functions in $ H^{\alpha,m}(\T) $. 
\end{theorem}

\begin{proof}

It follows from \cite[Theorem 1.2]{BMV2} that there exists an $\epsilon_0 >0 $ such that the following map is analytic
\begin{align*}
    F:  &(H^{\alpha,m}_{\mathtt{ev}} (\T) \cap B^{\alpha, m}(\epsilon_0))\times H^{\alpha,m}_{\mathtt{odd}}(\T)\times \R \to H^{\alpha,m - 1}_{\mathtt{odd}}(\T) \times H^{\alpha,m - 4}_{\mathtt{ev}} (\T), \\
    F(\eta,\psi,c) &:= \begin{bmatrix}
        c \, \eta_x+G(\eta)\psi \\
        c \, \psi_x -  \eta - \dfrac{\psi_x^2}{2} + 
\dfrac{1}{2(1+\eta_x^2)} \big( G(\eta) \psi + \eta_x \psi_x \big)^2  +\kappa\sigma(\eta)-b\left(\partial^2_s \sigma(\eta)+\frac{1}{2}\sigma(\eta)^3\right).
    \end{bmatrix}
\end{align*}

It is not hard to show that $F(0,0,c) = 0$ for all $c \in \R$. A direct calculation gives that 
\begin{align*}
    \de_{\eta,\psi}F(0,0,c): & H^{\alpha,m}_{\mathtt{ev}} (\T)\times H^{\alpha,m}_{\mathtt{odd}}(\T) \to H^{\alpha,{m - 1}}_{\mathtt{odd}}(\T) \times H^{\alpha,{m - 4}}_{\mathtt{ev}} (\T), \\
    \de_{\eta,\psi}F(0,0,c) \begin{bmatrix}
        \hat{\eta} \\
        \hat{\psi}
    \end{bmatrix} & := \begin{bmatrix}
        c \partial_x & |D| \\
        -1 + \kappa \partial_x^2 - b \partial_x^4 & c \partial_x
    \end{bmatrix} \begin{bmatrix}
        \hat{\eta} \\
        \hat{\psi}
    \end{bmatrix}.
\end{align*}

Our assumptions guarantee that with the choice $c = c_0$ given above, we have
\[
\mathrm{ker} (\de_{\eta,\psi}F(0,0,c_0)) = \mathrm{span} \left\{ \begin{bmatrix}
     \cos(x) \\
    \sqrt{1 + \kappa + b} \sin(x)
\end{bmatrix} \right\}
\]
and 
\begin{align*}
\mathrm{Rg} (\de_{\eta,\psi}F(0,0,c_0)) = \mathrm{span} \left\{ \begin{bmatrix}
   0 \\
    1
\end{bmatrix} \right\} &\oplus  \mathrm{span} \left\{ \begin{bmatrix}
     \sin(x) \\
    \sqrt{1 + \kappa + b} \cos(x)
\end{bmatrix} \right\} \\
&\oplus 
     \overline{\mathrm{span} \left\{ \begin{bmatrix}
    \sin(jx) \\
    0
\end{bmatrix},\begin{bmatrix}
    0 \\
    \cos(jx)
\end{bmatrix}, j \ge 2 \right\}}^{H^{\alpha,{m - 1}}_{\mathtt{odd}}(\T) \times H^{\alpha,{m - 4}}_{\mathtt{ev}} (\T)}.
\end{align*}
So it is clear that the kernel of $\de_{\eta,\psi}F(0,0,c_0)$ is $1$-dimensional, the range of $\de_{\eta,\psi}F(0,0,c_0)$ is closed and the cokernel of $\de_{\eta,\psi}F(0,0,c_0)$ is also $1$-dimensional. Finally we show that $\partial_c \de_{\eta,\psi}F(0,0,c_0) [u^*]$ does not belong to the range of $\de_{\eta,\psi}F(0,0,c_0)$, where $u^*$ is a nonzero element in the $\mathrm{ker} (\de_{\eta,\psi}F(0,0,c_0))$ and here we can choose it to be
\[
u^* := \begin{bmatrix}
     \cos(x) \\
    \sqrt{1 + \kappa + b} \sin(x)
\end{bmatrix}.
\]
By a direct calculation we obtain
\[
\partial_c \de_{\eta,\psi}F(0,0,c_0) [u^*] = \begin{bmatrix}
         \partial_x & 0 \\
        0 &  \partial_x
    \end{bmatrix} \begin{bmatrix}
     \cos(x) \\
    \sqrt{1 + \kappa + b} \sin(x)
\end{bmatrix} = \begin{bmatrix}
     -\sin(x) \\
    \sqrt{1 + \kappa + b} \cos(x)
\end{bmatrix},
\]
which is clearly not in the range of $\de_{\eta,\psi}F(0,0,c_0)$.
Then we can conclude that  $F$ satisfies all the conditions of Crandall--Rabinowitz bifurcation theorem.
\end{proof}

\section{{Expansion of the Stokes waves in infinite depth}}\label{sec:App2}

{In this Appendix we provide the expansions  \eqref{exp:Sto}-\eqref{expcoef}, \eqref{expfe}, 
\eqref{pino1fd}-\eqref{aino2fd}.}
\\[1mm]
\noindent
{\bf Proof  of \eqref{exp:Sto}-\eqref{expcoef}.}
Writing
 \begin{equation}\label{etapsic}
\begin{aligned}
 & \eta_\e(x) = \e \eta_1(x) + \e^2 \eta_2(x) + \mathcal{O}(\e^3) \, , \\
 &  \psi_\e(x) = \e \psi_1(x) + \e^2 \psi_2(x) + \mathcal{O}(\e^3) \, ,  
  \end{aligned}
\qquad \quad c_\e = c_{\kappa,b} + \e c_1 + \e^2 c_2+ \mathcal{O}(\e^3) \, ,  
\end{equation}
where  $\eta_i$ is $even(x)$ and $\psi_i$ is $odd(x)$ for $i=1,2$, 
we solve order by order in $ \e $ the equations \eqref{travelingWWstokes},  
that we rewrite  as
\begin{equation}
\label{Sts}
\begin{cases}
-c \, \psi_x +  \eta   + \dfrac{\psi_x^2}{2} - 
\dfrac{\eta_x^2}{2(1+\eta_x^2)} ( c  -  \psi_x )^2-\kappa\sigma(\eta)+b\left(\partial^2_s \sigma(\eta)+\frac{1}{2}\sigma(\eta)^3\right)  = 0 \\
c \, \eta_x+G(\eta)\psi = 0 	\, ,
\end{cases}
\end{equation}
having substituted $G(\eta)\psi $ with $-c \, \eta_x $  in the first equation.
 We expand the Dirichlet-Neumann operator 
$ G(\eta)=  G_0+ G_1(\eta) + G_2(\eta) + \mathcal{O}(\eta^3)  $
where, according to \cite[formula (2.14)]{CS}, 
\begin{equation} \label{expDiriNeu}
\begin{aligned}
 G_0 & := |D|  \,, \\
 G_1(\eta) & := D \eta D -  G_0\eta G_0 = D \eta D -  |D|\eta  |D|, \\
 G_2(\eta) & := -\frac12  \Big( G_0 
 {\eta}^2 |D|^2 +|D|^2{\eta}^2 G_0 - 2G_0\eta G_0\eta G_0\Big) = -\frac12  \Big(|D| 
 {\eta}^2 |D|^2 +|D|^2{\eta}^2 |D| - 2|D|\eta |D|\eta |D|\Big)\, .
\end{aligned}
\end{equation}
{\bf First order in $ \e $.}  Substituting the expansions in \eqref{etapsic} into \eqref{Sts}, we get the linear system 
\begin{equation}\label{cB0}
 \left\{\begin{matrix} -c_{\kappa,b} (\psi_1)_x + \eta_1-\kappa (\eta_1)_{xx} +b (\eta_1)_{xxxx}= 0 \\
 c_{\kappa,b} (\eta_1)_x + G_0\psi_1 =0 \, ,  \end{matrix}\right.
 \quad  \text{i.e.} \, \vet{\eta_1}{\psi_1} \in \mathit{Ker }\cB_0  \text{ with } \cB_0 := \begin{bmatrix} 1-\kappa\pa^2_{x}+b\pa^4_{x} & -c_{\kappa,b}\pa_x \\ c_{\kappa,b}\pa_x & G_0  \end{bmatrix},  
\end{equation}
where $\eta_1$ is $\mathit{even}(x)$ and $\psi_1$ is $\mathit{odd}(x)$.
\begin{lemma}\label{lem:B0}
If $(\kappa,b)\notin\mathfrak R$,  the kernel of 
 the linear operator $\cB_0$ in \eqref{cB0} is one dimensional and given by
\begin{equation}\label{chk}
\mathit{Ker }\cB_0= \mathit{span}\,\Big\{\vet{\cos(x)}{(1+b+\kappa)^{\frac{1}{2}}\sin(x)} \Big\}.
\end{equation}
\end{lemma}
\begin{proof} The action of $\cB_0$ on each subspace spanned by $\footnotesize{\,\Big\{\vet{\cos(kx)}{0}, \vet{0}{\sin(kx)}\Big\}} $, $k\in \N$, is represented by the $2\times 2$ matrix $\footnotesize{ \begin{bmatrix} 1+\kappa k^2+b k^4 & -c_{\kappa,b} k \\ -c_{\kappa,b} k & k  \end{bmatrix}}$. Its determinant is given by
 $$ (1+\kappa k^2+b k^4)k - c_{\kappa,b}^2 k^2\stackrel{\eqref{exp:Sto}}{=} k^2 \Big(\frac{(1+\kappa k^2+b k^4)}{k} -(1+b+\kappa) \Big)$$
 and, provided $(\kappa,b)\notin\mathfrak R$, 
 vanishes  if and only if $k=1$.
Indeed, for $k$ large enough, the determinant goes asymptotically to $+\infty$ so it is uniformly bounded away from zero,  whereas if for some $k \in \N$ it vanishes, it implies that $(\kappa,b) \in \mathfrak R_k \subset \mathfrak R$, absurd.
  For $k=1$ we obtain the kernel of $\cB_0$ given in \eqref{chk}. For $k=0$ it has no kernel since $\psi_1(x)$ is odd.
\end{proof}

We set
$
\eta_1(x) := \cos(x)$, $\psi_1(x) :=(1+b+\kappa)^{\frac{1}{2}} \sin(x) 
$
in agreement with 
\eqref{exp:Sto}. 
\\[1mm]
{\bf Second order in $ \e $.} 
By \eqref{Sts}, and since 
$ c_{\kappa,b}^2 (\eta_1)_x^2 = (G_0\psi_1)^2  $, we get  the linear system
\begin{equation}\label{syslin2}
 \cB_0 \vet{\eta_2}{\psi_2} = \vet{c_1(\psi_1)_x-\frac12 (\psi_1)_x^2 + \frac12 (G_0\psi_1)^2 }{-c_1(\eta_1)_x - G_1(\eta_1)\psi_1} \, , 
\end{equation}
where  $\cB_0$ is the self-adjoint operator in \eqref{cB0}. 
System \eqref{syslin2} admits a solution if and only if its right-hand term is orthogonal to the kernel  of $\cB_0$ in \eqref{chk}, namely
\begin{equation}\label{orth1}
 \Big(\vet{c_1(\psi_1)_x-\frac12 (\psi_1)_x^2 + \frac12 (G_0\psi_1)^2 }{-c_1(\eta_1)_x - G_1(\eta_1)\psi_1}\;,\;\vet{\cos(x)}{(1+b+\kappa)^{\frac{1}{2}}\sin(x)}\Big)=0 \, . 
\end{equation}
In view of the first order expansion \eqref{exp:Sto}, \eqref{expDiriNeu},
  it results $
 [G_0\psi_1](x)=  (1+b+\kappa)^{\frac{1}{2}}\sin(x)$
so that \eqref{orth1} implies
 $c_1=0$, in agrement with \eqref{exp:Sto}. Equation
 \eqref{syslin2} reduces to 
\begin{align}\label{sisto2}
\begin{bmatrix} 1-\kappa\pa^2_{x}+b\pa^4_{x} & -c_{\kappa,b}\pa_x \\ c_{\kappa,b}\pa_x & G_0  \end{bmatrix}
\vet{\eta_2}{\psi_2} 
  = \vet{-\frac{1+b+\kappa}{2}\cos(2x) }{ 0}.
\end{align}
Setting $ \eta_2 = \eta_2^{[0]} +  \eta_2^{[2]} \cos(2x) $ and  $ \psi_2 =
 \psi_2^{[2]}  \sin (2x) $, system 
\eqref{sisto2} amounts to 
\begin{align*}
\left\{ \begin{matrix} \eta_2^{[0]} +\big( (1+4\kappa+16 b)\eta_2^{[2]} -2c_{\kappa,b}  \psi_2^{[2]} \big) \cos(2x)  =  -\frac{1+b+\kappa}{2}\cos(2x)   \\ (-c_{\kappa,b} \eta_2^{[2]}   +   \psi_2^{[2]})\sin(2x)  =   0 \, , \end{matrix}\right.
\end{align*}
which yields $\eta_2^{[0]}=0$ together with the coefficients $\eta_2^{[2]}$ and $\psi_2^{[2]}$ in \eqref{expcoef}. 

\noindent
{\bf Third order in $ \e $.} 
It remains to determine 
$ c_2 $ in \eqref{expc2}.
 We get the linear system 
\begin{equation}\label{syslin3}
 \cB_0 \vet{\eta_3}{\psi_3} = \vet{c_2(\psi_1)_x 
 - (\psi_1)_x (\psi_2)_x - (\eta_1)_x^2 (\psi_1)_x c_{\kappa,b} - \frac{3\kappa}{2}(\eta_1)_x^2(\eta_1)_{xx} +
 (\eta_1)_x (\eta_2)_x c_{\kappa,b}^2+ b\mathfrak{D}(\eta_1)
 }{-c_2(\eta_1)_x - G_1(\eta_1)\psi_2- G_1(\eta_2)\psi_1 - G_2(\eta_1)\psi_1} \, , 
\end{equation}
where 
\begin{align}
    \mathfrak{D}(\eta):=\frac{5}{2}\left((\eta)^3_{xx}+(\eta)^2_x(\eta)_{xxxx}\right)+10(\eta)_x(\eta)_{xx}(\eta)_{xxx}.
\end{align}
System \eqref{syslin3} has 
a solution if and only if the right hand side is orthogonal to the Kernel of
$ \cB_0 $ given in \eqref{chk}. This condition determines uniquely $ c_2 $.
Denoting  $\Pi_1$ the $L^2$-orthogonal projector on span$\, \{\cos(x),\sin(x)\} $, it results 
\begin{align*} 
& c_2 (\psi_1)_x = c_2  (1+b+\kappa)^{\frac{1}{2}} \cos(x)\, , \quad c_2 (\eta_1)_x = -c_2 \sin(x) \, , \quad \Pi_1[ (\psi_1)_x (\psi_2)_x] = 
\psi_2^{[2]}  (1+b+\kappa)^{\frac{1}{2}} \cos(x)\, ,\\ 
& \Pi_1 [c_{\kappa,b} (\eta_1)_x^2 (\psi_1)_x ] = \tfrac14 (1+b+\kappa)\cos(x)\, , \quad \Pi_1\!\left[\frac{3\kappa}{2}(\eta_1)_x^2(\eta_1)_{xx}\right] = -\frac{3\kappa}{8}\cos(x)\,,\\
& \Pi_1[c_{\kappa,b}^2 (\eta_1)_x (\eta_2)_x] = \eta_2^{[2]}(1+b+\kappa) \cos(x) ,\quad\Pi_1[\mathfrak{D}(\eta_1)]=\frac{5}{4}\cos(x)\, ,
\end{align*}
and, in view of \eqref{expDiriNeu}, and \eqref{exp:Sto}, \eqref{expcoef}, 
\begin{align*}
 \Pi_1[ G_1(\eta_1)\psi_2]  = 0 \, ,  \quad 
  \Pi_1[G_2(\eta_1)\psi_1] 
 =   \frac{c_{\kappa,b}}{4} \sin(x) \, ,\quad \Pi_1[G_1(\eta_2)\psi_1] = c_{\kappa,b} \eta_2^{[2]}\sin(x) \, . 
\end{align*}
Imposing the orthogonality condition gives $c_2$ as in   \eqref{expc2}.

\begin{proof}[Proof of \eqref{expfe}] 
 We expand the function $\mathfrak{p}(x) = \epsilon \mathfrak{p}_1(x) + \epsilon^2 \mathfrak{p}_2(x) + \mathcal{O}(\epsilon^3)$ defined by the fixed point equation in infinite depth
$$
\mathfrak{p}(x) = \mathfrak{H}[\eta_\epsilon(x+\mathfrak{p}(x))]
$$
where $\mathfrak{H}$ denotes the Hilbert transform.

We use the expansion of the Stokes wave $\eta_\epsilon(x)$ in \eqref{exp:Sto} and substituting the expansions into the fixed point equation, we get
$$
\epsilon \mathfrak{p}_1 + \epsilon^2 \mathfrak{p}_2 = \mathfrak{H} \left[ \epsilon \cos(x+\mathfrak{p}) + \epsilon^2 \eta_2^{[2]} \cos(2(x+\mathfrak{p})) \right] + \mathcal{O}(\epsilon^3).
$$
Using the Taylor expansion $\cos(x+\mathfrak{p}) = \cos(x) - \mathfrak{p}\sin(x) + \mathcal{O}(\mathfrak{p}^2)$, we solve order by order.

\noindent \textbf{First order in $\epsilon$:}
\begin{align}
  \mathfrak{p}_1(x) = \mathfrak{H}[\cos(x)] = \sin(x).\label{pfra1}  
\end{align}

\noindent \textbf{Second order in $\epsilon$:}
\begin{align}\label{pfra2}
\mathfrak{p}_2(x) = \mathfrak{H} \left[ -\mathfrak{p}_1(x)\sin(x) + \eta_2^{[2]} \cos(2x) \right].
\end{align}

Using $\mathfrak{p}_1(x) = \sin(x)$, we calculate 
\begin{align*}
\mathfrak{p}_2(x) &= \mathfrak{H} \left[ -\frac{1}{2} + \left( \frac{1}{2} + \eta_2^{[2]} \right) \cos(2x) \right] = \frac{2-13b-\kappa}{2(1-14b-2\kappa)} \sin(2x).
\end{align*} 
\end{proof}
\begin{proof}[Proof of Lemma \ref{lem:pa.exp}] 
In view of \eqref{exp:Sto}-\eqref{expcoef}, the expansions of the functions $B$, $V$  in \eqref{espV} and \eqref{espB} are
\begin{align}
 B&= : \e B_1(x) + \e^2 B_2(x) + \mathcal{O}(\e^3) = \e c_{\kappa,b} \sin(x)-\e^{2}\frac{c_{\kappa,b}(16b+4\kappa+1)}{2(14b+2\kappa-1)} \sin(2x)+ \mathcal{O}(\e^3) \label{espB1}  
  \end{align}
  and
  \begin{align}
 V&= : \e V_1(x) + \e^2 V_2(x)  + \mathcal{O}(\e^3) =  \e c_{\kappa,b} \cos(x)+\e^{2} \left(\frac{c_{\kappa,b}}{2}-\frac{c_{\kappa,b}(16b+4\kappa+1)}{2(14b+2\kappa-1)}\cos(2x)\right)+\mathcal{O}(\e^3)
  \, .\label{espV1} 
\end{align}

In  view of  \eqref{def:pa}, denoting derivatives w.r.t $x$ with  a prime and suppressing dependence on $x$ when trivial, we have
\begin{align}
 c_{\kappa,b}+p_\e(x)  &= ( c_{\kappa,b}+\e^2 c_2 - V(x) - V'(x)\mathfrak{p}(x)+\mathcal{O}(\e^3)) (1-\mathfrak{p}'(x)+(\mathfrak{p}'(x))^2+\mathcal{O}(\e^3))\notag \\
&=  c_{\kappa,b} + \e \underbrace{(- V_1 - c_{\kappa,b} \mathfrak{p}_1')}_{=: p_1}+\e^2 \underbrace{\big( c_2 + V_1\mathfrak{p}_1' - V_2 - V_1'\mathfrak{p}_1 - c_{\kappa,b} \mathfrak{p}_2' + c_{\kappa,b} (\mathfrak{p}_1')^2 \big)}_{=:p_2} + \, \mathcal{O}(\e^3) \, .\label{pino12imp}
\end{align} 
Similarly by \eqref{def:pa}
\begin{align}
 1+a_\e(&x) : = \frac{1}{1+\mathfrak{p}_x(x)} - ( c_{\kappa,b}+p_\e(x))B_x(x+\mathfrak{p}(x)) \notag \\
   = &  1+\e\underbrace{\big(-\mathfrak{p}_1'-  c_{\kappa,b} B_1'\big)}_{=:a_1} +\e^2\underbrace{\big((\mathfrak{p}_1')^2-\mathfrak{p}_2'-  c_{\kappa,b} B_2'- c_{\kappa,b} B_1''\mathfrak{p}_1(x)+ B_1' V_1 +  c_{\kappa,b}  B_1'\mathfrak{p}_1' \big)}_{=:a_2}+\mathcal{O}(\e^3)\, . \label{aino12imp}
 \end{align}

By \eqref{exp:Sto}, \eqref{pfra1}, \eqref{pfra2}, \eqref{espB1}, \eqref{espV1} we 
deduce that the functions $p_1 $, $p_2 $, $a_1 $, $a_2 $ in \eqref{pino12imp} and \eqref{aino12imp} have an expansion as in \eqref{pino1fd}-\eqref{aino2fd}.

Finally the operator $\beta_\e$ in \eqref{def:Sigma g} expands as
$\beta_\e = \pa^4_{x} + \e \beta_1 + \e^2 \beta_2 + \mathcal{O}(\e^3)$
where
\begin{align}\label{Sigma 1 expansion}
    \beta_1:=& -5\mathfrak{p}_1'\pa^4_x-10\mathfrak{p}_1''\pa^3_x-10\mathfrak{p}_1'''\pa^2_x-5\mathfrak{p}_1''''\pa_x-\mathfrak{p}_1''''' 
\end{align}
and
\begin{equation} \label{Sigma 2}
\begin{aligned}
\beta_2:=&\left(F^{[4]}_2(x)-5\mathfrak{p}_2'+15(\mathfrak{p}_1')^2\right)\pa^4_x+\left(F^{[3]}_2(x)+60\mathfrak{p}_1'\mathfrak{p}_1''-10\mathfrak{p}_2''\right)\pa^3_x\\ 
    &+\left(F^{[2]}_2(x)+45(\mathfrak{p}_1'')^2+60\mathfrak{p}_1'\mathfrak{p}_1'''-10\mathfrak{p}_2'''\right)\pa^2_x+\left(F^{[1]}_2(x)+30\mathfrak{p}_1'\mathfrak{p}_1''''+60\mathfrak{p}_1''\mathfrak{p}_1'''-5\mathfrak{p}_2''''\right)\pa_x\\
    &+6\mathfrak{p}_1'\mathfrak{p}_1'''''+15\mathfrak{p}_1''\mathfrak{p}_1''''+10(\mathfrak{p}_1''')^2-\mathfrak{p}_2'''''.
\end{aligned}
\end{equation}
Inserting the expansions of $F^{[1]}_2,F^{[2]}_2,F^{[3]}_2,F^{[4]}_2$ in \eqref{F12 F22 F32 F42} and those of $\mathfrak{p}_1$, $\mathfrak{p}_2$ in \eqref{pfra1}, \eqref{pfra2} proves the claimed expansions 
in \eqref{betaj} with coefficients in \eqref{b110}--\eqref{b224}.

Similarly, the operator $\tau_\e$ in \eqref{tau operator} expands as 
$\tau_\e = \pa^2_{x} + \e \tau_1 + \e^2 \tau_2 + \mathcal{O}(\e^3)$
where
\begin{align}\label{tau 1 expansion}
    \tau_1:=& -3\mathfrak{p}_1'\pa^2_x-3\mathfrak{p}_1''\pa_x-\mathfrak{p}_1''' 
\end{align}
and
\begin{equation} \label{tau 2 expansion}
\begin{aligned}
\tau_2:=&\left(\Sigma^{[2]}_2(x)-3\mathfrak{p}_2'+6(\mathfrak{p}_1')^2\right)\pa^2_x+\left(\Sigma^{[1]}_2(x)-3\mathfrak{p}_2''+12\mathfrak{p}_1'\mathfrak{p}_1''\right)\pa_x+3(\mathfrak{p}_1'')^2+4\mathfrak{p}_1'\mathfrak{p}_1'''-\mathfrak{p}_2'''.
\end{aligned}
\end{equation}
Inserting the expansions of $\Sigma^{[1]}_2,\Sigma^{[2]}_2$ in \eqref{Sigma coefficients} and those of $\mathfrak{p}_1$, $\mathfrak{p}_2$ in \eqref{pfra1}, \eqref{pfra2} proves the claimed expansions 
in \eqref{tauj} with coefficients in \eqref{tau11}--\eqref{tau22}.
\end{proof}

\section{The generalized kernel of the periodic fiber}
\label{app:periodic-generalized-kernel}
\color{black}
In this appendix we prove
Proposition~\ref{prop:periodic-generalized-kernel}.  The argument follows
the symmetry construction of \cite[Theorem~4.1]{NS}.  The capillary and
bending forces do not alter the argument, since they are invariant under
horizontal translations and under vertical translations of the entire
infinite-depth fluid domain.

\begin{proof}[Proof of Proposition~\ref{prop:periodic-generalized-kernel}]
Introduce a Bernoulli parameter $P\in\mathbb R$ in the stationary system by
writing
\begin{equation}\label{eq:stationary-system-with-P}
\mathscr F(\eta,\psi,c,P)
:=
\begin{bmatrix}
 c\eta_x+G(\eta)\psi
\\[1mm]
 c\psi_x-\eta-\dfrac{\psi_x^2}{2}
 +\dfrac{\bigl(G(\eta)\psi+\eta_x\psi_x\bigr)^2}
 {2(1+\eta_x^2)}
 +\kappa\sigma(\eta)-bF(\eta)+P
\end{bmatrix}.
\end{equation}
Here the stationary equation is considered on the full periodic
function space, without imposing a zero-mean condition on $\eta$.
For every constant $P\in\mathbb R$, vertical translation of the
infinite-depth fluid domain gives
\[
G(\eta+P)\psi=G(\eta)\psi.
\]
Moreover,
\[
\sigma(\eta+P)=\sigma(\eta),
\qquad
F(\eta+P)=F(\eta),
\]
because the capillary and bending forces depend only on derivatives
of $\eta$. Since
\[
-(\eta_a+P)+P=-\eta_a,
\]
the functions
\begin{equation}\label{eq:two-parameter-Stokes-family}
\eta(a,P)=\eta_a+P,
\qquad
\psi(a,P)=\psi_a,
\qquad
c(a,P)=c_a
\end{equation}
satisfy
\[
\mathscr F(\eta(a,P),\psi(a,P),c(a,P),P)=0.
\]
In particular,
\begin{equation}\label{eq:P-derivatives-Stokes-family}
\partial_P\eta(a,P)=1,
\qquad
\partial_P\psi(a,P)=0,
\qquad
\partial_Pc(a,P)=0.
\end{equation}

Let $\mathscr A_\epsilon$ be the real linearized operator defined in
\eqref{first linear eq}. Differentiating the stationary system gives
\begin{equation}\label{eq:stationary-derivative-linearized-operator}
D_{(\eta,\psi)}
\mathscr F(\eta_\epsilon,\psi_\epsilon,c_\epsilon,0)
=
\mathscr A_\epsilon.
\end{equation}
By construction of the good-unknown and flattening transformations,
\begin{equation}\label{eq:conjugacy-original-flattened}
\mathscr L_{0,\epsilon}
=
\mathscr T_\epsilon
\mathscr A_\epsilon
\mathscr T_\epsilon^{-1}.
\end{equation}

\smallskip
\noindent
{\bf Gauge and translation eigenvectors.}
The stationary equations are invariant under adding a constant to the
surface potential.  Equivalently,
\[
\mathscr A_\epsilon
\begin{bmatrix}0\\1\end{bmatrix}=0.
\]
Since $\mathscr T_\epsilon(0,1)^{\mathsf T}=(0,1)^{\mathsf T}$,
\eqref{eq:conjugacy-original-flattened} gives
\[
\mathscr L_{0,\epsilon}U_1=0.
\]

Next differentiate
$\mathscr F(\eta_\epsilon,\psi_\epsilon,c_\epsilon,0)=0$ with respect to
$x$.  Horizontal translation invariance yields
\[
\mathscr A_\epsilon
\begin{bmatrix}
\partial_x\eta_\epsilon\\
\partial_x\psi_\epsilon
\end{bmatrix}=0.
\]
Applying \eqref{eq:conjugacy-original-flattened}, we obtain
\[
\mathscr L_{0,\epsilon}\widetilde U_2=0.
\]
Since $U_1$ also belongs to the kernel, the normalization
\eqref{eq:normalized-U2} implies
\[
\mathscr L_{0,\epsilon}U_2=0.
\]

We now verify the analyticity at $\epsilon=0$.  The Stokes branch,
$B_\epsilon$, $\mathfrak p_\epsilon$, and $\mathfrak P_\epsilon$ depend
real-analytically on $\epsilon$.  Since
$\eta_0=\psi_0=0$, one has
$\widetilde U_2(0)=0$, and therefore
$\epsilon^{-1}\widetilde U_2(\epsilon)$ has a real-analytic extension to
$\epsilon=0$.  Using
\[
\eta_\epsilon=\epsilon\cos (x)+\mathcal O(\epsilon^2),
\qquad
\psi_\epsilon=\epsilon c_{\kappa,b}\sin (x)
+\mathcal O(\epsilon^2),
\qquad
B_\epsilon=\mathcal O(\epsilon),
\qquad
\mathfrak p_\epsilon=\mathcal O(\epsilon),
\]
in \eqref{eq:T-epsilon-generalized-kernel}, we find
\[
\frac{1}{\epsilon}\widetilde U_2(\epsilon)
=
\begin{bmatrix}
-\sin (x)\\ c_{\kappa,b}\cos (x)
\end{bmatrix}
+
\mathcal O(\epsilon).
\]
The leading second component has zero spatial average, and hence
\eqref{eq:normalized-U2} yields \eqref{eq:U2-leading-value}.

The parity of the Stokes branch gives
$\eta_\epsilon$ even and $\psi_\epsilon$ odd.  Moreover,
$B_\epsilon$ and $\mathfrak p_\epsilon$ are odd, while
$1+\mathfrak p_{\epsilon,x}$ is even and the map
$x\mapsto x+\mathfrak p_\epsilon(x)$ is odd.  Consequently,
$\widetilde U_2$ and $U_2$ have the componentwise parity
\[
\widetilde U_2,\ U_2
=
\begin{bmatrix}\mathit{odd}(x)\\\mathit{even}(x)\end{bmatrix},
\]
whereas $U_1=(0,1)^{\mathsf T}$.  They are therefore real and
anti-reversible, proving \eqref{eq:Uj-reversibility} for $j=1,2$.
The zero-average identity \eqref{eq:U2-zero-average} follows directly from
\eqref{eq:normalized-U2}.

\smallskip
\noindent
{\bf Amplitude and Bernoulli generalized vectors.}
Differentiate
$\mathscr F(\eta(a,0),\psi(a,0),c(a,0),0)=0$ with respect to $a$ and then
set $a=\epsilon$.  Since
\[
\partial_c\mathscr F(\eta_\epsilon,\psi_\epsilon,c_\epsilon,0)
=
\begin{bmatrix}
\partial_x\eta_\epsilon\\
\partial_x\psi_\epsilon
\end{bmatrix},
\]
we obtain
\[
\mathscr A_\epsilon
\begin{bmatrix}
\partial_a\eta_a|_{a=\epsilon}\\
\partial_a\psi_a|_{a=\epsilon}
\end{bmatrix}
=
-\partial_\epsilon c_\epsilon
\begin{bmatrix}
\partial_x\eta_\epsilon\\
\partial_x\psi_\epsilon
\end{bmatrix}.
\]
Applying $\mathscr T_\epsilon$ and using
\eqref{eq:conjugacy-original-flattened}, we get
\[
\mathscr L_{0,\epsilon}U_3
=-\partial_\epsilon c_\epsilon\,\widetilde U_2.
\]

Similarly, differentiating
$\mathscr F(\eta(a,P),\psi(a,P),c(a,P),P)=0$ with respect to $P$ at
$(a,P)=(\epsilon,0)$ and using
\eqref{eq:P-derivatives-Stokes-family}, we obtain
\[
\mathscr A_\epsilon
\begin{bmatrix}1\\0\end{bmatrix}
=-
\begin{bmatrix}0\\1\end{bmatrix}.
\]
Because $\mathscr T_\epsilon(0,1)^{\mathsf T}=U_1$, it follows that
\[
\mathscr L_{0,\epsilon}U_4=-U_1.
\]
The vectors $U_3$ and $U_4$ are real and have componentwise parity
$(\mathit{even}(x),\mathit{odd}(x))^{\mathsf T}$; hence they are reversible.
This completes the proof of
\eqref{eq:generalized-kernel-relations} and
\eqref{eq:Uj-reversibility}.

At $\epsilon=0$, the previous definitions and the first-order Stokes
expansion give
\[
U_1(0)=f_0^-,
\qquad
U_2(0)=c_{\kappa,b}^{1/2}f_1^-,
\qquad
U_3(0)=c_{\kappa,b}^{1/2}f_1^+,
\qquad
U_4(0)=f_0^+.
\]
These vectors are linearly independent.  Their real-analytic dependence on
$\epsilon$ implies that they remain linearly independent for all
sufficiently small $|\epsilon|$.

Relations \eqref{eq:generalized-kernel-relations} imply that each $U_j$
belongs to the generalized kernel of $\mathscr L_{0,\epsilon}$ and is
annihilated by its square.  Since zero lies inside the Riesz contour
$\Gamma$, all four vectors belong to $\mathcal V_{0,\epsilon}$.  By
Lemma~\ref{kato thm},
$\dim\mathcal V_{0,\epsilon}=4$; therefore they form a basis of
$\mathcal V_{0,\epsilon}$, which proves
\eqref{eq:V-generalized-kernel-span}.  The relations also yield
\eqref{eq:L-square-zero-on-V}.  Hence the restriction of
$\mathscr L_{0,\epsilon}$ to $\mathcal V_{0,\epsilon}$ has only the
eigenvalue zero, whose algebraic multiplicity is four.
\end{proof}
\color{black}
\section{Expansion of the Kato Basis}\label{secA1}

In this appendix we prove Lemma \ref{lem:4.2}. We provide the expansion of the basis $f_k^{\pm}(\mu,\epsilon) = U_{\mu,\epsilon} f_k^{\pm}$, $k = 0,1$, in \eqref{F basis set and f}, where $f_k^{\pm}$ defined in \eqref{eigenfunc of mathcall L00 2} belong to the subspace $\mathcal{V}_{0,0} := \operatorname{Rg}(P_{0,0})$. We first Taylor-expand the transformation operators $U_{\mu,\epsilon}$ defined in \eqref{U transformation operators}. We denote $\partial_{\epsilon}$ with a prime and $\partial_{\mu}$ with a dot.
The next lemma follows as \cite[Lemma A.1]{BMV1}
\begin{lemma} \label{U mu epsilon P00}
    The first jets of $U_{\mu,\epsilon} P_{0,0}$ are
\begin{align} \label{A1}
    &U_{0,0} P_{0,0} = P_{0,0}, \quad\quad U'_{0,0} P_{0,0} = P'_{0,0} P_{0,0}, \quad\quad \dot{U}_{0,0} P_{0,0} = \dot{P}_{0,0} P_{0,0},\\ \label{A2}
    &\dot{U}'_{0,0} P_{0,0} = \left( \dot{P}'_{0,0} - \frac{1}{2} P_{0,0} \dot{P}'_{0,0} \right) P_{0,0}, 
\end{align}
where
\begin{align} \label{A3}
    P'_{0,0} &= \frac{1}{2\pi \im} \oint_{\Gamma} (\mathscr{L}_{0,0} - \lambda)^{-1} \mathscr{L}'_{0,0} (\mathscr{L}_{0,0} - \lambda)^{-1} d\lambda,  \\ \label{A4}
    \dot{P}_{0,0} &= \frac{1}{2\pi \im} \oint_{\Gamma} (\mathscr{L}_{0,0} - \lambda)^{-1} \dot{\mathscr{L}}_{0,0} (\mathscr{L}_{0,0} - \lambda)^{-1} d\lambda, 
\end{align}
and
\begin{align} \label{eq:dotPprime-term1}
    \dot{P}'_{0,0} &= -\frac{1}{2\pi \im} \oint_{\Gamma} (\mathscr{L}_{0,0} - \lambda)^{-1} \dot{\mathscr{L}}_{0,0} (\mathscr{L}_{0,0} - \lambda)^{-1} \mathscr{L}'_{0,0} (\mathscr{L}_{0,0} - \lambda)^{-1} d\lambda \\ \label{eq:dotPprime-term2}
    &\quad -\frac{1}{2\pi \im} \oint_{\Gamma} (\mathscr{L}_{0,0} - \lambda)^{-1} \mathscr{L}'_{0,0} (\mathscr{L}_{0,0} - \lambda)^{-1} \dot{\mathscr{L}}_{0,0} (\mathscr{L}_{0,0} - \lambda)^{-1} d\lambda \\ \label{eq:dotPprime-term3}
    &\quad + \frac{1}{2\pi \im} \oint_{\Gamma} (\mathscr{L}_{0,0} - \lambda)^{-1} \dot{\mathscr{L}}'_{0,0} (\mathscr{L}_{0,0} - \lambda)^{-1} d\lambda.
\end{align}
The operators $\mathscr{L}'_{0,0}$ and $\dot{\mathscr{L}}_{0,0}$ are
\begin{align} \label{A6}
    \mathscr{L}'_{0,0}=\begin{bmatrix}
        \partial_x\circ p_1(x)  &0\\
        -a_1(x)+\kappa \tau_1-b\beta_1 &p_1(x)\circ \partial_x
    \end{bmatrix},~~\dot{\mathscr{L}}_{0,0}=\begin{bmatrix}
        0 &\mathrm{sgn}(D)+\Pi_0\\
        2i\kappa \partial_x-4ib\partial_x^3 &0
    \end{bmatrix},
\end{align}
where $\mathrm{sgn}(D)$, $\Pi_0$ are defined in \eqref{D+mu}, \eqref{sgn D} and $a_1(x)$, $p_1(x)$, $\beta_1$ are given in Lemma \ref{lem:pa.exp}.

The operator $\dot{\mathscr{L}}'_{0,0}$ is
\begin{align} \label{eq:mathscr_L_prime_dot_00}
    \dot{\mathscr{L}}'_{0,0}
=
\begin{bmatrix}
\im p_1(x) & 0\\
\im\kappa\big(\tau_{1,1}(x)+2\tau_{1,2}(x)\partial_x\big)
-\im b\displaystyle\sum_{j=1}^4 j\,b_{1,j}(x)\partial_x^{j-1}
& \im p_1(x)
\end{bmatrix}
\end{align}
with $ b_{1,j}(x)$ are given in Lemma \ref{lem:pa.exp}. 
\end{lemma}

By the Lemma \ref{U mu epsilon P00}, we have the Taylor expansion
\begin{equation} \label{the expandsion of f sigma mu}
    \begin{aligned}
    f_k^\sigma (\mu, \epsilon) &= f_k^\sigma + \epsilon P'_{0,0} f_k^\sigma + \mu \dot{P}_{0,0} f_k^\sigma  + \mu \epsilon \big( \dot{P}'_{0,0} - \frac{1}{2} P_{0,0} \dot{P}'_{0,0} \big) f_k^\sigma + \mathcal{O}(\mu^2, \epsilon^2).
\end{aligned}
\end{equation}
In order to compute the vectors $P'_{0,0} f_k^\sigma$ and $\dot{P}_{0,0} f_k^\sigma$ using \eqref{A3} and \eqref{A4}, it is useful to know the action of $(\mathscr{L}_{0,0} - \lambda)^{-1}$ on the vectors
\begin{equation} \label{A10}
    \begin{aligned}
        f^+_k&:=\vet{(1+b+\kappa)^{-1/4}\cos(kx)}{(1+b+\kappa)^{1/4}\sin(kx)},~~f^-_k:=\vet{-(1+b+\kappa)^{-1/4}\sin(kx)}{(1+b+\kappa)^{1/4}\cos(kx)},\\
        f^+_{-k}&:=\vet{(1+b+\kappa)^{-1/4}\cos(kx)}{-(1+b+\kappa)^{1/4}\sin(kx)},~~f^-_{-k}:=\vet{(1+b+\kappa)^{-1/4}\sin(kx)}{(1+b+\kappa)^{1/4}\cos(kx)},~~k\in\mathbb{N}.
    \end{aligned}
\end{equation}

\begin{lemma}
    The space $Y=H^4(\mathbb{T},\mathbb{C})\times H^1(\mathbb{T},\mathbb{C})$ decomposes as $Y = \mathcal{V}_{0,0} \oplus \mathcal{U} \oplus \mathcal{W}_{Y}$, with
\begin{equation*}
    \mathcal{W}_{Y} = \overline{\bigoplus_{k=2}^{\infty} \mathcal{W}_k}^{Y}
\end{equation*}
where the subspaces $\mathcal{V}_{0,0}, \mathcal{U}$, and $\mathcal{W}_k$, defined below, are invariant under $\mathscr{L}_{0,0}$ and the following properties hold:

\begin{itemize}
\item[(i)] $\mathcal{V}_{0,0} = \text{span}\{f_1^+, f_1^-, f_0^+, f_0^-\}$ is the generalized kernel of $\mathscr{L}_{0,0}$. For any $\lambda \neq 0$ the operator $\mathscr{L}_{0,0} - \lambda : \mathcal{V}_{0,0} \to \mathcal{V}_{0,0}$ is invertible and
    \begin{equation} \label{A11-0}
        (\mathscr{L}_{0,0} - \lambda)^{-1} f_1^+ = -\frac{1}{\lambda} f_1^+, \quad (\mathscr{L}_{0,0} - \lambda)^{-1} f_1^- = -\frac{1}{\lambda} f_1^-,
    \end{equation}
    \begin{equation} \label{A11}
        (\mathscr{L}_{0,0} - \lambda)^{-1} f_0^- = -\frac{1}{\lambda} f_0^-,
    \end{equation}
    \begin{equation} \label{A12}
        (\mathscr{L}_{0,0} - \lambda)^{-1} f_0^+ = -\frac{1}{\lambda} f_0^+ + \frac{1}{\lambda^2} f_0^- .
    \end{equation}

    \item[(ii)] $\mathcal{U} := \text{span}\{ f_{-1}^+, f_{-1}^- \}$. For any $\lambda \neq \pm \im \,2c_{\kappa,b}$ the operator $\mathscr{L}_{0,0} - \lambda : \mathcal{U} \to \mathcal{U}$ is invertible and
    \begin{equation} \label{A13}
    \begin{aligned}
        (\mathscr{L}_{0,0} - \lambda)^{-1} f_{-1}^+ &= \frac{1}{\lambda^2 + 4c_{\kappa,b}^2} \left(-\lambda f_{-1}^+ + 2c_{\kappa,b} f_{-1}^-\right),\\
        (\mathscr{L}_{0,0} - \lambda)^{-1} f_{-1}^- &= \frac{1}{\lambda^2 + 4c_{\kappa,b}^2} \left(-2c_{\kappa,b} f_{-1}^+ - \lambda f_{-1}^-\right).
    \end{aligned}
    \end{equation}

    \item[(iii)] Each subspace $\mathcal{W}_k := \text{span}\{ f_k^+, f_k^-, f_{-k}^+, f_{-k}^- \}$ is invariant under $\mathscr{L}_{0,0}$. Let
    \begin{equation*}
        \mathcal{W}_{L^2} = \overline{\bigoplus_{k=2}^{\infty} \mathcal{W}_k}^{ L^2}.
    \end{equation*}
    For any $|\lambda| < \delta(\tth)$ small enough, the operator $\mathscr{L}_{0,0} - \lambda : \mathcal{W}_{Y} \to \mathcal{W}_{L^2}$ is invertible and for any $f \in \mathcal{W}_{L^2}$,
    \begin{equation} \label{A14}
        (\mathscr{L}_{0,0} - \lambda)^{-1} f = \left( c_{\kappa,b}^2 \partial^2_{x} + |D| (1-\kappa \partial_x^2+b \pa^4_{x}) \right)^{-1} \begin{bmatrix} c_{\kappa,b} \partial_x & -|D|  \\ 1-\kappa \partial_x^2+b \pa^4_{x} & c_{\kappa,b} \partial_x \end{bmatrix} f
        + \lambda \varphi_f(\lambda, \cdot),
    \end{equation}
    for some analytic function $\lambda \mapsto \varphi_f(\lambda, \cdot) \in Y=H^4(\mathbb{T},\mathbb{C})\times H^1(\mathbb{T},\mathbb{C})$.
\end{itemize}
\end{lemma}

\begin{proof}
    By inspection the spaces $\mathcal{V}_{0,0}, \mathcal{U}$ and $\mathcal{W}_k$ are invariant under $\mathscr{L}_{0,0}$ and, by Fourier series, they decompose $Y=H^4(\mathbb{T},\mathbb{C})\times H^1(\mathbb{T},\mathbb{C})$. Formulas \eqref{A11}-\eqref{A12} follow using that $f_1^+, f_1^-, f_0^-$ are in the kernel of $\mathscr{L}_{0,0}$, and $\mathscr{L}_{0,0} f_0^+ = -f_0^-$. Formula \eqref{A13} follows using that $\mathscr{L}_{0,0} f_{-1}^+ = -2c_{\kappa,b} f_{-1}^-$ and $\mathscr{L}_{0,0} f_{-1}^- = 2c_{\kappa,b} f_{-1}^+$. Let us prove item $(iii)$. Let $\mathcal{W} := \mathcal{W}_{Y}$. The operator $(\mathscr{L}_{0,0} - \lambda \mathrm{Id})|_{\mathcal{W}}$ is invertible for any 
$$\lambda \notin \{ \pm i kc_{\kappa,b} \pm i\sqrt{|k|\left(1+\kappa k^2 + b k^4\right)},\quad k \geq 2,\quad k \in \mathbb{N} \}$$
and 
$$
(\mathscr{L}_{0,0}|_{\mathcal{W}})^{-1} = \left( c_{\kappa,b}^2 \partial_{x}^2 + |D| (1-\kappa \partial_x^2+b \pa_{x}^4) \right)^{-1} \begin{bmatrix} c_{\kappa,b} \partial_x & -|D|  \\ 1-\kappa \partial_x^2+b \pa_{x}^4 & c_{\kappa,b} \partial_x \end{bmatrix} \Big|_{\mathcal{W}}.
$$
By Neumann series, for any $\lambda$ such that 
$$|\lambda| \| (\mathscr{L}_{0,0}|_{\mathcal{W}})^{-1} \|_{\mathcal{L}(\mathcal{W}, Y)} < 1$$
we have 
$$
(\mathscr{L}_{0,0}|_{\mathcal{W}} - \lambda)^{-1} = (\mathscr{L}_{0,0}|_{\mathcal{W}})^{-1} (\text{Id} - \lambda (\mathscr{L}_{0,0}|_{\mathcal{W}})^{-1})^{-1} 
= (\mathscr{L}_{0,0}|_{\mathcal{W}})^{-1} \sum_{k \geq 0} ((\mathscr{L}_{0,0}|_{\mathcal{W}})^{-1} \lambda)^k.
$$
Formula \eqref{A14} follows with 
$$\varphi_f(\lambda, x) := (\mathscr{L}_{0,0}|_{\mathcal{W}})^{-1} \sum_{k \geq 1} \lambda^{k-1} [(\mathscr{L}_{0,0}|_{\mathcal{W}})^{-1}]^k f.$$ 
\end{proof}

To prove Lemma \ref{lem:4.2}, we shall also use the following formulas obtained by \eqref{A6}, \eqref{D+mu}, and \eqref{eigenfunc of mathcall L00 2}:
\begin{equation}\label{A15}
\begin{aligned}
&\mathscr{L}'_{0,0} f^+_1
=\vet{2c_{\kappa,b}^{\frac12}\sin(2x)}{3(5b+\kappa)c_{\kappa,b}^{-\frac12}\cos(2x)},\qquad\mathscr{L}'_{0,0} f^-_1
=\vet{2c_{\kappa,b}^{\frac12}\cos(2x)}{-3(5b+\kappa)c_{\kappa,b}^{-\frac12}\sin(2x)},\\
&\mathscr{L}'_{0,0} f^+_0
=\vet{2c_{\kappa,b}\sin(x)}{2c_{\kappa,b}^2\cos(x)},
\qquad
\mathscr{L}'_{0,0}f^-_0=\vet{0}{0},\\
&\dot{\mathscr{L}}_{0,0} f^+_1
=\vet{-\im c_{\kappa,b}^{\frac12}\cos(x)}{-\im2(\kappa+2b)c_{\kappa,b}^{-\frac12}\sin(x)},
\qquad
\dot{\mathscr{L}}_{0,0} f^-_1
=\vet{\im c_{\kappa,b}^{\frac12}\sin(x)}{-\im2(\kappa+2b)c_{\kappa,b}^{-\frac12}\cos(x)},\\
&\dot{\mathscr{L}}_{0,0} f^+_0=\vet{0}{0},
\qquad
\dot{\mathscr{L}}_{0,0}f^-_0=\vet{1}{0}.
\end{aligned}
\end{equation}
We finally calculate $P'_{0,0} f^\sigma_k$ and $\dot{P}_{0,0} f^\sigma_k$.
\begin{lemma}
One has
\[
P'_{0,0} f^+_1
=\vet{\alpha_{\kappa,b}\cos(2x)}{\beta_{\kappa,b}\sin(2x)},
\qquad
P'_{0,0} f^-_1
=\vet{-\alpha_{\kappa,b}\sin(2x)}{\beta_{\kappa,b}\cos(2x)},
\qquad
P'_{0,0} f^+_0=\delta_{\kappa,b}f^+_{-1},
\]
\[
P'_{0,0} f^-_0=\vet00,
\qquad
\dot P_{0,0}f^+_0=\vet00,
\qquad
\dot P_{0,0}f^-_0=\vet00,
\qquad
\dot P_{0,0}f^+_1=i\frac{\gamma_{\kappa,b}}{4}f^-_{-1},
\qquad
\dot P_{0,0}f^-_1=i\frac{\gamma_{\kappa,b}}{4}f^+_{-1},
\]
where
\[
\alpha_{\kappa,b}
=
-\frac{2-13b-\kappa}{14b+2\kappa-1}\,c_{\kappa,b}^{-1/2},
\qquad
\beta_{\kappa,b}
=
-\frac{1+b+\kappa}{14b+2\kappa-1}\,c_{\kappa,b}^{1/2},
\qquad\gamma_{\kappa,b}
=
\frac{1-\kappa-3b}{1+\kappa+b},
\qquad
\delta_{\kappa,b}=c_{\kappa,b}^{1/2}.
\]
\end{lemma}
\begin{proof}
We first calculate $P'_{0,0} f^+_1$. By the definitions of the operators and the unperturbed basis $f^+_1$, we deduce
\begin{align*}
\mathscr{L}'_{0,0} f^+_1 =\vet{2c_{\kappa,b}^{\frac12}\sin(2x)}{3(5b+\kappa)c_{\kappa,b}^{-\frac12}\cos(2x)} \in \mathcal{W}.
\end{align*}
Note that unlike the finite-depth gravity-capillary case, the constant spatial average term perfectly cancels out due to the infinite-depth hydroelastic dispersion properties. Hence, $\mathscr{L}'_{0,0} f^+_1$ belongs purely to the high-frequency subspace $\mathcal{W}$. Thus, we have
\begin{align*}
P'_{0,0} f^+_1=-\frac{1}{2\pi \im}\oint_{\Gamma} \frac{1}{\lambda} \left(\mathscr{L}_{0,0}-\lambda\right)^{-1}
\vet{2 c_{\kappa,b}^{1/2} \sin(2x)}{3(5b+\kappa)c_{\kappa,b}^{-1/2} \cos(2x)} d\lambda.
\end{align*}
by the residue theorem. Similarly, one may calculate $P'_{0,0} f^-_1$.
Since $\mathscr{L}'_{0,0}f^-_0=\vet{0}{0}$, one trivially has $P'_{0,0}f^-_0=0$.
Next, we calculate $P'_{0,0}f^+_0$. By direct evaluation, we get
\begin{align*}
    \mathscr{L}'_{0,0} f^+_0 = \vet{2c_{\kappa,b}\sin(x)}{2c_{\kappa,b}^2\cos(x)} = 2c_{\kappa,b}^{3/2} f^-_{-1}.
\end{align*}
Then we can write the projection as
\[
P'_{0,0} f^+_0 = -\frac{1}{2\pi \im} \oint_{\Gamma} \frac{1}{\lambda} (\mathscr{L}_{0,0} - \lambda)^{-1} \left( 2c_{\kappa,b}^{3/2} f^-_{-1} \right) d\lambda.
\]
Using the action of the resolvent on the resonant subspace $\mathcal{U}$, we have
\[
P'_{0,0} f^+_0 = -\frac{1}{2\pi \im} \oint_{\Gamma}
\frac{2c_{\kappa,b}^{3/2}}{\lambda (\lambda^2 + 4c_{\kappa,b}^2)} \left( -2c_{\kappa,b} f^+_{-1} - \lambda f^-_{-1} \right) d\lambda.
\]
Applying the Cauchy residue theorem at the simple pole $\lambda=0$, we obtain
\[
P'_{0,0} f^+_0 = c_{\kappa,b}^{1/2} f^+_{-1} = \delta_{\kappa,b} f^+_{-1},
\]
which gives the third identity.

Now, we calculate $\dot{P}_{0,0} f^+_1$. First, the derivative operator yields $\dot{\mathscr{L}}_{0,0} f^+_1 = -\im \vet{c_{\kappa,b}^{1/2} \cos(x)}{2(\kappa+2b) c_{\kappa,b}^{-1/2} \sin(x)}$, which can be decomposed into the basis of $\mathcal{V}_{0,0}$ and $\mathcal{U}$ as follows:
\begin{align*}
    \dot{\mathscr{L}}_{0,0} f^+_1 = -\im \frac{1+3\kappa+5b}{2 c_{\kappa,b}} f^+_1 - \im \frac{1-\kappa -3b}{2 c_{\kappa,b}} f^+_{-1}.
\end{align*}
Substituting this into the contour integral and applying the residue theorem at $\lambda=0$, the $f^+_1$ component yields $0$ residue, and we conclude
\begin{align*}
    \dot{P}_{0,0}f^+_1
    = -\frac{1}{2\pi \im}\oint_\Gamma \frac{1}{\lambda}\left(\mathscr{L}_{0,0}-\lambda\right)^{-1}
    \left( -\im \frac{1-\kappa-3b}{2 c_{\kappa,b}} f^+_{-1} \right) d\lambda
    = \im \frac{1-\kappa-3b}{4c_{\kappa,b}^2} f^-_{-1}
    = \im\frac{\gamma_{\kappa,b}}{4} f^-_{-1}.
\end{align*}
Similarly, following the exact same steps, we have $\dot{P}_{0,0}f^-_1=\im\frac{\gamma_{\kappa,b}}{4} f^+_{-1}$.

Finally, we calculate the Floquet derivatives for the zero modes $f^\pm_0$. In the infinite depth case, we established that $\dot{\mathscr{L}}_{0,0} f^-_0 = f^+_0$ and $\dot{\mathscr{L}}_{0,0} f^+_0 = \vet{0}{0}$. Therefore, we compute
\begin{equation*}
    \begin{aligned}
        \dot{P}_{0,0}f^+_0=&\frac{1}{2\pi \im}\oint_\Gamma \left(\mathscr{L}_{0,0}-\lambda\right)^{-1}\dot{\mathscr{L}}_{0,0}\left(\frac{1}{\lambda^2}f^-_0-\frac{1}{\lambda}f^+_0\right) d\lambda \\
        =&\frac{1}{2\pi \im}\oint_\Gamma \left(\mathscr{L}_{0,0}-\lambda\right)^{-1} \left( \frac{1}{\lambda^2} f^+_0 \right) d\lambda
        = \frac{1}{2\pi \im}\oint_\Gamma \left( -\frac{1}{\lambda^3}f^+_0 + \frac{1}{\lambda^4}f^-_0 \right) d\lambda=0, \\
        \dot{P}_{0,0}f^-_0=&\frac{1}{2\pi \im}\oint_\Gamma \left(\mathscr{L}_{0,0}-\lambda\right)^{-1}\dot{\mathscr{L}}_{0,0}\left(-\frac{1}{\lambda}f^-_0\right) d\lambda \\
        =&\frac{1}{2\pi \im}\oint_\Gamma \left(\mathscr{L}_{0,0}-\lambda\right)^{-1}\left(-\frac{1}{\lambda}f^+_0\right) d\lambda
        = \frac{1}{2\pi \im}\oint_\Gamma \left( \frac{1}{\lambda^2}f^+_0 - \frac{1}{\lambda^3}f^-_0 \right) d\lambda = 0.
    \end{aligned}
\end{equation*}
Notice that despite $\dot{\mathscr{L}}_{0,0} f^-_0 \neq \vet{0}{0}$ for infinite depth, the residues still perfectly vanish. In conclusion, all the formulas are proven.
\end{proof}
So far we have obtained the linear terms of the expansions \eqref{expansion_f1+}, \eqref{expansion_f1-},
\eqref{expansion_f0+} and \eqref{expansion_f0-}. We now provide further information about the expansion of the basis at $\mu = 0$. 
\begin{lemma} 
     The basis $\{ f_k^{\sigma}(0,\e), \quad k = 0,1, \quad \sigma = \pm \}$ is real. For any $\epsilon$ it results $f_0^-(0,\epsilon) \equiv f_0^-$. The property \eqref{eq:f_0_pm_0_epsilon} holds.
\end{lemma}

\begin{proof}
The reality of the basis $f_k^\sigma(0,\epsilon)$ follows from
Lemma~\ref{properties of U and P}-(iii). Moreover, a direct inspection
of \eqref{mathscr L mu e}--\eqref{mathcal B mu e} gives
\[
\mathscr L_{0,\epsilon}f_0^-=0
\]
for every sufficiently small $|\epsilon|$. Therefore,
\[
(\mathscr L_{0,\epsilon}-\lambda)^{-1}f_0^-
=
-\frac{1}{\lambda}f_0^-.
\]
and then, using also the residue theorem,
\[
P_{0,\epsilon} f_0^- = -\frac{1}{2\pi \im} \oint_\Gamma (\mathscr{L}_{0,\epsilon} - \lambda)^{-1} f_0^- \, d\lambda = f_0^-.
\]
In particular $P_{0,\epsilon} f_0^- = P_{0,0} f_0^-$, for any $\epsilon$, and we get, by \eqref{U transformation operators},
\[
f_0^-(0,\epsilon) = U_{0,\epsilon} f_0^- = f_0^-,
\]
for any $\epsilon$.

Let us prove property \eqref{eq:f_0_pm_0_epsilon}. In view of \eqref{Parity structure} and since the basis is real, we know that
\[
f_k^+(0,\epsilon) = \begin{bmatrix}\mathit{even}(x)\\ \mathit{odd}(x)\end{bmatrix}, 
\qquad 
f_k^-(0,\epsilon) = \begin{bmatrix}\mathit{odd}(x)\\ \mathit{even}(x)\end{bmatrix},
\]
for any $k = 0, 1$. By Lemma \ref{F is symplectic and reversible} the basis $\{f_k^\sigma(0,\epsilon)\}$ is symplectic (cf. \eqref{basis is symplectic}) and, since
\[
\mathcal{J} f_0^-(0,\epsilon) = \mathcal{J} f_0^- = \begin{bmatrix}1\\0\end{bmatrix},
\]
for any $\epsilon$, we get
\[
0 = (\mathcal{J} f_0^-(0,\epsilon),\, f_1^+(0,\epsilon)) 
= \left(\begin{bmatrix}1\\0\end{bmatrix}, f_1^+(0,\epsilon)\right), 
\]
and
\[
1 = (\mathcal{J} f_0^-(0,\epsilon),\, f_0^+(0,\epsilon)) 
= \left(\begin{bmatrix}1\\0\end{bmatrix}, f_0^+(0,\epsilon)\right).
\]
Thus the first component of both $f_1^+(0,\epsilon)$ and $f_0^+(0,\epsilon) - \begin{bmatrix}1\\0\end{bmatrix}$ has zero average, proving \eqref{eq:f_0_pm_0_epsilon}.
\end{proof}

\begin{lemma} \label{lem:B5}
    For any small $\mu$, we have $f_0^+(\mu,0) \equiv f_0^+$ and $f_0^-(\mu,0) \equiv f_0^-$. Moreover, the vectors $f_1^\pm(\mu,0)$ have both components with zero space average.
\end{lemma}

\begin{proof}
    We consider the unperturbed infinite-depth operator at $\epsilon=0$. In the moving frame, after subtracting the trivial drift $i c_{\kappa,b} \mu$, it takes the form
\[
\mathscr{L}_{\mu,0}
=
\begin{bmatrix}
c_{\kappa,b}\partial_x & |D+\mu| \\
-1+\kappa(\partial_x+\im\mu)^2-b(\partial_x+\im\mu)^4
& c_{\kappa,b}\partial_x
\end{bmatrix}.
\]
    We study its action on the zero-mode subspace
    \[
    \mathcal{Z}:=\mathrm{span}\{f_0^+,f_0^-\}.
    \]
    A direct computation gives
    \[
    \mathscr L_{\mu,0} f_0^+
=
-(1+\kappa\mu^2+b\mu^4)f_0^-,
\qquad
\mathscr L_{\mu,0} f_0^-
=
\mu f_0^+,
    \]

    Hence $\mathscr{Z}_{\mu,0}$ leaves the subspace $\mathcal{Z}$ invariant. The restriction $\mathscr{Z}_{\mu,0}|_{\mathcal{Z}}$ has the eigenvalues
    \[
    \lambda_\pm(\mu)=\pm \im  \sqrt{\mu(1+\kappa\mu^2+b\mu^4)}.
    \]
    For $\mu$ small, these eigenvalues are close to zero and lie strictly inside the contour $\Gamma$. Therefore
    \[
    \mathcal{Z}\subseteq \mathcal{V}_{\mu,0}=\mathrm{Rg}(P_{\mu,0}),
    \]
    and thus
    \[
    P_{\mu,0}f_0^\pm=f_0^\pm.
    \]
    By \eqref{U transformation operators}, this yields
    \[
    f_0^\pm(\mu,0)=U_{\mu,0}f_0^\pm=f_0^\pm.
    \]

    It remains to prove that both components of $f_1^\pm(\mu,0)$ have zero space average. Since the basis $\{f_k^\sigma(\mu,0)\}$ is symplectic, and
    \[
    \mathcal{J}f_0^+(\mu,0)=\mathcal{J}f_0^+=\begin{bmatrix}0\\-1\end{bmatrix},
    \qquad
    \mathcal{J}f_0^-(\mu,0)=\mathcal{J}f_0^-=\begin{bmatrix}1\\0\end{bmatrix},
    \]
    we infer
    \[
    0=(\mathcal{J}f_0^\pm(\mu,0),\,f_1^\sigma(\mu,0)),
    \qquad \sigma=\pm.
    \]
    Therefore the first and second components of $f_1^\pm(\mu,0)$ have zero spatial average.
\end{proof}

\begin{lemma} \label{lem:B6}
    The mixed cross-derivatives
    \[
    (\partial_\mu\partial_\epsilon f_k^\sigma)(0,0)
    =
    \left(\dot{P}_{0,0}' - \frac{1}{2}P_{0,0}\dot{P}_{0,0}'\right)f_k^\sigma
    \]
    satisfy
    \begin{align*}
        (\partial_\mu\partial_\epsilon f_1^+)(0,0) &= \im \vet{\mathit{odd}(x)}{\mathit{even}(x)}, \qquad (\partial_\mu\partial_\epsilon f_1^-)(0,0) = \im \vet{\mathit{even}(x)}{\mathit{odd}(x)}, \\(\partial_\mu\partial_\epsilon f_0^+)(0,0) &= \frac{1}{4c_{\kappa,b}^{3/2}}\,f^+_{-1} + \im \vet{\mathit{odd}(x)}{\mathit{even}_0(x)}, \qquad(\partial_\mu\partial_\epsilon f_0^-)(0,0) 
= \vet{\dfrac{1}{2c_{\kappa,b}}\sin(x)}{\dfrac{1}{2}\cos (x)}
+ \im \vet{\mathit{even}_0(x)}{\mathit{odd}(x)}.
    \end{align*}
\end{lemma}
\begin{proof}
    We write
    \[
    (\partial_\mu\partial_\epsilon f_k^\sigma)(0,0)
    =
    \left(\dot{P}_{0,0}' - \frac{1}{2}P_{0,0}\dot{P}_{0,0}'\right)f_k^\sigma,
    \]
    where, according to the Kato expansion,
    \[
    \dot P'_{0,0}=\eqref{eq:dotPprime-term1}+\eqref{eq:dotPprime-term2}+\eqref{eq:dotPprime-term3},
    \]
    with
    \[
    \eqref{eq:dotPprime-term1}=
    -\frac{1}{2\pi\im}\oint_\Gamma
    R(\lambda)\dot{\mathscr L}_{0,0}R(\lambda)\mathscr L'_{0,0}R(\lambda)\,d\lambda,
    \]
    \[
    \eqref{eq:dotPprime-term2}=
    -\frac{1}{2\pi\im}\oint_\Gamma
    R(\lambda)\mathscr L'_{0,0}R(\lambda)\dot{\mathscr L}_{0,0}R(\lambda)\,d\lambda,
    \]
    \[
    \eqref{eq:dotPprime-term3}=
    \frac{1}{2\pi\im}\oint_\Gamma
    R(\lambda)\dot{\mathscr L}'_{0,0}R(\lambda)\,d\lambda,
    \qquad
    R(\lambda):=(\mathscr L_{0,0}-\lambda)^{-1}.
    \]

Differentiating the operator $\mathscr L_{\mu,\epsilon}$ with respect to
$\epsilon$, $\mu$, and then both variables at $(\mu,\epsilon)=(0,0)$, we obtain
\[
\mathscr L'_{0,0}
=
\begin{bmatrix}
\partial_x\circ p_1(x) & 0 \\
-a_1(x)+\kappa\tau_1-b\beta_1 & p_1(x)\partial_x
\end{bmatrix},
\]
\[
\dot{\mathscr L}_{0,0}
=
\begin{bmatrix}
0 & \sgn(D)+\Pi_0 \\
2\im\kappa\partial_x-4\im b\,\partial_x^3 & 0
\end{bmatrix},
\]
and
\[
\dot{\mathscr L}'_{0,0}
=
\begin{bmatrix}
\im p_1(x) & 0 \\
\im\kappa\big(\tau_{1,1}(x)+2\tau_{1,2}(x)\partial_x\big)
-\im b\sum_{j=1}^4 j\,b_{1,j}(x)\partial_x^{j-1}
& \im p_1(x)
\end{bmatrix}.
\]

We next show that the projection $\Pi_0$ in
$\dot{\mathscr L}_{0,0}$ gives no contribution to the mixed derivatives of
$f_1^\pm$. The only terms in $\dot P'_{0,0}$ where $\Pi_0$ can appear are
the two terms containing $\dot{\mathscr L}_{0,0}$, namely $\eqref{eq:dotPprime-term1}$ and
$\eqref{eq:dotPprime-term2}$.

In the term $\eqref{eq:dotPprime-term2}$, one first has
\[
R(\lambda)f_1^\pm=-\frac1\lambda f_1^\pm .
\]
The second component of $f_1^+$ is proportional to $\sin x$, while the
second component of $f_1^-$ is proportional to $\cos x$. Both functions have
zero spatial average. Hence the $\Pi_0$-part of $\dot{\mathscr L}_{0,0}$
vanishes in $\eqref{eq:dotPprime-term2}$.

In the term $\eqref{eq:dotPprime-term1}$, using the formulas for
$\mathscr L'_{0,0}f_1^\pm$, we have
\[
\mathscr L'_{0,0}f_1^+
=
\begin{bmatrix}
2c_{\kappa,b}^{1/2}\sin(2x)\\
3(5b+\kappa)c_{\kappa,b}^{-1/2}\cos(2x)
\end{bmatrix},
\qquad
\mathscr L'_{0,0}f_1^-
=
\begin{bmatrix}
2c_{\kappa,b}^{1/2}\cos(2x)\\
-3(5b+\kappa)c_{\kappa,b}^{-1/2}\sin(2x)
\end{bmatrix}.
\]
Thus $\mathscr L'_{0,0}f_1^\pm$ contains only second Fourier modes.
Since the constant-coefficient resolvent
$R(\lambda)=(\mathscr L_{0,0}-\lambda)^{-1}$ preserves each nonzero
Fourier subspace, the vector
\[
R(\lambda)\mathscr L'_{0,0}R(\lambda)f_1^\pm
\]
has no zero Fourier mode in its second component. Therefore the $\Pi_0$-part
of $\dot{\mathscr L}_{0,0}$ also vanishes in $\eqref{eq:dotPprime-term1}$. Finally, the term
$\eqref{eq:dotPprime-term3}$ contains $\dot{\mathscr L}'_{0,0}$, which has no $\Pi_0$-term.

It remains to keep track of parity. The coefficients
$p_1,a_1,\tau_{1,j}$ and $b_{1,j}$ are first harmonics with the parities
listed in Lemma~\ref{lem:pa.exp}. Moreover, $\partial_x$ changes parity, and
$\sgn(D)$ maps real even zero-average functions to purely imaginary odd
functions and real odd functions to purely imaginary even zero-average
functions. Consequently the three contour-integral contributions
$\eqref{eq:dotPprime-term1}$--$\eqref{eq:dotPprime-term3}$, and hence also
\[
\left(\dot P'_{0,0}
-\frac12 P_{0,0}\dot P'_{0,0}\right)f_1^\pm,
\]
have the parity structure
\[
(\partial_\mu\partial_\epsilon f_1^+)(0,0)
=
\im
\begin{bmatrix}
\mathit{odd}(x)\\
\mathit{even}(x)
\end{bmatrix},
\qquad
(\partial_\mu\partial_\epsilon f_1^-)(0,0)
=
\im
\begin{bmatrix}
\mathit{even}(x)\\
\mathit{odd}(x)
\end{bmatrix}.
\]

    As in the infinite-depth case without the $\kappa\sigma(\eta)$ term, For $f_1^\pm$ the intermediate vectors generated by the resolvent and by the operators above have zero spatial average in the component on which $\Pi_0$ acts. Therefore $\Pi_0$ gives no contribution for $f_1^\pm$.

For $f_0^-$, however, the projector $\Pi_0$ does contribute, since
\[
\dot{\mathscr L}_{0,0}f_0^-=f_0^+.
\]
This produces a real resonant contribution. 

We compute this contribution from the term $\eqref{eq:dotPprime-term2}$:
\[
T_{2,-}:=
-\frac{1}{2\pi\im}\oint_\Gamma
R(\lambda)\mathscr L'_{0,0}R(\lambda)\dot{\mathscr L}_{0,0}R(\lambda)f_0^-\,d\lambda .
\]
Using
\[
R(\lambda)f_0^-=-\frac1\lambda f_0^-,
\qquad
\dot{\mathscr L}_{0,0}f_0^-=f_0^+,
\]
we obtain
\[
R(\lambda)\dot{\mathscr L}_{0,0}R(\lambda)f_0^-
=
\frac1{\lambda^2}f_0^+
-
\frac1{\lambda^3}f_0^- .
\]
Since
\[
\mathscr L'_{0,0}f_0^+=2c_{\kappa,b}^{3/2}f_{-1}^-,
\qquad
\mathscr L'_{0,0}f_0^-=0,
\]
we get
\[
T_{2,-}
=
-\frac{2c_{\kappa,b}^{3/2}}{2\pi\im}
\oint_\Gamma
\frac1{\lambda^2}R(\lambda)f_{-1}^-\,d\lambda .
\]
Using
\[
R(\lambda)f_{-1}^-=
\frac{-2c_{\kappa,b} f_{-1}^+-\lambda f_{-1}^-}{\lambda^2+4c_{\kappa,b}^2},
\]
only the $f_{-1}^-$ component has a residue at $\lambda=0$, and therefore
\[
T_{2,-}^{\mathrm{real}}
=
\frac{1}{2c_{\kappa,b}^{1/2}}f_{-1}^-.
\]
Since
\[
f_{-1}^-=
\vet{c_{\kappa,b}^{-1/2}\sin (x)}{c_{\kappa,b}^{1/2}\cos (x)},
\]
this gives
\[
T_{2,-}^{\mathrm{real}}
=
\vet{\dfrac{\sin (x)}{2c_{\kappa,b}}}{\dfrac{\cos (x)}{2}}.
\]
Therefore
\[
(\partial_\mu\partial_\epsilon f_0^-)(0,0)
=
\vet{\dfrac{\sin (x)}{2c_{\kappa,b}}}{\dfrac{\cos (x)}{2}}
+
\im\vet{\mathit{even}_0(x)}{\mathit{odd}(x)}.
\]
The resulting vectors remain purely imaginary with the parity structure
    \[
    (\partial_\mu\partial_\epsilon f_1^+)(0,0)=\im\vet{\mathit{odd}(x)}{\mathit{even}(x)},
    \qquad
    (\partial_\mu\partial_\epsilon f_1^-)(0,0)=\im\vet{\mathit{even}(x)}{\mathit{odd}(x)},
    \]

    We now compute the real resonant contribution to $(\partial_\mu\partial_\epsilon f_0^+)(0,0)$. As in the case without the $\kappa\sigma(\eta)$ term, the only real contribution comes from the term $\eqref{eq:dotPprime-term2}$:
    \[
    T_2:=
    -\frac{1}{2\pi\im}\oint_\Gamma
    R(\lambda)\mathscr L'_{0,0}R(\lambda)\dot{\mathscr L}_{0,0}R(\lambda)f_0^+\,d\lambda.
    \]

    We first use the Jordan structure at the origin:
    \[
    \mathscr L_{0,0}f_0^+=-f_0^-,
    \qquad
    \mathscr L_{0,0}f_0^-=0.
    \]
    Therefore
    \[
    R(\lambda)f_0^+
    =
    -\frac1\lambda f_0^+ + \frac1{\lambda^2}f_0^-.
    \]
    Next, using
    \[
    \dot{\mathscr L}_{0,0}f_0^+=0,
    \qquad
    \dot{\mathscr L}_{0,0}f_0^-=f_0^+,
    \]
    we obtain
    \[
    \dot{\mathscr L}_{0,0}R(\lambda)f_0^+
    =
    \frac1{\lambda^2}f_0^+.
    \]
    Applying the resolvent once more yields
    \[
    R(\lambda)\dot{\mathscr L}_{0,0}R(\lambda)f_0^+
    =
    \frac1{\lambda^2}R(\lambda)f_0^+
    =
    -\frac1{\lambda^3}f_0^+ + \frac1{\lambda^4}f_0^-.
    \]

    Since
    \[
    \mathscr L'_{0,0}f_0^+=2c_{\kappa,b}^{3/2}f_{-1}^-,
    \qquad
    \mathscr L'_{0,0}f_0^-=0,
    \]
    it follows that
    \[
    \mathscr L'_{0,0}R(\lambda)\dot{\mathscr L}_{0,0}R(\lambda)f_0^+
    =
    -\frac{2c_{\kappa,b}^{3/2}}{\lambda^3}f_{-1}^-.
    \]
    Hence
    \[
    T_2
    =
    \frac{2c_{\kappa,b}^{3/2}}{2\pi\im}\oint_\Gamma
    \frac1{\lambda^3}R(\lambda)f_{-1}^-\,d\lambda.
    \]

    We now use the resolvent formula on the resonant subspace
    \[
    \mathcal U=\mathrm{span}\{f_{-1}^+,f_{-1}^-\},
    \]
    namely
    \[
    R(\lambda)f_{-1}^-=
    \frac{-2c_{\kappa,b} f_{-1}^+-\lambda f_{-1}^-}{\lambda^2+4c_{\kappa,b}^2}.
    \]
    Therefore
    \[
    T_2
    =
    \frac{2c_{\kappa,b}^{3/2}}{2\pi\im}\oint_\Gamma
    \frac{-2c_{\kappa,b} f_{-1}^+-\lambda f_{-1}^-}{\lambda^3(\lambda^2+4c_{\kappa,b}^2)}\,d\lambda.
    \]
    The $f_{-1}^-$-component has no residue at $\lambda=0$, so only the $f_{-1}^+$-component contributes:
    \[
    T_2^{\mathrm{real}}
    =
    \frac{1}{2\pi\im}\oint_\Gamma
    \frac{-4c_{\kappa,b}^{5/2}}{\lambda^3(\lambda^2+4c_{\kappa,b}^2)}\,d\lambda\, f_{-1}^+.
    \]

    Expanding near $\lambda=0$,
    \[
    \frac{1}{\lambda^2+4c_{\kappa,b}^2}
    =
    \frac{1}{4c_{\kappa,b}^2}
    -\frac{\lambda^2}{16c_{\kappa,b}^4}
    +O(\lambda^4),
    \]
    hence
    \[
    \frac{-4c_{\kappa,b}^{5/2}}{\lambda^3(\lambda^2+4c_{\kappa,b}^2)}
    =
    -\frac{c_{\kappa,b}^{1/2}}{\lambda^3}
    +\frac{1}{4c_{\kappa,b}^{3/2}}\frac1\lambda
    +O(\lambda).
    \]
    The residue at $\lambda=0$ is thus
    \[
    \Res_{\lambda=0}
    \frac{-4c_{\kappa,b}^{5/2}}{\lambda^3(\lambda^2+4c_{\kappa,b}^2)}
    =
    \frac{1}{4c_{\kappa,b}^{3/2}}.
    \]
    By the residue theorem we conclude that
    \[
    T_2^{\mathrm{real}}
    =
    \frac{1}{4c_{\kappa,b}^{3/2}}\,f_{-1}^+.
    \]

    Since $f_{-1}^+\in\mathcal U\subset \ker(P_{0,0})$, the projection term
    \[
    -\frac12 P_{0,0}\dot P'_{0,0}f_0^+
    \]
    does not affect this resonant component. Therefore
    \[
    (\partial_\mu\partial_\epsilon f_0^+)(0,0)
    =
    \frac{1}{4c_{\kappa,b}^{3/2}}\,f_{-1}^+
    +\im\vet{\mathit{odd}(x)}{\mathit{even}_0(x)}.
    \]

    This proves the statement.
\end{proof}

This completes the proof of Lemma \ref{lem:4.2}.

\footnotesize

\vspace{1em}
\noindent{\bf Data availability statement:}
Data will be made available on reasonable request\\

\noindent{\bf Conflict of interest:}
We do not have any conflict of interest.\\

\noindent{\bf Acknowledgments:}
T.-Y. Hsiao is supported by the European Union  ERC CONSOLIDATOR GRANT 2023 GUnDHam, Project Number: 101124921. He is deeply grateful to Vera Hur and Zhao Yang for their guidance and support during his doctoral studies, and to Alberto Maspero for his ongoing mentorship. Y. Zhang is supported by the European Union ERC STARTING GRANT 2020 GeoSub, Project Number: 945655.  Views and opinions expressed are however those of the authors only and do not necessarily reflect those of the European Union or the European Research Council. Neither the European Union nor the granting authority can be held responsible for them.

\end{document}